\documentclass[a4paper, 10pt]{article}

\usepackage[utf8]{inputenc}
\usepackage{amsmath,amsfonts,amsthm,amssymb,mathrsfs,enumerate,enumitem,graphicx,float,hyperref,verbatim,comment}
\hypersetup{colorlinks=true,linkcolor=blue}   

\usepackage[letterpaper,top=2.5cm, bottom=2.5cm, left=2.5cm, right=2.5cm]{geometry}

\usepackage{pgf,tikz}
\usetikzlibrary{shapes,positioning,intersections,quotes,patterns}
\usetikzlibrary[topaths]
\usepackage{circuitikz}
\usepackage{color}
\usepackage{pgfplots}
\pgfplotsset{compat=1.7}
\usepackage{caption}
\usepackage{subcaption}

\def\R{{\mathbb R}} 
\def\N{{\mathbb N}}

\def\C{{\mathbb C}}

\def\S{{\mathbb{S}}}
\def\D{{\mathscr{D}}}
\def\A{{\mathcal{A}}}

\def\ds{\displaystyle}
\newcommand{\norm}[1]{\left\Vert#1\right\Vert}
\renewcommand{\leq}{\leqslant}                 
\renewcommand{\geq}{\geqslant}
\renewcommand{\tilde}{\widetilde}
\renewcommand{\hat}{\widehat}
\renewcommand{\bar}{\overline}
\renewcommand{\div}{\hbox{div\,}}

\renewcommand{\i}{\mathrm{\bf i}}                 
\renewcommand{\O}{\mathcal{O}}

\newtheorem{theorem}{Theorem}[section]
\newtheorem{remark}[theorem]{Remark}
\newtheorem{lemma}[theorem]{Lemma}

\newtheorem{proposition}[theorem]{Proposition}

\numberwithin{equation}{section}

\def\keywords{
	\vspace{1ex}
	\noindent
	\if@twocolumn
	\small{\bf  Keywords}\/---$\!$    \else
	\begin{center}\small\ {\bf Keywords}\end{center}\quotation\small
	\fi}
\def\endkeywords{\vspace{0.6em}\par\if@twocolumn\else\endquotation\fi
	\normalsize\rm}

\title{Quantitative unique continuation for non-regular perturbations of the Laplacian\footnote{P. C. belongs to Ikerbasque and is funded by MICIU/AEI/10.13039/501100011033 and ERDF, EU through the grant PID2021-122156NB-I00. Additional funds come from Excellence Acreditation "Severo Ochoa", and BERC Programme 2022-2025 Eusko Jaurlaritza. S. E. and L. T. are partially supported by the PHC Utique 46359ZJ, code CMCU: 21G1502 and the Project TRECOS ANR-20-CE40-0009 funded by the ANR. L. T. has also benefited from a financial support from the University of Bordeaux as an ``initiative d’excellence'', within the setting of the ``plan France 2030''.}}

\author{
	Pedro Caro\footnote{Ikerbasque and Basque Center for Applied Mathematics, Bilbao, Spain,  \texttt{pcaro@bcamath.org}}
	\and
	Sylvain Ervedoza\footnote{Institut de Mathématiques de Bordeaux, UMR 5251, Université de Bordeaux, CNRS, Bordeaux INP, F-33400 Talence, France,  
		\texttt{sylvain.ervedoza@math.u-bordeaux.fr}}, 
	\and
	Lotfi Thabouti\footnote{Institut de Mathématiques de Bordeaux, UMR 5251, Université de Bordeaux, CNRS, Bordeaux INP, F-33400 Talence, France, and Département de mathématiques, Faculté des sciences de Tunis, Université de Tunis El Manar, 2092 El Manar \& Ecole Nationale d'Ingénieurs de Tunis, ENIT-LAMSIN, B.P. 37, 1002 Tunis, Tunisia, \texttt{lotfi.thabouti@math.u-bordeaux.fr,  lotfi.thabouti@etudiant-fst.utm.tn}}
} 
\begin{document}
	
	\maketitle
	
	\begin{abstract}
		In this work, we investigate the quantitative estimates of the unique continuation property for solutions of an elliptic equation $\Delta u = V u + W_1 \cdot \nabla u + \hbox{div\,} (W_2 u)$ in an open, connected subset of $\mathbb{R}^d$, where $d \geq 3$. Here, $V \in L^{q_0}$, $W_1 \in L^{q_1}$, and $W_2 \in L^{q_2}$ with $q_0 > d/2$, $q_1 > d$, and $q_2 > d$. Our aim is to provide an explicit quantification of the unique continuation property with respect to the norms of the potentials. To achieve this, we revisit the Carleman estimates established in  \cite{Dehman-Ervedoza-Thabouti-2023} and prove a refined version of them, and we combine them with an argument due to T. Wolff introduced in \cite{wolff1992property} for the proof of unique continuation for solutions of equations of the form $\Delta u = V u + W_1 \cdot \nabla u$.
	\end{abstract}


	\setcounter{tocdepth}{1}
	\tableofcontents
	%
	%
	\section{Introduction}
	
	{\bf Main result.} Our main goal is to prove the following quantitative unique continuation result. 
	
	\begin{theorem} \label{Thm-QUCP} 
		Let $d\geqslant 3$, $\Omega \subset \R^{d}$ be a bounded domain, and $\omega$ and $\O$ be non-empty open subsets of $\Omega$ with $\omega \subset \overline \omega \subset \O \subset \overline\O \subset \Omega$. Further assume the following geometric condition: 
		\begin{enumerate}[topsep=0pt,itemsep=0pt,parsep=0pt]
			\item[{\bf (GC)}]\label{GC} 
			For all $y \in \overline\O$, there exist $x_0 \in \omega$, $r_{y}>0$ and a smooth path $\gamma_{y}$ of finite length such that $\gamma_{ y}(0) = x_0$, $\gamma_{y}(1) = y$, and $\cup_{s \in [0,1]} B_{\gamma_{y}(s)}(r_{y}) \subset \Omega$, where $B_{\gamma_y(s)}(r_y)$ is the ball centered in $\gamma_y(s)$ and of radius $r_y$.
		\end{enumerate}
		Then there exist constants $C=C(\omega, \O, \Omega)>0$ and $\alpha \in (0,1)$ depending only on $\omega$, $\O$ and $\Omega$ so that for any  solution $u\in H^{1}(\Omega)$ of 
		\begin{equation}
		\label{Elliptic-UC}
		\Delta u = V u + W_1 \cdot \nabla u + \operatorname{div} ( W_2 u) \quad \text{ in }\D'( \Omega), 
		\end{equation}
		where
		\begin{align} \label{potential-assumptions}
		V \in L^{q_0}(\Omega)  \hbox{ with }q_0 \in \left(\frac{d}{2}, \infty \right], 
		\qquad \text{ and } \qquad
		W_j \in L^{q_j}(\Omega; \C^d) 
		\text{ with } q_j \in \left(d, \infty\right],
		\text{ for } j \in \{1, 2\},
		\end{align}
		we have 
		\begin{equation} \label{Quantitative-0}
		\norm{u }_{H^1(\O )} \leqslant 
		C 
		e^{
			C \left(\norm{V}_{L^{q_0}(\Omega)}^{\gamma(q_0)}+ \norm{W_1}_{L^{q_1}(\Omega)}^{\delta(q_1)}
			+ \norm{W_2}_{L^{q_2}(\Omega)}^{\delta(q_2)}
			\right)
		}\norm{u}_{H^1(\omega)}^\alpha
		\norm{u}_{H^1(\Omega)}^{1-\alpha},
		\end{equation}
		with  
		\begin{equation} \label{Conditions-gamma(q)-delta(q)}
		\gamma (q)= \left\{
		\begin{array}{ll}
		\ds  \dfrac{1}{\ds \frac{3}{2} \left(1- \frac{d}{2q}\right) + \frac{1}{2q} } & \text{ if } q \geq d,
		\medskip\\
		\ds \dfrac{1}{ \ds \left(\frac{3}{4} + \frac{1}{2d}\right)  \left(2- \frac{d}{q}\right) } & \text{ if } q \in \left(\ds \frac{d}{2},d\right],
		\end{array}
		\right.
		\qquad \text{ and } \qquad 
		\delta(q) = \frac{2}{\ds 1 - \frac{d}{q}}, \quad  \text{ if } q > d. 
		\end{equation}
	\end{theorem}

	Before going further, let us note that the geometric condition {\bf (GC)} assumed in Theorem \ref{Thm-QUCP} is a very mild technical condition. This condition can be violated in some fractal type sets, but it is certainly satisfied for most geometrical settings. In particular, this geometric condition {\bf (GC)} is satisfied when $\omega$, $\mathcal{O}$ and $\Omega$ are concentric balls, case in which \eqref{Quantitative-0} reduces to the usual $3$-balls type estimate, with an explicit quantification with respect to the lower order terms. 
	
	Of course, a direct application of Theorem \ref{Thm-QUCP} yields the following unique continuation property for the Laplace operator: If $u \in H^1(\Omega)$ satisfies \eqref{Elliptic-UC} with potentials $V$, $W_1$ and $W_2$ as in \eqref{potential-assumptions} and $u = 0$ in $\omega$, then $u = 0$ in any set $\O$ satisfying $\O \Subset \Omega$ and {\bf (GC)}. Theorem \ref{Thm-QUCP} is a quantification of this property as it states that, if $u \in H^1(\Omega)$ is small in $\omega$ and satisfies \eqref{Elliptic-UC} with potentials $V$, $W_1$ and $W_2$ as in \eqref{potential-assumptions}, then $u$ is small in $\O$, with a precise quantification in terms of the norms of the potentials.

	It is well-known that unique continuation holds for general $V \in L^{q_0}(\Omega)$, $W_1 \in L^{q_1}(\Omega; \C^d)$, and $W_2 \in L^{q_2}(\Omega; \C^d)$, where $q_0 \geq d/2$, $q_1 \geq d$, and $q_2 \geq d$ (see \cite{Wolff-1993}). Moreover, strong unique continuation has even been proven in cases when $q_0 > d/2$, $q_1 > d$, and $q_2 > d$ (see \cite{koch2001carleman}). These classes of integrability for the potentials are sharp, as shown in \cite{Koch-Tataru-02}. However, establishing unique continuation results requires the use of a Carleman estimate and a delicate argument due to T. Wolff, see \cite{wolff1992property}, as also discussed in \cite{koch2001carleman}. This argument requires to choose a weight function in the Carleman estimate depending on the solution $u$ itself. The question of quantifying the unique continuation property with respect to the norms of the potentials is thus quite delicate, and this is the main novelty of our work. 
	
	We also mention the work \cite{malinnikova2012quantitative}, which quantifies unique continuation properties for the Laplace operator with lower-order terms in the sharp integrability class. However, it does not provide an explicit quantification with respect to the norms of the potentials. Again, since this work builds upon \cite{koch2001carleman}, as mentioned earlier, it remains unclear how the proof in \cite{malinnikova2012quantitative} can be made quantitative in terms of the norms of the potentials.

	When trying to quantify the unique continuation property with respect to the norms of the lower order terms, the known results rely only on the use of a Carleman estimate such as the one presented in \cite[Theorem 1.1]{Dehman-Ervedoza-Thabouti-2023}, which, as pointed out in \cite{barcelo1988weighted}, does not allow to go beyond $W_1 \in L^{\frac{3d-2}{2}}(\Omega; \C^d)$. This corresponds to what is done in \cite{davey2019quantitative,davey2020quantitative,Dehman-Ervedoza-Thabouti-2023} using $L^p$ Carleman estimate. The results in \cite{davey2019quantitative} describing the maximal order of vanishing of solutions of elliptic equations require $V$ and $W_1$ respectively in $L^{q_0}(\Omega)$ with $q_0 > d(3d-2)/(5d-2)$ and in $L^{q_1}(\Omega; \C^d)$ with $q_1 > (3d-2)/2$, and $W_2=0$. This was improved in \cite{Dehman-Ervedoza-Thabouti-2023} using new $L^p$ Carleman estimates allowing to handle $V$ and $W_1$ respectively in $L^{q_0}(\Omega)$ with $q_0 > d/2$ and in $L^{q_1}(\Omega; \C^d)$ with $q_1 > (3d-2)/2$, exhibiting the same dependence in $\|V \|_{L^{q_0}(\Omega)}$ as in \cite{davey2020quantitative}, which was limited to the case $W_1 = W_2 = 0$.
	
	We also point out that $L^2$ Carleman estimates does not allow to reach the sharp integrability class for potentials and lower order terms, see for instance \cite{Saut1982-sur-l'unitcite} where it is shown that one can obtain quantified unique continuation results for potentials $V \in L^{q_0}(\Omega)$ with $q_0 > 2d/3$. In fact, even if one uses $L^p$ Carleman estimates, one can obtain unique continuation results for potentials $V$ in the sharp class of integrability ($V \in L^{q_0}(\Omega)$ with $q_0 > d/2$), but under restrictive integrability conditions on $W_1$ and $W_2$, see for instance \cite{barcelo1988weighted} and \cite{Wolff-1993}.
	
	Theorem \ref{Thm-QUCP} is an improvement of \cite[Theorem 1.3]{Dehman-Ervedoza-Thabouti-2023}, since Theorem \ref{Thm-QUCP} allows non-trivial lower order terms $W_1\in L^{q_1}(\Omega; \C^d)$ and $W_2 \in L^{q_2}(\Omega; \C^d)$ with $q_1 > d$ and $q_2 >d$, while Theorem 1.3 in \cite{Dehman-Ervedoza-Thabouti-2023} is restricted to higher integrability class. Indeed, Theorem 1.3 in \cite{Dehman-Ervedoza-Thabouti-2023} states that if $W_1\in L^{q_1}(\Omega; \C^d)$ and $W_2 \in L^{q_2}(\Omega; \C^d)$, $q_1 >(3d-2)/2$, $q_2 > (3d-2)/2$ and $1/q_1 +1/q_2 < 4 (1 - 1/d )/ (3d-2)$, then any  solution $u\in H^{1}(\Omega)$ of \eqref{Elliptic-UC} satisfies  
	\begin{equation} \label{Quantitative-0-Previous}
	\norm{u }_{H^1(\O )} \leqslant 
	C 
	e^{
		C \left(\norm{V}_{L^{q_0}(\Omega)}^{\gamma(q_0)}
		+
		\| W_1\|_{L^{q_1}}^{\tilde \delta(q_1) } +  \| W_2\|_{L^{q_2}}^{\tilde \delta(q_2) } + \left(\| W_1\|_{L^{q_1}}\| W_1\|_{L^{q_2}}\right)^{\gamma(q_1,q_2) } 
		\right)
	}\norm{u}_{H^1(\omega)}^\alpha
	\norm{u}_{H^1(\Omega)}^{1-\alpha},
	\end{equation}
	%
	%
	with $$\tilde \delta(q ) = \frac{2}{ (1-\frac{(3d-2)}{2q} )} \quad \text{ and } \, \gamma(q_1, q_2) = \frac{1}{(1-\frac{1}{d})-(\frac{3}{4}-\frac{1}{2d})(\frac{d}{q_1}+\frac{d}{q_2})}.$$ 
	Accordingly, the dependence in terms of the norms of the potentials $W_1$ and $W_2$ is also weaker in \eqref{Quantitative-0} than in \eqref{Quantitative-0-Previous}, as $\delta$ given by \eqref{Conditions-gamma(q)-delta(q)} is smaller than $\tilde \delta$, even for $q > (3d-2)/2$. 
	
	In fact, this also suggests to analyze the optimality of the coefficients  $\gamma$ and $\delta$ in \eqref{Conditions-gamma(q)-delta(q)}, but this question is, to our knowledge, fully open, except when $W_1 = W_2 = 0$ and $V \in L^\infty(\Omega)$. Indeed, in this case, it is known that in dimension $d \geqslant 3$ the dependence of the constant in the quantification of unique continuation is of the form $C \exp( C \| V \|_{L^\infty}^{\frac{2}{3}})$, see \cite{Mes91,duyckaerts2008optimality}. This coincides with our estimate as $\gamma (\infty ) = 2/3$. 
	
	Our proof does not allow to derive an estimate up to the boundary of $\Omega$. This is due to a technical fact, coming from the use of Wolff's argument on one hand and of the use of the Carleman estimate in \cite{Dehman-Ervedoza-Thabouti-2023} on the other hand. Indeed, as we will see, roughly speaking, Wolff's argument requires the possibility to play with (the gradient of)  the weight function within the Carleman estimate. But the Carleman estimate in \cite{Dehman-Ervedoza-Thabouti-2023}, which we will use and revisit within this work, requires the boundary of the domain to be a level set of the weight function. These two conditions are thus not compatible and cause trouble when working in a neighborhood of the boundary.
	
	We also would like to emphasize that the Schrödinger operators in \eqref{Elliptic-UC} also include the consideration of potentials $V \in W^{-1,d+\epsilon}(\Omega)$, where $\epsilon>0$. Indeed, any such potential can be represented under the form $V \equiv V_0 + \div ( W)$, with $V_0 \in L^{d+\epsilon}(\Omega) \subset L^{\frac{d}{2}+\epsilon}(\Omega) $ and $W \in L^{d+\epsilon}(\Omega, \mathbb{C}^d)$ (see, for example, \cite[Theorem 3.9]{adams2003sobolev}). Consequently, one can then rewrite $Vu$ under the form $Vu = V_0 u + \div(W u) - W \cdot \nabla u$. 

	\medskip
	
	\noindent{\bf Outline.}
	Let us briefly comment on the structure of the article. In the next section, we first briefly recall an auxiliary result due to T. Wolff in \cite{wolff1992property} (cf. Lemma \ref{Lemma-Wolff}), followed by the presentation of some improved $L^p$ Carleman type estimates (cf. Theorem \ref{Carleman-General-1-Improved} in Section \ref{Auxiliary Results}) inspired by \cite{Dehman-Ervedoza-Thabouti-2023}. Following this, Section \ref{Sec-Proof-Thm-Carleman-Wolff-type-estimates} is devoted to the proof of these Carleman estimates. Then, in Section \ref{A specific geometric setting}, we establish a local quantification of the unique continuation within a specific geometric framework (cf. Lemma \ref{lemma-a-specific-geomertric-setting}). In Section \ref{Other geometries and proof of Theorem 1.1}, we subsequently employ this to establish a three balls estimate, which directly yields the quantitative unique continuation result stated in Theorem \ref{Thm-QUCP}, after a few classical manipulations.
	\medskip
	
	\noindent{\bf Notations.}  Let us finally introduce some of the notation that we will use throughout the article: 
	\begin{itemize}[topsep=0pt,itemsep=0pt,parsep=0pt]
		\item For every $x\in \R^{d}$, $x=(x_1,..,x_d)$, we  set $x=(x_1,x')$, where $x'=(x_2,..,x_d) \in \R^{d-1}$.  
		
		\item For $x \in \R^d$ and $r > 0$, $B_{x}(r)$ denotes the ball centered at $x \in \R^d$ and of radius $r > 0$.  
		
		\item  The notations $\nabla$ and $\Delta$ respectively stand for the gradient and the Laplacian with respect to $x=(x_1,..,x_d)$, and $\nabla'=(\partial_{2},..,\partial_{{d}})$ and $\Delta'= \sum_{j=2}^{d} \partial_{j}^2$ are, respectively, the vertical gradient and Laplacian operators. 
		\item  The Fourier transform is always taken to be the Fourier transform with respect to $x'=(x_2,..,x_d)$, and then its dual variable $\xi' \in \R^{d-1}$ is indexed by $\xi'=(\xi_2,..,\xi_d)$. Note that for a function $f$ defined on $\R^d$ (or a vertical strip $(X_0, X_1) \times \R^{d-1}$) such that $f(x_1,\cdot ) \in \mathscr{S}(\R^{d-1})$, $\hat f(x_1, \cdot)$ denotes the partial Fourier transform with respect to $x'$, that is: 
		\begin{equation}
		\label{Partial-Fourier-Transform}
		\hat f(x_1,\xi') = \frac{1}{(2\pi)^{\frac{d-1}{2}}} \int_{\R^{d-1}} e^{-\i  x' \cdot \xi'} f(x_1,x') \, dx', \qquad \qquad\xi' \in \R^{d-1}.
		\end{equation}
		%
		\item For a measurable subset $E$ of $\R^d$, we denote by $|E|$ its Lebesgue measure. 
	\end{itemize}

	%
	%
	\section{Auxiliary Results} \label{Auxiliary Results}
	
	In this section, we begin by recalling Wolff's lemma, and we present a refined version of the $L^p$ Carleman estimates obtained in \cite{Dehman-Ervedoza-Thabouti-2023}. These elements are the main points in the proof of Theorem \ref{Thm-QUCP}. 
	%
	%
	%
	\subsection{Wolff’s Lemma}
	
	We hereby present Wolff's argument, introduced in \cite{wolff1992property}:
	\begin{lemma}[{\cite[Lemma 1]{wolff1992property}}]
		\label{Lemma-Wolff} \label{exponential-decay}
		Suppose $\mu$ is a positive measure in $\R^d$ which has faster than exponential decay in the following sense
		\begin{equation}
		\lim_{T \to \infty} T^{-1}\log(\mu\{x \in \R^{d} \, , |x| \geq T\})=- \infty.  
		\end{equation}
		For $k \in \R^d$, define the measure $\mu_k$ by $d\mu_k(x) = e^{k\cdot x} d\mu(x)$. Suppose $\mathcal{C} \subset \mathbb{R}^d$ is a compact convex set. Then there is a family $ (k_j)_{j \in J}$ of elements of $\mathcal{C}$ and a family of two by two disjoint convex sets $(E_{k_j })_{j \in J}$ included in $\R^d$ so that the measures $d\mu_{k_j}$ are concentrated in $E_{k_j}$,
		\begin{equation} \label{concentration-property}
		\mu_{k_j}(\R^d \setminus (1+T) E_{k_j}) \leq \frac{1}{2} e^{- T /C_W} \| \mu_{k_j}\|, \quad \forall T \geq 0.  
		\end{equation} 
		and such that 
		\begin{equation} \label{summation-property}
		\sum_j  |E_{k_j}|^{-1} \geq C_W^{-1} |\mathcal{C}|,
		\end{equation}
		where $C_W$ is a positive constant depending only on $d$, and $(1 + T)E_{k_j}$ is the dilation of $E_{k_j}$ around its barycentre by a factor of $1 + T$.
	\end{lemma}  
	
	Lemma \ref{Lemma-Wolff} is the main argument in \cite{wolff1992property}, and the basis for the proof of unique continuation in \cite{wolff1992property} for potentials $V$ and $W_1$ in the sharp class of integrability (in \cite{wolff1992property}, the term $W_2$ is not considered).
	 
	%
	%
	\subsection{Carleman type estimates.} 
	
	We first present the Carleman estimates obtained in \cite[Theorem 1.1]{Dehman-Ervedoza-Thabouti-2023}:
	
	\begin{theorem}[{\cite[Theorem 1.1]{Dehman-Ervedoza-Thabouti-2023}}] \label{Thm-Carleman-estimates-DET}
		Let $d\geqslant 3$. Consider a bounded domain $\Omega \subset \R^{d}$ of class $C^3$, and non-empty open subsets $\omega_0$ and $\omega$ of $\Omega$ with $\omega_0 \Subset \omega \Subset \Omega$. 
		Let $\varphi \in C^3(\overline{\Omega})$ be such that 
		\begin{equation}
		\label{CarlemanWeight-Cond0} 
		\forall x \in \partial \Omega, \, \varphi(x) = 0 \text{ and } 
		\partial_n \varphi(x) < 0,
		\end{equation}
		and there exist $\alpha, \beta > 0$ for which 
		\begin{equation}
		\label{CarlemanWeight-Cond1} 
		\inf_{\overline{\Omega}\setminus \omega_0} \vert \nabla \varphi \vert > \alpha,
		\end{equation}
		and
		\begin{multline}
		\forall x \in \overline{\Omega} \setminus \omega_0, \, \forall \xi \in \mathbb{R}^d \text{ with } \vert \nabla \varphi(x)  \vert = \vert \xi \vert \text{ and } \nabla \varphi(x) \cdot \xi = 0, 
		\\	\label{CarlemanWeight-Cond2}
		({\rm Hess\,} \varphi(x)) \nabla \varphi(x) \cdot \nabla \varphi(x) +   ({\rm Hess\,} \varphi(x))  \xi  \cdot \xi  \geq \beta \vert \nabla \varphi(x) \vert^2, 
		\end{multline}
		where ${\rm Hess\,} \varphi$ denotes the Hessian matrix of $\varphi$. 
		
		Then there exist $C>0$ and $\tau_0 \geq 1$ (depending only on $\alpha$, $\beta$, $\| \varphi \|_{C^3(\overline\Omega)}$, and the geometric configuration of $\Omega$, $\omega$, and $\omega_0$) such that for all $u \in H^1_0(\Omega)$ satisfying 
		\begin{align}\label{Elliptic-Eq-general}
		- \Delta u = f_2 + f_{2*'} + \div F \quad  &\text{in } \D'(\Omega),
		\end{align}
		with $(f_2, f_{2*'}, F )$ satisfying
		\begin{equation}
		\label{Reg-Source-Terms-Improved-Gal-DET} 
		f_2 \in L^2(\Omega), 
		\ \ 
		f_{2*'} \in L^{ \frac{2d}{d+2}}(\Omega),
		\ \ 
		F \in L^2(\Omega; \C^{d}), 
		\end{equation}
		we have, for all $\tau \geq \tau_0$, 
		\begin{multline}
		\label{Carleman-General-1-DET}	
		\tau^{\frac{3}{2}} \| e^{\tau \varphi}  u \|_{L^2(\Omega)} +
		\tau^{\frac{1}{2}} \| e^{\tau \varphi}  \nabla u \|_{L^2(\Omega)} 
		\leq 
		C \left( 
		\| e^{\tau \varphi}  f_2 \|_{L^2(\Omega)}
		+ 
		\tau 	\| e^{\tau \varphi}  F \|_{L^2(\Omega)}
		\right.
		\\
		\left.
		+
		\tau^{\frac{3}{4}-\frac{1}{2d}}
		\| e^{\tau \varphi} f_{2*'} \|_{L^{\frac{2d}{d+2}} (\Omega)}
		+ 
		\tau^{\frac{3}{2}} \| e^{\tau \varphi} u\|_{L^2(\omega)}
		+ 
		\tau^{\frac{3}{4}}\| e^{\tau \varphi} u \|_{L^{\frac{2d}{d-2}} (\omega)}
		\right), 
		\end{multline}
		and
		\begin{multline}
		\label{Carleman-General-2-Improved-DET}		 
		\tau^{\frac{3}{4}+\frac{1}{2d}} \| e^{\tau \varphi} u \|_{L^{\frac{2d}{d-2}} (\Omega)}
		\leq 
		C \left( 
		\| e^{\tau \varphi}  f_2 \|_{L^2(\Omega)}
		+
		\tau 	\| e^{\tau \varphi}  F \|_{L^2(\Omega)}
		+\tau^{\frac{3}{4}+\frac{1}{2d}} 
		\| e^{\tau \varphi} f_{2*'} \|_{L^{\frac{2d}{d+2}}(\Omega)} 
		\right.
		\\
		\left.
		+ 
		\tau^{\frac{3}{2}} \| e^{\tau \varphi} u\|_{L^2(\omega)}
		+
		\tau^{\frac{3}{4}+\frac{1}{2d}} \| e^{\tau \varphi} u \|_{L^{\frac{2d}{d-2}} (\omega)}
		\right).
		\end{multline}
	\end{theorem} 
	\begin{remark}
		The notations $2*$ and $2*'$ stem from the Sobolev's embedding $H^1(\Omega) \subset L^{2*}(\Omega)$, with $2* = 2d/(d-2)$ and $L^{2*'}(\Omega) \subset H^{-1}(\Omega)$, with $2*' = 2d/(d+2)$.
	\end{remark}
	
	In order to get Theorem \ref{Thm-QUCP}, we prove the following 
	refined version of Theorem \ref{Thm-Carleman-estimates-DET}, whose proof will be given in Section \ref{Sec-Proof-Thm-Carleman-Wolff-type-estimates}.

	\begin{theorem} \label{Thm-Improved-Carleman-estimates}
		Let $d\geqslant 3$. Consider a bounded domain $\Omega \subset \R^{d}$, and non-empty open subsets $\omega_0$ and $\omega$ of $\Omega$ with $\omega_0 \Subset \omega \Subset \Omega$.
		Let $\varphi \in C^3(\overline{\Omega})$ satisfying conditions \eqref{CarlemanWeight-Cond1}--\eqref{CarlemanWeight-Cond2}.
		
		Then, for all compact subset $K$ of $\Omega$, there exist $C>0$ and $\tau_0 \geq 1$ (depending only on $\alpha$, $\beta$, $\| \varphi \|_{C^3(\overline\Omega)}$, and the geometric configuration of $\Omega$, $\omega$, $\omega_0$ and $K$) such that for all $u \in H^1(\Omega)$ satisfying $\operatorname{supp} u \subset K$ and \eqref{Elliptic-Eq-general}
		with $(f_2, f_{2*'}, F = F_2 + F_{2*'})$ satisfying
		\begin{equation}
		\label{Reg-Source-Terms-Improved-Gal} 
		f_2 \in L^2(\Omega), 
		\ \ 
		f_{2*'} \in L^{ \frac{2d}{d+2}}(\Omega),
		\ \ 
		F_2 \in L^2(\Omega; \C^{d}), 
		\ \  \hbox{ and } \ \ 
		F_{2*'} \in L^{ \frac{2d}{d+2}}(\Omega; \C^d)\cap L^2(\Omega; \C^{d}),
		\end{equation}
		we have, for all $\tau \geq \tau_0$, 
		\begin{multline}
		\label{Carleman-General-1-Improved}	
		\tau^{\frac{3}{2}} \| e^{\tau \varphi}  u \|_{L^2(\Omega)} +
		\tau^{\frac{1}{2}} \| e^{\tau \varphi}  \nabla u \|_{L^2(\Omega)} 
		\leq 
		C \left( 
		\| e^{\tau \varphi}  f_2 \|_{L^2(\Omega)}
		+ 
		\tau 	\| e^{\tau \varphi}  F_2 \|_{L^2(\Omega)}+\tau^{\frac{1}{2}}   \| e^{\tau \varphi}  F_{2*'}\|_{L^2(\Omega)}
		\right.
		\\
		\left.
		+
		\tau^{\frac{3}{4}-\frac{1}{2d}}
		\left( \| e^{\tau \varphi} f_{2*'} \|_{L^{\frac{2d}{d+2}} (\Omega)}
		+ 
		\tau 	\| e^{\tau \varphi}  F_{2*'} \|_{L^{\frac{2d}{d+2}}(\Omega)}
		\right)
		+ 
		\tau^{\frac{3}{2}} \| e^{\tau \varphi} u\|_{H^1(\omega)}
		\right), 
		\end{multline}
		and, for all measurable sets $E$ of $\Omega$, 
		\begin{multline}
		\label{Carleman-General-2-Improved}		 
		\tau^{\frac{3}{4}+\frac{1}{2d}} \| e^{\tau \varphi} u \|_{L^{\frac{2d}{d-2}} (\Omega)}
		+ \tau^{\frac{3}{4}+\frac{1}{2d}}  \min\left\{\frac{1}{\tau |E|^{\frac{1}{d}}} ,1\right\}
		\left( \tau \| e^{\tau \varphi} u \|_{L^2(E)} + \| e^{\tau \varphi} \nabla u \|_{L^2(E)}\right) 
		\\
		\leq 
		C \left( 
		\| e^{\tau \varphi}  f_2 \|_{L^2(\Omega)}
		+
		\tau 	\| e^{\tau \varphi}  F_2 \|_{L^2(\Omega)}
		+\tau^{\frac{3}{4}+\frac{1}{2d}} \left(
		\|e^{\tau \varphi} f_{2*'} \|_{L^{\frac{2d}{d+2}}(\Omega)}  + \tau \| e^{\tau \varphi}F_{2*'} \|_{L^{\frac{2d}{d+2}}(\Omega)} + \|e^{\tau \varphi} F_{2*'} \|_{L^2(\Omega)} \right) 
		\right.
		\\
		\left.
		+ 
		\tau^{\frac{7}{4}+\frac{1}{2d}} \| e^{\tau \varphi} u\|_{H^1(\omega)}
		\right).
		\end{multline}
	\end{theorem} 

	\begin{remark}\label{remark-Improved-Compared-to-BET}
		Theorem \ref{Thm-Improved-Carleman-estimates} presents several new features compared to Theorem \ref{Thm-Carleman-estimates-DET}: 
		\begin{itemize}[topsep=0pt,itemsep=0pt,parsep=0pt]
			\item One of the main difference between Theorem \ref{Thm-Improved-Carleman-estimates} and Theorem \ref{Thm-Carleman-estimates-DET} is that Theorem  \ref{Thm-Improved-Carleman-estimates} allows a source term of the form $\div(F_{2*'})$ with $F_{2 *'} \in L^{ \frac{2d}{d+2}}(\Omega; \C^d)\cap L^2(\Omega; \C^{d})$, and quantifies the estimate on the solution $u$ of \eqref{Elliptic-Eq-general} in terms of the $L^{ \frac{2d}{d+2}}(\Omega; \C^d)\cap L^2(\Omega; \C^{d})$-norm of $F_{2*'}$. Note that the $L^2(\Omega; \C^{d})$-norm of $F_{2*'}$ appearing in the right hand sides of \eqref{Carleman-General-1-Improved}--\eqref{Carleman-General-2-Improved} appears with a power of the Carleman parameter which is strictly smaller than the one appearing for the $L^2(\Omega; \C^{d})$-norms of $F_{2}$. In some sense, for $F_{2*'}$, the loss of ellipticity of the conjugated Laplace operator $e^{\tau \varphi} \Delta (e^{-\tau \varphi} \cdot) $ appears within the $L^{ \frac{2d}{d+2}}(\Omega; \C^d)$-norm of $F_{2*'}$ and not within its $L^{ 2}(\Omega; \C^d)$-norm.
			
			\item Estimate \eqref{Carleman-General-2-Improved} also presents an estimate of the $H^1$-norm of $u$ on measurable sets $E$. Similar improvements appear in the work \cite[Section 6]{wolff1992property}. Note that these estimates are particularly relevant when $|E| \lesssim \tau^{-d}$, that is on small measurable subsets, for which we get from \eqref{Carleman-General-2-Improved} an estimate on $\tau^{\frac{7}{4}+ \frac{1}{2d}} \| e^{\tau \varphi} u \|_{L^2(E)} + \tau^{\frac{3}{4}+ \frac{1}{2d}} \| e^{\tau \varphi} \nabla u \|_{L^2(E)}$.  This estimate is reasonable and should be compared with the term $\tau^{\frac{3}{4}+ \frac{1}{2d}} \| e^{\tau \varphi} u \|_{L^{\frac{2d}{d-2}}(\Omega)}$  
			since, by H\"older's estimate, 
			$$
			\tau^{\frac{7}{4}+ \frac{1}{2d}} \| e^{\tau \varphi} u \|_{L^2(E)} \leq \tau^{\frac{3}{4}+ \frac{1}{2d}}\| e^{\tau \varphi} u \|_{L^{\frac{2d}{d-2}}(E)} (\tau |E|^{\frac{1}{d}}) 
			\lesssim \tau^{\frac{3}{4}+ \frac{1}{2d}}\| e^{\tau \varphi} u \|_{L^{\frac{2d}{d-2}}(\Omega)}. 
			$$
			Note however that for large sets $E$, the estimate in  \eqref{Carleman-General-2-Improved} of the $L^2$ norm of $e^{\tau \varphi} u$ is worse than the estimate given by \eqref{Carleman-General-1-Improved}. 
			\item Regarding the observation terms, i.e. the terms involving norms on $\omega$, the estimates \eqref{Carleman-General-1-Improved} and \eqref{Carleman-General-2-Improved} involve the $H^1(\omega)$-norm of $u$ instead of weaker norms as in \eqref{Carleman-General-1-DET} and \eqref{Carleman-General-2-Improved-DET}. In fact, we write it for convenience, as it will be of no impact on the result stated in Theorem \ref{Thm-QUCP}. 
			
			\item In this work, our focus is on deriving local Carleman estimates (in other words, we consider only functions $u$ which are compactly supported). Nevertheless, let us point that, following the proof of Theorem \ref{Thm-Carleman-estimates-DET} in \cite{Dehman-Ervedoza-Thabouti-2023} (in particular Subsection 6.4), Theorem \ref{Thm-Improved-Carleman-estimates} can be extended to functions  $u \in H^1(\Omega)$ with possibly non-homogeneous Dirichlet boundary conditions.
		\end{itemize}
	\end{remark} 
	%
	\section{Improved Carleman estimates: Proof of Theorem \ref{Thm-Improved-Carleman-estimates}} \label{Sec-Proof-Thm-Carleman-Wolff-type-estimates}
	
	This section is devoted to the proof of Theorem \ref{Thm-Improved-Carleman-estimates}. 
	
	Although Theorem \ref{Thm-Improved-Carleman-estimates} looks rather close to Theorem \ref{Thm-Carleman-estimates-DET}, we will need to revise in depth the proof of Theorem \ref{Thm-Carleman-estimates-DET} given in \cite{Dehman-Ervedoza-Thabouti-2023} and incorporate some changes along the way. 
	
	The key step in the proof of Theorem \ref{Thm-Carleman-estimates-DET} is to prove  suitable estimates on the inverse of an operator of the form 
	\begin{equation}
	\label{Conjugated-Operator-strip}
	\Delta -x_{1} \sum_{j=2}^{d} \lambda_{j} \partial_{j}^{2}  
	- 2 \tau \partial_1 + \tau^2, 
	\end{equation}
	in a vertical strip  
	\begin{equation}
	\label{Def-Strip}
	\Omega = (X_0,X_1) \times \R^{d-1}, \text{ with } X_0 < 0 < X_1 \text{ and } \max\{|X_0|, X_1|\} \leq 1, 
	\end{equation}
	where the coefficients $(\lambda_j)_{j \in \{1,\cdots,d\}} \in \R^d$ satisfy $\lambda_1 = 0$, 
	\begin{equation}
	\label{Coercivity}
	\exists c_0 >0, \quad \forall x_1 \in [X_0, X_1],\,  \forall \xi \in \R^{d}, \quad 
	\frac{1}{c_0^2} |\xi|^2 \leq \sum_{j =1}^d (1 - x_1 \lambda_j) |\xi_j|^2 \leq c_0^2 |\xi|^2, 
	\end{equation}
	and 
	\begin{equation}
	\label{Pseudo-Convexity-Strip}
	0< m_*\leq \min_{j \in \{2, \cdots, d\}} \lambda_j 
	\leq \max_{j \in \{2, \cdots, d\}} \lambda_j \leq M_*.
	\end{equation}
	Here, $\tau \geq 1$ plays the role of the Carleman parameter, as the operator in \eqref{Conjugated-Operator-strip} coincides with 
	$$
	e^{\tau x_1} \left(\Delta -x_{1} \sum_{j=2}^{d} \lambda_{j} \partial_{j}^{2}  \right) (e^{-\tau x_1} \cdot ).
	$$
	This really is the main step of the proof of Theorem \ref{Thm-Carleman-estimates-DET}, as one can then use the local character of Carleman estimates to recover Theorem \ref{Thm-Carleman-estimates-DET} in its full generality (see \cite{Dehman-Ervedoza-Thabouti-2023} for details; this strategy will be briefly recalled in Section \ref{Subsec-Local+Glue}). 
	
	Our proof of Theorem \ref{Thm-Improved-Carleman-estimates} follows the same path, and the key estimate is a refined estimate on the inverse of the operator \eqref{Conjugated-Operator-strip} in a strip, which will be done in Subsection \ref{subsec-Carleman-strip}. Once this will be done, the proof of Theorem \ref{Thm-Improved-Carleman-estimates} can be deduced similarly as in \cite{Dehman-Ervedoza-Thabouti-2023} by a suitable localization process, which is rapidly explained in Subsection \ref{Subsec-Local+Glue} for the convenience of the reader. 
	
	%
	%
	%
	\subsection{Main step: Carleman estimates for solutions in a strip.} \label{subsec-Carleman-strip}
	
	In all this section, $\Omega$ is a strip of the form $(X_0,X_1) \times \R^{d-1}$ as in \eqref{Def-Strip}.

	Our goal in this subsection is to prove the following Carleman estimate, which is a refined version of \cite[Theorem 2.2]{Dehman-Ervedoza-Thabouti-2023}.
	
	\begin{theorem}
		\label{Thm-Carleman-Strip-2-improved}
		Let $\Omega = (X_0,X_1) \times \R^{d-1}$ as in \eqref{Def-Strip}, with $d \geq 3$, and assume that the coefficients $(\lambda_j)_{j \in \{1,\cdots,d\}} \in \R^d$ satisfy $\lambda_1 = 0$,  \eqref{Coercivity} and \eqref{Pseudo-Convexity-Strip}.	Then there exist constants $C>0$ and $\tau_0 \geq 1$ depending on $c_0$, $m_*$ and $M_*$, all independent of $X_0, X_1$, such that
		for all $\tau \geq \tau_0$, for all $w \in H^1(\Omega)$ compactly supported in $\Omega$ and satisfying
		\begin{equation}
		\label{Elliptic-Strip-w}
		\Delta w-x_{1} \sum_{j=2}^{d} \lambda_{j} \partial_{j}^{2}  w
		- 2 \tau \partial_1 w + \tau^2 w = f_2 + f_{2*'} + \div (F_2 + F_{2*'}), \qquad \hbox{ in } \Omega, 
		\end{equation}
		with 
		\begin{equation}
		\label{Reg-Source-Terms-Strip-Improved}
		f_2 \in L^2(\Omega), 
		\ \ 
		f_{2*'} \in L^{\frac{2d}{d+2}}(\Omega),
		\ \ 
		F_2 \in L^2(\Omega; \C^{d}), 
		\ \ 
		\hbox{and} 
		\ \  F_{2*'} \in L^{\frac{2d}{d+2}}(\Omega; \C^{d}) \cap L^2(\Omega; \C^{d}),
		\end{equation}
		we have
		\begin{multline}
		\label{Carleman-Strip-w-1-Improved}
		\tau^{\frac{3}{2}} \| w\|_{L^2(\Omega)}
		+ 
		\tau^{\frac{1}{2}} \| \nabla w \|_{L^2(\Omega)}
		\\
		\leq
		C\left(
		\| f_2 \|_{L^2(\Omega)} +
		\tau \| F_2 \|_{L^2(\Omega)}
		+ 
		\tau^{ \frac{3}{4}-\frac{1}{2d}} ( \| f_{2*'} \|_{L^{\frac{2d}{d+2}}(\Omega)}+ \tau \| F_{2*'} \|_{L^{\frac{2d}{d+2}}(\Omega)}) 
		+ 
		\tau^{\frac{1}{2}}  \| F_{2*'} \|_{L^2(\Omega)}  
		\right),
		\end{multline}
		and, and for all measurable subsets $E$ of $\Omega$,
		\begin{multline}
		\label{Carleman-Strip-w-2-Improved}
		\tau^{\frac{3}{4}+\frac{1}{2d}} \| w\|_{L^{\frac{2d}{d-2}}(\Omega)} 
		+ 
		\tau^{\frac{3}{4}+\frac{1}{2d}}  \min\left\{\frac{1}{\tau |E|^{\frac{1}{d}}} ,1\right\}
		\left( \tau \| w \|_{L^2(E)} + \| \nabla w \|_{L^2(E)}\right) 
		\\
		\leq
		C\left(
		\| f_2 \|_{L^2(\Omega)} 
		+  
		\tau \| F_2 \|_{L^2(\Omega)}
		+\tau^{\frac{3}{4}+\frac{1}{2d}} \left(
		\| f_{2*'} \|_{L^{\frac{2d}{d+2}}(\Omega)}  + \tau \| F_{2*'} \|_{L^{\frac{2d}{d+2}}(\Omega)} \right)  + 
		\tau^{\frac{3}{4}+\frac{1}{2d}}  \| F_{2*'} \|_{L^2(\Omega)}
		\right).
		\end{multline}
	\end{theorem}
	The proof of Theorem \ref{Thm-Carleman-Strip-2-improved} is of course very close to the one of \cite[Theorem 2.2.]{Dehman-Ervedoza-Thabouti-2023}, as Theorem \ref{Thm-Carleman-Strip-2-improved} is an improved version of Theorem 2.2 in \cite{Dehman-Ervedoza-Thabouti-2023} with respect to the same points as in Remark \ref{remark-Improved-Compared-to-BET}.
	
	It will require several steps: 
	\begin{enumerate}[topsep=0pt,itemsep=0pt,parsep=0pt]
		\item Construction of an explicit parametrix giving $w$ solving \eqref{Elliptic-Strip-w} in terms of the source terms $(f_2, f_{2*'}, F = F_2 + F_{2*'})$, based on \cite[Proposition 3.1]{Dehman-Ervedoza-Thabouti-2023}, that we will recall hereafter in Subsubsection \ref{Subsubsec-Parametrix}.
		\item Estimates on the operators involved in the parametrix, which will be also mainly corresponding to the ones in Section 6 of \cite{Dehman-Ervedoza-Thabouti-2023}, which we also recall. New estimates will also be needed to get the complete estimates of Theorem \ref{Thm-Carleman-Strip-2-improved}, relying on the decomposition of $F_{2*'}$ in its low and high frequency components, see Subsubsection \ref{Subsubsec-Operators-K}.
		\item A suitable combination of the estimates on the various operators involved in the parametrix, see Subsubsection \ref{Subsubsec-Proof-Thm-Strip}.
	\end{enumerate}
	
	In particular, we will rely upon the following explicit parametrix giving $w$ solution of \eqref{Elliptic-Strip-w} in terms of the source terms $(f_2, f_{2*'}, F = F_2 + F_{2*'})$.
	%
	%
	\subsubsection{An explicit parametrix}\label{Subsubsec-Parametrix}
	
	We use the parametrix constructed in the work \cite{Dehman-Ervedoza-Thabouti-2023}. 
	To do so, we introduce the function $\psi : \overline\Omega \to \R$ defined as follows:
	\begin{equation}
	\label{Def-Psi}
	\psi(x_1, \xi')=\sqrt{ \sum_{j=2}^{d}(1-x_{1}  \lambda_{j}) \xi_{j} ^2},
	\qquad x_1 \in [X_0,X_1], \  \xi'\in \R^{d-1}.
	\end{equation}  
	
	According to \cite[Proposition 3.1]{Dehman-Ervedoza-Thabouti-2023}, we then have the following explicit parametrix:

	\begin{proposition}[{\cite[Proposition 3.1]{Dehman-Ervedoza-Thabouti-2023}}]\label{Prop-parametrix-Strip-1}
		Under the same setting of Theorem \ref{Thm-Carleman-Strip-2-improved}. For all $\tau \geq 1$, if $w$ is compactly supported and satisfies \eqref{Elliptic-Strip-w} with source terms $(f_2, f_{2*'},F_2, F_{2*'})$ as in \eqref{Reg-Source-Terms-Strip-Improved}, then 
		\begin{equation} \label{parametrix1}
		w = K_{\tau, 0} (f_2 + f_{2*'} ) 
		+ \sum_{j = 1}^d K_{\tau,j} ((F_2 + F_{2*'}) \cdot e_j)+R_\tau(w),
		\end{equation}
		where the family of vectors $(e_j)_{j \in \{1, \cdots, d\}}$ is the canonical basis of $\R^d$ and, using the partial Fourier transform, the operators $K_{\tau,j}$, for $j \in \{0,\cdots, d\}$, and $R_\tau$ are defined for $f$ depending on $(x_1, x') \in \Omega$ by 
		\begin{align}
		\widehat {K_{\tau,j} f}(x_1, \xi') & =  \int_{y_1 \in (X_0,X_1)} k_{\tau,j}(x_1,y_1,  \xi') \widehat f(y_1, \xi') \, dy_1, 
		\qquad (x_1, \xi') \in \Omega,
		\label{Def-K-tau-j-Fourier}
		\\
		\widehat {R_{\tau} f}(x_1, \xi') & =  \int_{y_1 \in (X_0,X_1)} r_{\tau}(x_1,y_1,  \xi')  \widehat f(y_1, \xi') \, dy_1, 
		\qquad (x_1, \xi') \in \Omega,
		\label{Def-R-tau-Fourier}
		\end{align}
		with kernels given, for $(x_1,y_1,\xi') \in [X_0,X_1]^2\times \R^{d-1}$, by 
		\begin{align}
		k_{\tau,0}(x_1, y_1, \xi') 
		& = 
		-1_{\psi(x_1,\xi') > \tau} \int_{X_0}^{\min \{ x_1, y_1\}}  e^{-\tau(y_1-x_1)-\int_{\tilde{x}_1}^{x_1} \psi(\tilde{y}_1,\xi') \, d\tilde{y}_1 - \int_{{\tilde{x}_1}}^{y_1} \psi(\tilde{y}_1,\xi') \,d \tilde{y}_1 }  \, d \tilde{x}_1
		\nonumber
		\\
		&\quad +1_{\psi(x_1,\xi') \leq \tau} 1_{y_1>x_1} \int_{x_1}^{y_1}  e^{ -\tau (y_1-x_1)+\int_{x_1}^{\tilde{x}_1} \psi(\tilde{y}_1,\xi') \, d \tilde{y}_1-\int_{{\tilde{x}_1}}^{y_1} \psi(\tilde{y}_1,\xi') \, d y_1} \, d \tilde{x}_1,
		\label{Def-k-tau-0-x1-y1-xi'}
		\\
		k_{\tau,1} (x_1,y_1, \xi')
		& = 
		- 1_{\psi(x_1, \xi') \leq \tau} 1_{y_1>x_1} e^{-\tau(y_1-x_1)+\int_{x_1}^{y_1} \psi(\tilde{y}_1,\xi') \, d \tilde{y}_1}
		\nonumber
		\\ 
		&\quad + 1_{\psi(x_1,\xi') > \tau} 1_{y_1 < x_1} e^{\tau(x_1-y_1) - \int_{y_1}^{x_1} \psi(\tilde{y}_1,\xi') \, d \tilde{y}_1}
		+ k_{\tau,0}(x_1,y_1,\xi') (\tau + \psi(y_1,\xi')),
		\label{Def-k-tau-1-x1-y1-xi'}
		\\
		\label{Def-k-tau-j-x1-y1-xi'}
		k_{\tau,j}(x_1, y_1, \xi') 
		& = \i \xi_j k_{\tau,0}(x_1,y_1, \xi'), \qquad\qquad \qquad  j \in \{2, \cdots, d\},
		\\
		\label{Def-r-tau-x1-y1-xi'}
		r_\tau(x_1,y_1, \xi) 
		& = k_{\tau, 0}(x_1,y_1, \xi') \partial_1 \psi(y_1,\xi').
		\end{align}
	\end{proposition}  
	
	The key to prove Proposition \ref{Prop-parametrix-Strip-1} is to remark that the partial Fourier transform of the operator in \eqref{Conjugated-Operator-strip} is of the form 
	$$
	(\partial_{1} - \tau)^2 - \sum_{j = 2}^d |\xi_j|^2 (1 - x_1 \lambda_j)  = (\partial_1 - \tau - \psi(x_1, \xi')) (\partial_1 - \tau + \psi(x_1,\xi') ) - \partial_{1} \psi(x_1, \xi').  
	$$
	
	It is then clear that the set in which the estimates may degenerate is the set $\{(x_1, \xi') \in (X_0, X_1) \times \R^{d-1}, \text{ with } \psi(x_1, \xi') = \tau \}$.  Accordingly, it is interesting to use a kind of projection operator $P_{hf, \tau}$ on the high-frequency components acting on $L^2(\Omega)$ and given in  Fourier, for $f \in L^2(\Omega)$, by the formula
	\begin{equation}\label{Def-P-hf}
	\widehat {P_{hf, \tau} f }(x_1, \xi' ) = \eta\left( \frac{\psi(X_1, \xi')}{\tau} \right) \widehat f(x_1, \xi'), \qquad \qquad (x_1, \xi') \in (X_0, X_1) \times \R^{d-1}, 
	\end{equation} 
	where $\eta $ is a smooth function in $\mathscr{C}^\infty ([0, \infty), \R)$, taking value $0$ in $[0,2]$, taking value $1$ outside $[0, 3]$, and bounded by $1$.
	
	Note that since the operator $P_{hf, \tau}$ corresponds to a convolution in the $x'$ variable with a $L^1$ kernel $x' \mapsto \tau^{d-1} \eta_1 (\tau x')$ where $ \eta_1(z')=\frac{1}{(2 \pi ) ^{(d-1)/2}}\int_{\R^{d-1}} e^{\i z'\cdot\xi'}\eta(\psi(X_1,\xi')) \,  d \xi'$, one can check  through a simple scaling argument that, for all $p \in [1, \infty]$, there exists $C_p$ such that for all $f \in L^p (\Omega)$, 
	$$
	\| P_{hf, \tau} f \|_{L^p(\Omega)} \leq C_p   \| f \|_{L^p(\Omega)}.
	$$

	This operator presents the advantage of localizing in the frequencies $\xi'$ such that for all $x_1 \in (X_0, X_1)$, $\psi(x_1, \xi') \geq \psi(X_1, \xi') \geq 2 \tau$, and thus far away from the critical set $\{(x_1, \xi') \in (X_0, X_1) \times \R^{d-1}, \text{ with } \psi(x_1, \xi') = \tau \}$. It is also easy to check that it commutes with all the operators $(K_{\tau,j})_{j \in \{0, \cdots, d\}}$ and $R_\tau$.

	This operator can be used in particular on $F_{2*'}$, that we will write as 
	$$
	F_{2*'} = P_{hf, \tau} F_{2*'} + (I - P_{hf, \tau})F_{2*'}.
	$$
	Using the notations $F_{2*'}'$ to denote the last $d-1$ components of $F_{2*'}$ and $\div'$ to denote the divergence operator on $\R^{d-1}_{x'}$, it is easy to check that $ \div'((I- P_{hf,\tau}) F_{2*'}')$ belongs to $L^2(\Omega)$, and thus the identity \eqref{parametrix1} can be written as 
	\begin{multline} \label{parametrix2}
	w = K_{\tau, 0} (f_2 + f_{2*'} + \div'((I- P_{hf,\tau}) F_{2*'}') ) + K_{\tau,1}((F_2 + P_{hf, \tau }F_{2*'} )\cdot e_1) 
	\\
	+
	K_{\tau,1}((( I - P_{hf,\tau}) F_{2*'} )\cdot e_1)
	+ \sum_{j = 2}^d K_{\tau,j} ((F_2 + P_{hf, \tau} F_{2*'}) \cdot e_j)+R_\tau(w).
	\end{multline}
	This is one of the formula that we will use next. Note that it involves all the operators $K_{\tau, j}$ and $R_\tau$ appearing in Proposition \ref{Prop-parametrix-Strip-1}, so that Theorem \ref{Thm-Carleman-Strip-2-improved} will be derived using the known estimates on these operators obtained in \cite{Dehman-Ervedoza-Thabouti-2023} and the suitable gains that we will have by considering how these operators act at low and high-frequencies. 
	
	%
	%
	
	\subsubsection{Boundedness of the operators $K_{\tau,j}$}\label{Subsubsec-Operators-K}
	
	Our proof will be based on estimates for each of the operators $(K_{\tau,j})_{j \in \{0,\cdots, d\}}$. 
	%
	%
	
	\paragraph{Known estimates.} Several estimates have already been obtained in \cite{Dehman-Ervedoza-Thabouti-2023} and are recalled here: 
	
	\begin{proposition}[Proposition 6.2 in \cite{Dehman-Ervedoza-Thabouti-2023}]
		\label{Prop-OpNorm-Est-K0}
		Under the setting of Theorem \ref{Thm-Carleman-Strip-2-improved}.
		There exist $C>0$ and $\tau_0 \geq 1$ independent of $X_0, X_1$ (and depending only on $c_0$, $m_*$ and $M_*$ in \eqref{Coercivity} and \eqref{Pseudo-Convexity-Strip}) such that for all $\tau \geq \tau_0$, 
		for all $f \in L^{\frac{2d}{d+2}}(\Omega)$, 
		\begin{multline}	
		\label{Est-K-tau-0-f-L-2*'}
		\| K_{\tau,0} f\|_{L^{\frac{2d}{d-2}}(\Omega)} 
		+
		\tau^{\frac{3}{4}+\frac{1}{2d}} \| K_{\tau,0} f\|_{L^{2}(\Omega)} 
		+
		\| \partial_1 \widehat{ K_{\tau,0} f}\|_{L^2(\Omega_{1, \tau})}
		+
		\tau^{-\frac{1}{4} +\frac{1}{2d}} \| \nabla' K_{\tau,0} f\|_{L^2(\Omega)} 
		\leq 
		C \| f \|_{L^{\frac{2d}{d+2}}(\Omega)}, 
		\end{multline}
		and, for all $f \in L^2(\Omega)$,
		\begin{multline}	
		\label{Est-K-tau-0-f-L-2}
		\tau^{\frac{3}{4}+\frac{1}{2d}} \| K_{\tau,0} f\|_{L^{\frac{2d}{d-2}}(\Omega)} 
		+
		\tau^{\frac{3}{2}} \| K_{\tau,0} f\|_{L^{2}(\Omega)} 
		+
		\tau
		\| \partial_1 \widehat{ K_{\tau,0} f}\|_{L^2(\Omega_{1, \tau})}
		+
		\tau^{\frac{1}{2}} \| \nabla' K_{\tau,0} f\|_{L^2(\Omega)} 
		\leq 
		C \| f \|_{L^2(\Omega)}, 
		\end{multline}
		with 
		\begin{equation}
		\label{Def-Omega-1-tau}
		\Omega_{1, \tau} = \{(x_1, \xi') \in (X_0, X_1) \times \R^{d-1}, \text{ with } \psi(x_1, \xi') \neq \tau \}.
		\end{equation}
	\end{proposition}
	
	\begin{proposition}[Proposition 6.7 and 6.10 in \cite{Dehman-Ervedoza-Thabouti-2023}]
		\label{Prop-OpNorm-Est-K1}
		Under the setting of Theorem \ref{Thm-Carleman-Strip-2-improved}.
		There exist $C>0$ and $\tau_0 \geq 1$ independent of $X_0, X_1$ (and depending only on $c_0$, $m_*$ and $M_*$ in \eqref{Coercivity} and \eqref{Pseudo-Convexity-Strip}) such that for all $j \in \{1, \cdots, d\}$, for all $\tau \geq \tau_0$ and for all $f \in L^2(\Omega)$,
		\begin{equation*}	
		\label{Est-K-tau-j-f-L-2}
		\tau^{-\frac{1}{4}+\frac{1}{2d}} \| K_{\tau,j} f\|_{L^{\frac{2d}{d-2}}(\Omega)} 
		+
		\tau^{\frac{1}{2}} \| K_{\tau,j} f\|_{L^{2}(\Omega)} 
		+
		\| \partial_1 \widehat{ K_{\tau,j} f}\|_{L^2(\Omega_{1, \tau})}
		+
		\tau^{-\frac{1}{2}} \| \nabla' K_{\tau,j} f\|_{L^2(\Omega)}
		\leq 
		C \| f \|_{L^2(\Omega)}. 
		\end{equation*}
	\end{proposition}
	
	\begin{proposition}[Proposition 6.13 in \cite{Dehman-Ervedoza-Thabouti-2023}]
		\label{Prop-OpNorm-Est-R-tau}
		Under the setting of Theorem \ref{Thm-Carleman-Strip-2-improved}. 
		There exist $C>0$ and $\tau_0 \geq 1$ independent of $X_0, X_1$ (and depending only on $c_0$, $m_*$ and $M_*$ in \eqref{Coercivity} and \eqref{Pseudo-Convexity-Strip}) such that for all $\tau \geq \tau_0$ and for all $f \in H^1(\Omega)$, 
		\begin{equation*}	
		\tau^{\frac{3}{4}+\frac{1}{2d}} \| R_{\tau} f\|_{L^{\frac{2d}{d-2}}(\Omega)} 
		+
		\tau^{\frac{3}{2}} \| R_{\tau} f\|_{L^{2}(\Omega)} 
		+
		\tau 
		\| \partial_1 \widehat{ R_{\tau} f}\|_{L^2(\Omega_{1, \tau})}
		+
		\tau^{\frac{1}{2}} \| \nabla' R_{\tau,0} f\|_{L^2(\Omega)}
		\leq 
		C \| \nabla' f \|_{L^2(\Omega)}. 
		\end{equation*}
	\end{proposition}

	It is natural to obtain better estimates for the operators  $(P_{hf, \tau} K_{\tau, j})_{ j \in \{0, \cdots, d\}}$ than for the operators $(K_{\tau, j})_{j \in \{0, \cdots, d\}}$, since the high-frequency projection operator $P_{hf, \tau}$ is a projection which projects on the part in which the conjugated operator in \eqref{Conjugated-Operator-strip} is elliptic. Although such estimates are known and rather classical in the Hilbertian setting, this needs to be made precise when trying to get estimates on these operators from $L^p(\Omega)$ to $L^q(\Omega)$ when $p$ or $q$ is different from $2$ (we refer to \cite{koch2001carleman} for estimates of that kind in a closely related context). This is precisely our next goal. 
	%
	%
	
	\paragraph{High-frequency estimates.}
	
	We list below the new estimates we obtain on the operators $(K_{\tau,j})_{j \in \{0, \cdots, d\} }$ at high frequencies, to be compared with the ones in Propositions \ref{Prop-OpNorm-Est-K0} and \ref{Prop-OpNorm-Est-K1}. These will be proved next. 
	
	\begin{proposition}
		\label{Prop-OpNorm-Est-K0-Improved}
		Under the setting of Theorem \ref{Thm-Carleman-Strip-2-improved}.
		There exist $C>0$ and $\tau_0 \geq 1$ independent of $X_0, X_1$ (and depending only on $c_0$, $m_*$ and $M_*$ in \eqref{Coercivity} and \eqref{Pseudo-Convexity-Strip}) such that for all $\tau \geq \tau_0$, 
		for all $f \in L^{\frac{2d}{d+2}}(\Omega)$, 
		\begin{equation}	
		\label{Est-K-tau-0-f-L-2*'-HF}
		\tau \| P_{hf, \tau}K_{\tau,0} f\|_{L^{2}(\Omega)} 
		+
		\| \nabla P_{hf, \tau} K_{\tau,0} f\|_{L^2(\Omega)}
		\leq 
		C \| f \|_{L^{\frac{2d}{d+2}}(\Omega)}, 
		\end{equation}
		and, for all $f \in L^2(\Omega)$,
		\begin{equation}	
		\label{Est-K-tau-0-f-L-2-HF}
		%
		\tau \| P_{hf, \tau}K_{\tau,0} f\|_{L^{\frac{2d}{d-2}}(\Omega)} 
		+
		\tau^{2} \|P_{hf, \tau} K_{\tau,0} f\|_{L^{2}(\Omega)} 
		+
		\tau \| \nabla P_{hf, \tau}K_{\tau,0} f\|_{L^2(\Omega)}
		\leq 
		C \| f \|_{L^2(\Omega)}. 
		\end{equation}
	\end{proposition} 
	\begin{remark}
		\label{Remark-gain-K-tau-0}
		It is interesting to compare the estimates in Proposition \ref{Prop-OpNorm-Est-K0-Improved} to the ones in Proposition \ref{Prop-OpNorm-Est-K0}. In particular, one sees that for the Hilbertian estimates, i.e. for the $\mathscr{L}(L^2(\Omega), H^1(\Omega))$ norm of $P_{hf} K_{\tau,0}$, the estimates \eqref{Est-K-tau-0-f-L-2-HF} are better than the ones in \eqref{Est-K-tau-0-f-L-2-HF} for $K_{\tau, 0}$ by a factor $\tau^{\frac{1}{2}}$. This improvement is weaker for the estimates in the  $\mathscr{L} (H^1(\Omega), L^{\frac{2d}{d-2}}(\Omega))$ and $\mathscr{L} (L^{\frac{2d}{d+2}}(\Omega), H^1(\Omega))$-norms of $P_{hf} K_{\tau,0}$, which still gains a factor $\tau^{\frac{1}{4}-\frac{1}{2d}}$ compared to the $\mathscr{L} (H^1(\Omega), L^{\frac{2d}{d-2}}(\Omega))$ and $\mathscr{L} (L^{\frac{2d}{d+2}}(\Omega), H^1(\Omega))$-norms of $K_{\tau,0}$.
	\end{remark}
	
	\begin{proposition}
		\label{Prop-OpNorm-Est-K1-Improved}
		Under the setting of Theorem \ref{Thm-Carleman-Strip-2-improved}. 
		There exist $C>0$ and $\tau_0 \geq 1$ independent of $X_0, X_1$ (and depending only on $c_0$, $m_*$ and $M_*$ in \eqref{Coercivity} and \eqref{Pseudo-Convexity-Strip}), such that for all $j \in \{1, \cdots, d\}$, for all $\tau \geq \tau_0$ and for all $f \in L^2(\Omega)$,
		\begin{equation}	
		\label{Est-K-tau-j-f-L-2-Improved}
		\| P_{hf, \tau} K_{\tau,j} f\|_{L^{\frac{2d}{d-2}}(\Omega)} 
		+
		\tau \| P_{hf, \tau}K_{\tau,j} f\|_{L^{2}(\Omega)} 
		+
		\| \nabla P_{hf, \tau} K_{\tau,j} f\|_{L^2(\Omega)}
		\leq 
		C \| f \|_{L^2(\Omega)}. 
		\end{equation}
	\end{proposition}
	
	\begin{remark}
		\label{Remark-gain-K-tau-j}
		Here again, as in Remark \ref{Remark-gain-K-tau-0}, comparing the estimates in Proposition \ref{Prop-OpNorm-Est-K1-Improved} to the ones in Proposition \ref{Prop-OpNorm-Est-K1}, we see that, for $j \geq 1$, there is again of a factor $\tau^{\frac{1}{4}-\frac{1}{2d}}$ when considering the $\mathscr{L} (L^2(\Omega), L^{\frac{2d}{d-2}}(\Omega))$-norm of the operator $P_{hf, \tau} K_{\tau,j}$ compared to the norm of $K_{\tau,j}$, and of a factor $\tau^{\frac{1}{2}}$ when considering the $\mathscr{L} (L^2(\Omega), H^1(\Omega))$-norms. 
	\end{remark}
	
	Before going into the proofs of Propositions \ref{Prop-OpNorm-Est-K0-Improved} and \ref{Prop-OpNorm-Est-K1-Improved}, we point out that each operator $K_{\tau,j}$ for $j \in \{0, \cdots, d\}$ is a Fourier multiplier operator in the vertical variable. We can therefore use the Stein-Tomas restriction Theorem \cite{tomas1975restriction} to estimate their behavior as an operator from $L^p(\R^{d-1})$ to $L^q(\R^{d-1})$ for some values of $p$ and $q$. This approach is briefly recalled in Appendix \ref{appendix Fourier multiplier operators} with a suitable parametrization of the phase space $\xi' \in \R^{d-1}$ as $\xi' \mapsto (\psi(a, \xi') , \xi' /\psi(a, \xi'))$ for $\psi$ as in \eqref{Def-Psi} and $a \in [X_0, X_1]$ adapted to the kernels appearing in Proposition \ref{Prop-parametrix-Strip-1}. The full details can be found in \cite[Section 5]{Dehman-Ervedoza-Thabouti-2023}.
	
	\begin{proof}[Proof of Proposition \ref{Prop-OpNorm-Est-K0-Improved}]
		In view of the results in Proposition \ref{Prop-Est-Fourier-Op}, we first estimate weighted norms of $k_{\tau,0}(x_1, y_1, \cdot)$ for $x_1$ and $y_1$ in $[X_0,X_1]$ (recall the definition of $k_{\tau,0}$ in \eqref{Def-k-tau-0-x1-y1-xi'}). We also identify $\xi' \in \R^{d-1}$ with pairs $(\lambda, \omega') \in \R_+ \times \Sigma_{x_1}$, where $\Sigma_{x_1} = \{ \omega' \in \R^{d-1}, \psi(x_1, \omega') = 1\}$, through the formula $\xi' = \lambda \omega'$, or equivalently $\lambda = \psi(x_1,\xi')$ and $\omega' = \xi' /\psi(x_1,\xi')$. With a slight abuse of notations, we denote $k_{\tau,0}$ similarly whether it is written in terms of $\xi' \in \R^{d-1}$ or in terms of $(\lambda, \omega') \in \R_+ \times \Sigma_{x_1}$.
		
		From \cite[Lemma 6.4]{Dehman-Ervedoza-Thabouti-2023}, there exists a constant $C>0$ such that for all $\tau \geqslant 1 $, $\lambda > \tau$, and all $x_1, y_1 \in [X_0,X_1]$, we have 
		\begin{equation}
		\label{Est-k-tau-0-x1-y1-High-Freq}
		\| k_{\tau, 0}(x_1,y_1,\lambda, \cdot) \|_{L^\infty(\Sigma_{x_1})} 
		\leq 
		\left\{ 
		\begin{array}{ll}
		\ds \frac{C}{\lambda} e^{-(\lambda-\tau) |x_1-y_1|- \lambda (x_1-y_1)^2/C_1}
		, & \text{ if }  y_1<x_1,
		\smallskip\\
		\ds \frac{C}{\lambda} e^{-(\lambda/C + \tau)|y_1 - x_1|}
		, & \text{ if }  y_1>x_1.
		\end{array}
		\right.
		\end{equation}
		Arguing as in  \cite[Lemma 6.6]{Dehman-Ervedoza-Thabouti-2023}, we deduce 
		\begin{equation}
		\left( \int _{2 \tau  } ^{\infty}\norm{  k_{\tau, 0} (x_1, y_1, \lambda, .)  }_{L^{\infty}(\Sigma_ {x_1}  ) } ^2 \lambda^{1-\frac{2}{d}} d \lambda \right)^{\frac{1}{2}} \leqslant  C e^{- \tau |y_1- x_1|} \tau ^{-\frac{1}{d}}.
		\end{equation}
		
		Accordingly, using Proposition \ref{Prop-Est-Fourier-Op}, Young's inequality and the fact that for all $x_1 \in [X_0, X_1]$, $  \lambda = \psi(x_1, \xi') \geq \psi(X_1, \xi')\geq 2 \tau$ due to  \eqref{Pseudo-Convexity-Strip}, we have, for $f \in L^{\frac{2d}{d+2}}(\Omega)$ and $\tau \geq 1$,
		\begin{align}
		&\| P_{hf, \tau} K_{\tau, 0} f \|_{L^{2}(\Omega)} 
		\leq 
		\norm{ \norm{ P_{hf, \tau} K_{\tau, 0} f(x_1, \cdot) }_{L_{x'}^{2}(\R^{d-1})}}_{L_{x_1}^{2}(X_0,X_1)}
		\notag
		\\
		&\quad
		\leq
		\norm{ \int_{X_0}^{X_1} \left(\int_{2 \tau}^\infty \| k_{\tau, 0}(x_1, y_1, \lambda,\cdot)\|_{L^\infty(\Sigma_{x_1}) }^2 \lambda^{1 - \frac{2}{d} } \, d\lambda\right)^{\frac{1}{2}} \norm{ f(y_1, \cdot)}_{L_{y'}^{\frac{2d}{d+2}}(\R^{d-1}) } d y_1  }_{L_{x_1}^{2}(X_0,X_1)}
		\notag
		\\
		&\quad\leq
		C 
		\norm{ \left(y_1 \mapsto e^{- \tau |y_1|} \tau^{-\frac{1}{d}}\right) \star_{y_1}  \norm{ f(y_1, \cdot)}_{L_{y'}^{\frac{2d}{d+2}}(\R^{d-1})} }_{L_{x_1}^{2}(X_0,X_1)}
		\notag
		\\
		&\quad \leq 
		C \norm{y_1 \mapsto e^{- \tau |y_1|} \tau^{-\frac{1}{d}}}_{L^{\frac{d}{d-1}}(\R)} \norm{  \norm{ f(y_1, \cdot)}_{L_{y'}^{\frac{2d}{d+2}}(\R^{d-1}) }}_{L_{y_1}^{\frac{2d}{d+2}}(X_0,X_1)} = C \tau^{-1} \norm{f}_{L^{\frac{2d}{d+2}}(\Omega)}.
		\label{Est-K-tau-0-hf-L2*'-L2}
		\end{align}
		Similarly, we get that for all $f \in L^{2}(\Omega)$ and $\tau \geq 1$,
		\begin{equation}
		\label{Est-K-tau-0-hf-L2-L2*}
		\| P_{hf, \tau} K_{\tau, 0} f \|_{L^{\frac{2d}{d-2}}(\Omega)}
		\leq 
		C \tau^{-1} \norm{f}_{L^{2}(\Omega)}.	
		\end{equation}
		
		Similarly, arguing as in  \cite[Lemma 6.6]{Dehman-Ervedoza-Thabouti-2023}, we get 
		\begin{equation}\label{Est-Phf-K-tau-0-f-L-2-Improved}
		\left( \int _{2 \tau   } ^{\infty}\norm{ \lambda \omega'   k_{\tau, 0} (x_1, y_1, \lambda, 	\omega')  }_{L^{\infty}_{\omega'}(\Sigma_ {x_1}  ) } ^2 \lambda^{1-\frac{2}{d}} d \lambda \right)^{\frac{1}{2}} \leqslant C | y_1- x_1  |^{-1+ \frac{1}{d}}. 
		\end{equation}
		Using Proposition \ref{Prop-Est-Fourier-Op} and the Hardy-Littlewood-Sobolev theorem in $1$-d, we then get, for all $f \in L^{\frac{2d}{d+2}}(\Omega)$ and $\tau \geq 1$,
		\begin{align}
		&\| \nabla' P_{hf, \tau} K_{\tau, 0} f \|_{L^{2}(\Omega)} 
		\leq 
		\norm{ \norm{ \nabla' P_{hf, \tau} K_{\tau, 0} f(x_1, \cdot) }_{L_{x'}^{2}(\R^{d-1})}}_{L_{x_1}^{2}(X_0,X_1)}
		\notag
		\\
		&\quad
		\leq
		\norm{ \int_{X_0}^{X_1} \left(\int_{2 \tau}^\infty \| \lambda k_{\tau, 0}(x_1, y_1, \lambda,\cdot)\|_{L^\infty(\Sigma_{x_1}) }^2 \lambda^{1 - \frac{2}{d} } \, d\lambda\right)^{\frac{1}{2}} \norm{ f(y_1, \cdot)}_{L_{y'}^{\frac{2d}{d+2}}(\R^{d-1}) } d y_1  }_{L_{x_1}^{2}(X_0,X_1)}
		\notag
		\\
		&\quad\leq
		C 
		\norm{ \left(y_1 \mapsto |y_1|^{-1+ \frac{1}{d}} \right) \star_{y_1}  \norm{ f(y_1, \cdot)}_{L_{y'}^{\frac{2d}{d+2}}}(\R^{d-1}) }_{L_{x_1}^{2}(X_0,X_1)}
		\notag
		\\
		&\quad \leq 
		C \norm{  \norm{ f(y_1, \cdot)}_{L_{y'}^{\frac{2d}{d+2}}(\R^{d-1}) }}_{L_{y_1}^{\frac{2d}{d+2}}(X_0,X_1)} = C \norm{f}_{L^{\frac{2d}{d+2}}(\Omega)}.
		\label{Est-nabla'-K-tau-0-hf-L2*'-L2}
		\end{align}
		We conclude the estimate \eqref{Est-K-tau-0-f-L-2*'-HF} by combining \eqref{Est-K-tau-0-hf-L2*'-L2}, \eqref{Est-nabla'-K-tau-0-hf-L2*'-L2} and \eqref{Est-K-tau-0-f-L-2*'} for the estimate of $\partial_1 P_{hf, \tau} K_{\tau, 0} f$ for $f \in L^{\frac{2d}{d+2}}(\Omega)$.
		
		To conclude \eqref{Est-K-tau-0-f-L-2-HF}, in view of \eqref{Est-K-tau-0-hf-L2-L2*}, we only need to estimate the 	$\mathscr{L}(L^2(\Omega))$ and $\mathscr{L}(L^2(\Omega), H^1(\Omega))$-norms of $P_{hf, \tau} K_{\tau, 0}$. These are easier since the estimate \eqref{Est-k-tau-0-x1-y1-High-Freq} yields that there exists a constant $C>0$ such that for all $\tau \geqslant 1 $, $\lambda > 2 \tau$, and all $x_1, y_1 \in [X_0,X_1]$, we have 
		\begin{equation}
		\label{Est-k-tau-0-x1-y1-High-Freq-bis}
		(\tau + \lambda) \| k_{\tau, 0}(x_1,y_1,\lambda, \cdot) \|_{L^\infty(\Sigma_{x_1})} 
		\leq 
		C e^{ - \tau |x_1 - y_1|}
		\end{equation}
		We then immediately get, by Young's inequality, that for $f \in L^2(\Omega)$ and $\tau \geq 1$, 
		\begin{align}
		&\tau \| P_{hf, \tau} K_{\tau, 0} f \|_{L^{2}(\Omega)} + \| \nabla' P_{hf, \tau} K_{\tau, 0} f \|_{L^{2}(\Omega)} 
		\notag
		\\
		&\quad
		\leq 
		\tau 
		\norm{ \norm{P_{hf, \tau} K_{\tau, 0} f(x_1, \cdot) }_{L_{x'}^{2}(\R^{d-1})}}_{L_{x_1}^{2}(X_0,X_1)}
		+
		\norm{ \norm{ \nabla' P_{hf, \tau} K_{\tau, 0} f(x_1, \cdot) }_{L_{x'}^{2}(\R^{d-1})}}_{L_{x_1}^{2}(X_0,X_1)}
		\notag
		\\
		&\quad
		\leq
		C\norm{ \int_{X_0}^{X_1} e^{ - \tau |x_1 - y_1|} \norm{ f(y_1, \cdot)}_{L_{y'}^{2}(\R^{d-1}) } d y_1  }_{L_{x_1}^{2}(X_0,X_1)}
		\leq 
		C\tau^{-1} \norm{f}_{L^{2}(\Omega)}.
		\label{Est-nabla'-K-tau-0-hf-L2-L2}
		\end{align}
		This estimate, together with \eqref{Est-K-tau-0-hf-L2-L2*} and the estimate \eqref{Est-K-tau-0-f-L-2} for the estimate of $\partial_1 P_{hf, \tau} K_{\tau, 0} f$ for $f \in L^{2}(\Omega)$, gives the estimate  \eqref{Est-K-tau-0-f-L-2-HF}.
	\end{proof} 
	
	\begin{proof}[Proof of Proposition \ref{Prop-OpNorm-Est-K1-Improved}]
		We only sketch the proof of Proposition \ref{Prop-OpNorm-Est-K1-Improved} since it relies on similar arguments as the ones used in the proof of Proposition \ref{Prop-OpNorm-Est-K0-Improved}. 
		
		Let us explain the main steps to get the estimate \eqref{Est-K-tau-j-f-L-2-Improved}. Using \cite[Lemma 6.8]{Dehman-Ervedoza-Thabouti-2023}, we get a constant $C >0$ independent of $X_0, X_1$ (and depending only on $c_0$, $m_*$ and $M_*$ in \eqref{Coercivity} and \eqref{Pseudo-Convexity-Strip}), such that for all $x_1$ and $y_1$ in $[X_0, X_1]$, for all $\tau \geq 1$,  and $\lambda >0$,
		\begin{equation}
		\label{Est-k-tau-1-x1-y1}
		\| k_{\tau, 1}(x_1,y_1,\lambda, \cdot) \|_{L^\infty(\Sigma_{x_1})} 
		\leq  C e^{-|\tau - \lambda| | y_1 - x_1 | - \lambda  (y_1 - x_1 )^2/C } + C (\tau + \lambda ) \| k_{\tau, 0}(x_1,y_1,\lambda, \cdot) \|_{L^\infty(\Sigma_{x_1})}. 
		\end{equation}
		and, for $j \in \{2, \cdots, d\}$,  
		\begin{equation}
		\label{Est-k-tau-j-x1-y1}
		\| k_{\tau, j}(x_1,y_1,\lambda, \cdot) \|_{L^\infty(\Sigma_{x_1})} 
		\leq  C \lambda \| k_{\tau, 0}(x_1,y_1,\lambda, \cdot) \|_{L^\infty(\Sigma_{x_1})}. 
		\end{equation}
		
		The estimate \eqref{Est-Phf-K-tau-0-f-L-2-Improved} then yields the existence of a constant $C>0$ such that for all $\tau \geqslant 1$, $j \in \{2, \cdots, d\}$, and all $x_1, y_1 \in [X_0,X_1]$, 
		\begin{equation}
		\left( \int _{2 \tau } ^{\infty}\norm{ k_{\tau, j} (x_1, y_1, \lambda, .)  }_{L^{\infty}(\Sigma_ {x_1}  ) } ^2 \lambda^{1-\frac{2}{d}} d \lambda \right)^{\frac{1}{2}} \leqslant  C  | y_1- x_1  |^{-1+ \frac{1}{d}}. 
		\end{equation} 
		Similarly, one can derive from \eqref{Est-k-tau-1-x1-y1} that this also holds for $j = 1$.
		
		Using Proposition \ref{Prop-Est-Fourier-Op} and the Hardy-Littlewood-Sobolev theorem, we then deduce that there exists $C>0$ such that for all $j \in \{1, \cdots, d\}$, for all $\tau \geq 1$ and for all $f \in L^2(\Omega)$,
		\begin{equation*}	
		\| P_{hf,\tau}K_{\tau,j} f\|_{L^{\frac{2d}{d-2}}(\Omega)} 
		\leq 
		C \| f \|_{L^2(\Omega)}.
		\end{equation*}

		The estimates on the $\mathscr{L} (L^2(\Omega))$-norms of $P_{hf,\tau} K_{\tau,j}$ and $\nabla' P_{hf,\tau} K_{\tau,j}$ can be achieved more easily and are left to the reader, and the $\mathscr{L} (L^2(\Omega))$-norm of $\partial_1 P_{hf,\tau} K_{\tau,j}$ follows from Proposition \ref{Prop-OpNorm-Est-K1}.
	\end{proof} 
	%
	%
	
	%
	\paragraph{Low-frequency estimates.} In our arguments next, we will also need to understand the behavior of the  operator $(I - P_{hf, \tau}) K_{\tau,1}$ and show how it acts on $L^{\frac{2d}{d+2}}(\Omega)$:
	\begin{proposition}
		\label{Prop-OpNorm-Est-K1-Improved-LF}
		Under the setting of Theorem \ref{Thm-Carleman-Strip-2-improved}.
		There exist $C>0$ and $\tau_0 \geq 1$ independent of $X_0, X_1$ (and depending only on $c_0$, $m_*$ and $M_*$ in \eqref{Coercivity} and \eqref{Pseudo-Convexity-Strip}) such that for all $\tau \geq \tau_0$ and for all $f \in L^{\frac{2d}{d+2}}(\Omega)$,
		\begin{multline}	
		\label{Est-K-tau-1-f-L-2*'-Improved-LF}
		\| (I-P_{hf, \tau} )K_{\tau,1} f\|_{L^{\frac{2d}{d-2}}(\Omega)} 
		+
		\tau^{\frac{3}{4}+\frac{1}{2d}} \| (I-P_{hf, \tau} ) K_{\tau,1} f\|_{L^{2}(\Omega)} 
		\\
		+
		\| \partial_1 \widehat{(I-P_{hf, \tau} ) K_{\tau,1} f}\|_{L^2(\Omega_{1, \tau})}
		+
		\tau^{-\frac{1}{4}+\frac{1}{2d}} \| \nabla' (I-P_{hf, \tau} )  K_{\tau,1} f\|_{L^2(\Omega)}
		\leq 
		C \tau \| f \|_{L^{\frac{2d}{d+2}}(\Omega)},  
		\end{multline}
		where $\Omega_{1, \tau}$ is the set defined in \eqref{Def-Omega-1-tau}.
	\end{proposition}
	
	\begin{proof} 
		Similarly as in the previous proofs, we will use suitable bounds on the kernel $k_{\tau,1}$. Namely, we will use the bound \eqref{Est-k-tau-1-x1-y1}, the bound \eqref{Est-k-tau-0-x1-y1-High-Freq} and the following bound, obtained in \cite[Lemma 6.4]{Dehman-Ervedoza-Thabouti-2023}:  	There exist constants $C>0$ and $C_1 >0$ independent of $X_0, X_1$ (and depending only on $c_0$, $m_*$ and $M_*$ in \eqref{Coercivity} and \eqref{Pseudo-Convexity-Strip}) such that for all $x_1$ and $y_1$ in $[X_0, X_1]$, for all $\tau \geq 1$,  and $\lambda \leq \tau $,
		the kernel $k_{\tau,0}$ defined in \eqref{Def-k-tau-0-x1-y1-xi'}  satisfies 
		\begin{equation}
		\label{Est-k-tau-0-x1-y1-Low-Freq}
		\| k_{\tau, 0}(x_1,y_1,\lambda, \cdot) \|_{L^\infty(\Sigma_{x_1})} 
		\leq 
		\left\{
		\begin{array}{ll}
		\ds  C | y_1- x_1 | e^{-\tau | y_1- x_1 | },  \quad & \text{ if }  \ \lambda |y_1 -x_1| \leqslant 1,
		\smallskip\\		
		\ds \frac{C}{\lambda }e^{- (\tau - \lambda) |y_1- x_1|  -  \lambda (y_1- x_1) ^ 2/C_1} , \quad & \text{ if }  \lambda |y_1 -x_1| \geqslant 1.
		\end{array}
		\right.  
		\end{equation}

		Lemma 6.9 in \cite{Dehman-Ervedoza-Thabouti-2023} states that there exists a constant $C >0$  independent of $X_0, X_1$ (and depending only on $c_0$, $m_*$ and $M_*$ in \eqref{Coercivity} and \eqref{Pseudo-Convexity-Strip}) such that for all $x_1$ and $y_1$ in $[X_0, X_1]$, for all $\tau \geq 1$, 
		\begin{align}
		\label{est-k1-x1-y1-Lp-L2}
		& 
		\left( \int_{\lambda >0} \norm{k_{\tau,1}(x_1, y_1,\lambda, \cdot) }^2_{L^\infty(\Sigma_{x_1})} \, \lambda^{1 - \frac{2}{d}}\, d\lambda\right)^{\frac{1}{2}}
		\leq C\frac{1}{|x_1-y_1|^{1-\frac{1}{d}}} +  \tilde k_{\tau,1}(x_1-y_1), 
		\\
		& \qquad \text{ with $\tilde k_{\tau,1} \in L^{\frac{d}{d-1}} (\R)$ and }
		\|\tilde k_{\tau,1}\|_{L^{\frac{d}{d-1} } (\R)} \leq C \tau^{\frac{1}{4}-\frac{1}{2d}}.
		\notag
		\end{align} 
		Using then Proposition \ref{Prop-Est-Fourier-Op}, the Hardy-Littlewood-Sobolev theorem and Young's inequality, we deduce that the $\mathscr{L}(L^{\frac{2d}{d+2}}(\Omega),L^{2}(\Omega))$-norm of $ (I-P_{hf, \tau} )K_{\tau,1}$ is bounded by $C \tau^{\frac{1}{4}-\frac{1}{2d}}$.

		To get a bound on the $\mathscr{L}(L^{\frac{2d}{d+2}}(\Omega),L^{\frac{2d}{d-2}}(\Omega))$-norm of  $(I-P_{hf, \tau} )K_{\tau,1}$ (recall that $I -P_{hf, \tau}$ localizes at frequency $\xi'$ such that $\psi(X_1, \xi') \leq 3 \tau$), we first show that there exists a constant $C >0$, such that for all $x_1$ and $y_1$ in $[X_0, X_1]$, for all $\tau \geq 1$,
		\begin{align*}
		&     \int_0^{3 c_1 \tau} \| k_{\tau, 1}(x_1,y_1,\lambda, \cdot) \|_{L^\infty(\Sigma_{x_1})} \lambda^{1-\frac{2}{d}} d \lambda 
		\\
		& \leq  C \int_0^{3c_1 \tau} e^{-|\tau - \lambda| | y_1 - x_1 | - \lambda  (y_1 - x_1 )^2/C }  \lambda^{1-\frac{2}{d}} d \lambda +\tau  C\int_0^{3c_1\tau}	  \| k_{\tau, 0}(x_1,y_1,\lambda, \cdot) \|_{L^\infty(\Sigma_{x_1})}\lambda^{1-\frac{2}{d}} d \lambda 
		\\
		& \leqslant C \tau |x_1-y_1|^{-1+\frac{2}{d}},
		\end{align*}
			where the bound on the first term on the right-hand side is derived through simple calculations similar to \cite[Section 6]{Dehman-Ervedoza-Thabouti-2023}. For the second term, we have used \cite[Lemma 6.6]{Dehman-Ervedoza-Thabouti-2023}, where $c_1$ is defined by 
		$$
		c_1 = \sup_{x_1 \in [X_0, X_1]} \sup \{ \psi(x_1, \xi'), \, \text{s. t. } \psi(X_1, \xi') = 1\}. 
		$$ 
		Using then Proposition \ref{Prop-Est-Fourier-Op} and the Hardy-Littlewood-Sobolev theorem, the $\mathscr{L}(L^{\frac{2d}{d+2}}(\Omega),L^{\frac{2d}{d-2}}(\Omega))$-norm of  $(I-P_{hf, \tau} )K_{\tau,1}$ is bounded by $C \tau$.

		For the estimate on the $\mathscr{L}(L^{\frac{2d}{d+2}}(\Omega),L^{\frac{2d}{d-2}}(\Omega))$-norm $\nabla'(I-P_{hf,\tau}) K_{\tau,1}$, the crucial point is to prove that there exists a constant $C>0$ such that for all $x_1, y_1 \in [X_0, X_1]$, and for all $\tau \geqslant 1$, 
		\begin{equation}
		\label{est-k1-x1-y1-Lp-nabla'-L2}		
		\left(
		\int_{0}^{3 c_1 \tau} \norm{ \lambda \omega' k_{\tau,1}(x_1, y_1, \lambda, \omega') }^2_{L_{\omega'}^\infty(\Sigma_{x_1})} \, \lambda^{1 -\frac{2}{d}}\, d\lambda
		\right)^{\frac{1}{2}}
		\leq 
		C \tau^{\frac{5}{4}-\frac{1}{2d}} |x_1 - y_1|^{-1+\frac{1}{d}}. 
		\end{equation}
		If so, using again Proposition \ref{Prop-Est-Fourier-Op} and the Hardy-Littlewood-Sobolev
		inequality,  the $\mathscr{L}(L^{\frac{2d}{d+2}}(\Omega),L^{\frac{2d}{d-2}}(\Omega))$-norm $\nabla'(I-P_{hf,\tau}) K_{\tau,1}$ is bounded by $\tau^{\frac{5}{4}-\frac{1}{2d}}$.

		To prove inequality \eqref{est-k1-x1-y1-Lp-nabla'-L2}, we can bound the term on the left-hand side using \eqref{Est-k-tau-1-x1-y1} as follows
		\begin{align*}
		C \left(\int_0^{3 c_1 \tau} e^{-|\tau - \lambda| | y_1 - x_1 | - \lambda  (y_1 - x_1 )^2/C }  \lambda^{3-\frac{2}{d}} d \lambda \right)^{\frac{1}{2}} 
		+C \tau  \left(\int_0^{3 c_1 \tau}	  \| \lambda \omega' k_{\tau, 0}(x_1,y_1,\lambda, \omega') \|^2_{L_{\omega'}^\infty(\Sigma_{x_1})}\lambda^{1-\frac{2}{d}} d \lambda \right)^{\frac{1}{2}}.   
		\end{align*}
		Then, using \cite[Lemma 6.6]{Dehman-Ervedoza-Thabouti-2023}, we obtain the desired estimate on the second term. A straightforward computation yields the same bound on the first term.

		The estimate the $\mathscr{L}(L^{\frac{2d}{d+2}}(\Omega),L^{\frac{2d}{d-2}}(\Omega))$-norm
		of $\partial_1(I-P_{hf,\tau})K_{\tau,1}$ is more technical. First, we compute the kernel of the operator $\partial_{1}K_{\tau,1}$, denoted by $k_{\tau,1, \partial_1} $: for all $(x_1, \xi') \in \Omega_{1, \tau}$ and all $y_1 \in (X_0,X_1)$,
		\begin{align}	\label{Def-k-tau-1-d1-x1-y1-xi'}
		k_{\tau,1, \partial_1} (x_1,y_1, \xi')
		&=  -  1_{\psi(x_1, \xi') \leq \tau} 1_{y_1>x_1} (\tau-\psi(x_1,\xi')) e^{-\tau(y_1-x_1)+\int_{x_1}^{y_1} \psi(\tilde{y}_1,\xi') \, d \tilde{y}_1} \nonumber
		\\ 
		& + 1_{ \tau < \psi(x_1,\xi') } 1_{y_1 < x_1} (\tau-\psi(x_1,\xi')) e^{\tau(x_1-y_1) - \int_{y_1}^{x_1} \psi(\tilde{y}_1,\xi') \, d \tilde{y}_1} \nonumber 
		\\ 
		&+ k_{\tau,0,\partial_1}(x_1,y_1,\xi') (\tau + \psi(y_1,\xi')), 
		\end{align}
		where $k_{\tau, 0, \partial_1}$ is defined for $x_1, \ y_1$ in $[X_0, X_1]$ and $\xi' \in \R^{d-1}$, by
		\begin{align*}
		k_{\tau,0,\partial_1} (x_1, y_1, \xi')
		&= 
		-1_{x_1 < y_1} e^{-\tau(y_1-x_1) - \int_{x_1}^{y_1} \psi(\tilde y_1, \xi') \, d\tilde y_1}
		\\
		&	\ \ - 1_{\psi(x_1, \xi')> \tau} 
		(\tau - \psi(x_1, \xi')) 
		\int_{X_0}^{\min\{x_1,y_1\}} e^{-\tau(y_1-x_1) - \int_{\tilde x_1}^{x_1} \psi(\tilde y_1, \xi') \, d\tilde y_1 - \int_{\tilde x_1}^{y_1} \psi(\tilde y_1, \xi') \, d\tilde y_1} d\tilde x_1
		\notag
		\\
		&	\ \ + 1_{\psi(x_1, \xi')\leq \tau}  1_{x_1 < y_1} (\tau - \psi(x_1, \xi')) \int_{x_1}^{y_1} 
		e^{-\tau(y_1-x_1) + \int_{x_1}^{\tilde x_1} \psi(\tilde y_1, \xi') \, d\tilde y_1 - \int_{\tilde x_1}^{y_1} \psi(\tilde y_1, \xi') \, d\tilde y_1} d\tilde x_1. 
		\notag
		\end{align*}
		The kernel of $\partial_{1}(I- P_{hf,\tau})K_{\tau,1}$, denoted by $k_{\tau,1, \partial_1,lf} $ is then given by
		\begin{align}	\label{Def-k-tau-1-d1-lf-x1-y1-xi'}
		k_{\tau,1, \partial_1,lf} (x_1,y_1, \xi')
		&=  \left(1-\eta\left(\frac{\psi(X_1,\xi')}{\tau}\right)\right) k_{\tau,1,\partial_1}(x_1,y_1,\xi').
		\end{align}
		
		A tedious computation (similar to the ones in \cite[Lemma 6.4]{Dehman-Ervedoza-Thabouti-2023}) then shows that there exists a constant $C >0$ independent of $X_0$ and $X_1$, such that for all $x_1$ and $y_1$ in $[X_0, X_1]$, for all $\tau \geq 1$, and $0<\lambda <3 c_1 \tau$,
		\begin{equation*}
		\| k_{\tau, 1,\partial_1,lf}(x_1,y_1,\lambda, \cdot) \|_{L^\infty(\Sigma_{x_1})} 
		\leq  C \tau \left( e^{-|\tau - \lambda| | y_1 - x_1 | - \lambda  (y_1 - x_1 )^2/C } +  \| k_{\tau, 0,\partial_1}(x_1,y_1,\lambda, \cdot) \|_{L^\infty(\Sigma_{x_1})} \right). 
		\end{equation*}
		As a consequence, using \cite[Lemma 6.6]{Dehman-Ervedoza-Thabouti-2023} to bound the second term on the right-hand side of the last inequality, we can prove
		\begin{equation}
		\label{est-k1-d1-x1-y1-Lp-L2}			
		\left(\int_{0}^{3 c_1 \tau} \norm{   k_{\tau,1, \partial_1}(x_1, y_1, \lambda, \cdot) }^2_{L^\infty (\Sigma_{x_1})} \, \lambda^{1 -\frac{2}{d}}\, d\lambda\right)^{\frac{1}{2}}
		\leq
		C  \tau |x_1 - y_1|^{-1+\frac{1}{d}}.
		\end{equation}
		Using the Hardy-Littlewood-Sobolev inequality then yields that  the $\mathscr{L}(L^{\frac{2d}{d+2}}(\Omega),L^{\frac{2d}{d-2}}(\Omega))$-norm
		of $\partial_1(I-P_{hf,\tau})K_{\tau,1}$ is bounded by $C \tau$. This concludes the proof of Proposition \ref{Prop-OpNorm-Est-K1-Improved-LF}.
	\end{proof} 
	 
	%
	%
	\subsubsection{Proof of Theorem \ref{Thm-Carleman-Strip-2-improved} }\label{Subsubsec-Proof-Thm-Strip}
	
	With the various estimates established in the propositions presented in the previous subsection, we are now in a position to prove Theorem \ref{Thm-Carleman-Strip-2-improved}, that is the Carleman estimates in the strip. 
	
	\begin{proof}[Proof of Theorem \ref{Thm-Carleman-Strip-2-improved}] 
		Using Proposition \ref{Prop-parametrix-Strip-1}, if $w$ is compactly supported and satisfies \eqref{Elliptic-Strip-w} for some source terms $(f_2, f_{2*'},F_2, F_{2*'})$ as in \eqref{Reg-Source-Terms-Strip-Improved}, then 
		\begin{equation*}
		w = K_{\tau, 0} (f_2 + f_{2*'} ) 
		+ \sum_{j = 1}^d K_{\tau,j} ((F_2 + F_{2*'}) \cdot e_j)+R_\tau(w).
		\end{equation*}
		
		Recall that the operator $P_{hf, \tau}$ commutes with all the operators $(K_{\tau,j})_{j \in \{0, \cdots, d\}}$ and $R_\tau$. Accordingly, the high-frequency part of $w_{hf, \tau } = P_{hf, \tau} w$ satisfies
		$$
		w_{hf, \tau} = P_{hf,\tau}  K_{\tau, 0} (f_2 + f_{2*'} ) 
		+ \sum_{j = 1}^d P_{hf,\tau}  K_{\tau,j} ((F_2 + F_{2*'}) \cdot e_j)+  R_\tau (w_{hf, \tau}).
		$$
		Using the various estimates in Propositions \ref{Prop-OpNorm-Est-K0-Improved}, \ref{Prop-OpNorm-Est-K1-Improved}, and \ref{Prop-OpNorm-Est-R-tau}, at high-frequency, we obtain 
		\begin{multline} \label{Carleman-Strip-w-hf-2-Improved-Preliminary0}
		\tau^{\frac{3}{2}} \| w_{hf,\tau}\|_{L^{2}(\Omega)} +\tau^{\frac{1}{2}} \| \nabla w_{hf,\tau}\|_{L^{2}(\Omega)}
		\\
		\leq
		C\left(
		\tau^{-\frac{1}{2}} \| f_2 \|_{L^2(\Omega)} 
		+ 
		\tau^{\frac{1}{2}} \| f_{2*'} \|_{L^{\frac{2d}{d+2}}(\Omega)} 
		+
		\tau^{\frac{1}{2}}  \| F_2 + F_{2*'} \|_{L^2(\Omega)}
		\right)
		+ C \| \nabla' w_{hf, \tau} \|_{L^2(\Omega)},
		\end{multline}
		and we have also (using also \eqref{Est-K-tau-0-f-L-2*'} and the fact that $\|P_{hf,\tau}\|_{L^p} \leqslant C_p$  for $p=2d/(d-2)$)
		\begin{multline}
		\label{Carleman-Strip-w-hf-2*'-Improved-Preliminary0}
		\tau^{\frac{3}{4}+\frac{1}{2d}} \| w_{hf,\tau}\|_{L^{\frac{2d}{d-2}}(\Omega)} 
		\\
		\leq
		C\left(
		\tau^{- \frac{1}{4} + \frac{1}{2d}} \| f_2 \|_{L^2(\Omega)} 
		+ 
		\tau^{\frac{3}{4}+\frac{1}{2d}} \| f_{2*'} \|_{L^{\frac{2d}{d+2}}(\Omega)}  
		+
		\tau^{\frac{3}{4}+\frac{1}{2d}} \| F_2 + F_{2*'} \|_{L^2(\Omega)}
		\right)
		+ C \| \nabla' w_{hf, \tau} \|_{L^2(\Omega)}. 
		\end{multline}
		Accordingly, there exists $\tau_0 >0$ such that for all $\tau \geq \tau_0$, 
		\begin{multline} \label{Carleman-Strip-w-hf-2-Improved-Preliminary}
		\tau^{\frac{3}{2}} \| w_{hf,\tau}\|_{L^{2}(\Omega)} +\tau^{\frac{1}{2}} \| \nabla w_{hf,\tau}\|_{L^{2}(\Omega)}
		\\
		\leq
		C\left(
		\tau^{-\frac{1}{2}} \| f_2 \|_{L^2(\Omega)} 
		+ 
		\tau^{\frac{1}{2}} \| f_{2*'} \|_{L^{\frac{2d}{d+2}}(\Omega)} 
		+
		\tau^{\frac{1}{2}}  \| F_2 + F_{2*'} \|_{L^2(\Omega)}
		\right),
		\end{multline}
		and thus
		\begin{multline}
		\label{Carleman-Strip-w-hf-2*'-Improved-Preliminary}
		\tau^{\frac{3}{4}+\frac{1}{2d}} \| w_{hf,\tau}\|_{L^{\frac{2d}{d-2}}(\Omega)} 
		\leq
		C\left(
		\tau^{- \frac{1}{4} + \frac{1}{2d}} \| f_2 \|_{L^2(\Omega)} 
		+ 
		\tau^{\frac{3}{4}+\frac{1}{2d}} \| f_{2*'} \|_{L^{\frac{2d}{d+2}}(\Omega)}  
		+
		\tau^{\frac{3}{4}+\frac{1}{2d}} \| F_2 + F_{2*'} \|_{L^2(\Omega)}
		\right). 
		\end{multline}
		
		We then define the low frequency part $w_{lf, \tau} = (I- P_{hf, \tau})w$ of $w$: Using \eqref{parametrix2}, we get 
		\begin{multline*} 
		w_{lf, \tau} = K_{\tau, 0} ( ( I - P_{hf,\tau}) (f_2 + f_{2*'}) + \div'((I- P_{hf,\tau}) F_{2*'}') ) + K_{\tau,1}(( I - P_{hf,\tau}) F_2 \cdot e_1) 
		\\
		+
		( I - P_{hf,\tau}) K_{\tau,1}( F_{2*'} \cdot e_1)
		+ \sum_{j = 2}^d K_{\tau,j} (( I - P_{hf,\tau}) F_2 \cdot e_j)+R_\tau(  w_{lf, \tau} ).
		\end{multline*}
		Applying Bernstein's inequality  (see, for instance, \cite[Lemma 2.1.]{Bahouri-Chemin-Danchin-book}), we get that $\div' ((I- P_{hf,\tau}) F_{2*'}')$ belongs to $L^{\frac{2d}{d+2}}(\Omega)$ and 
		$$
		\| \div' ((I- P_{hf,\tau}) F_{2*'}') \|_{L^{\frac{2d}{d+2}}(\Omega)} \leq C \tau \| F_{2*'} \|_{L^{\frac{2d}{d+2}}(\Omega)}.
		$$	
		
		We then use Propositions \ref{Prop-OpNorm-Est-K0}, \ref{Prop-OpNorm-Est-K1}, \ref{Prop-OpNorm-Est-R-tau}, and \ref{Prop-OpNorm-Est-K1-Improved-LF} to obtain  
		\begin{multline}\label{Carleman-Strip-w-lf-2-Improved-Preliminary0}
		\tau^{\frac{3}{2}} \| w_{lf,\tau}\|_{L^2(\Omega)}
		+ 
		\tau^{\frac{1}{2}} \| \nabla w_{lf,\tau} \|_{L^2(\Omega)}
		\\
		\leq
		C\left(
		\| f_2 \|_{L^2(\Omega)} +
		\tau \| F_2 \|_{L^2(\Omega)}
		+ 
		\tau^{ \frac{3}{4}-\frac{1}{2d}} \left( \| f_{2*'} \|_{L^{\frac{2d}{d+2}}(\Omega)}+ \tau \| F_{2*'} \|_{L^{\frac{2d}{d+2}}(\Omega)}\right) 
		\right)
		+ C \| \nabla' w_{lf,\tau} \|_{L^2(\Omega)},
		\end{multline}
		and 
		\begin{multline} \label{Carleman-Strip-w-lf-2*'-Improved-Preliminary0}
		\tau^{\frac{3}{4}+\frac{1}{2d}} \| w_{lf,\tau}\|_{L^{\frac{2d}{d-2}}(\Omega)} 
		\\
		\leq
		C\left(
		\| f_2 \|_{L^2(\Omega)} 
		+
		\tau \| F_2 \|_{L^2(\Omega)}+ 
		\tau^{\frac{3}{4}+\frac{1}{2d}}\left( \| f_{2*'} \|_{L^{\frac{2d}{d+2}}(\Omega)}  + \tau \| F_{2*'} \|_{L^{\frac{2d}{d+2}}(\Omega)} \right)  \right)
		+ C \| \nabla' w_{lf,\tau} \|_{L^2(\Omega)}.
		\end{multline}
		As before, we can then absorb the term $ \| \nabla' w_{lf,\tau} \|_{L^2(\Omega)}$ in the right hand-side of \eqref{Carleman-Strip-w-lf-2-Improved-Preliminary0} 
		by choosing $\tau$ large enough: Taking $\tau_0 \geq 1$ larger if necessary, we get, for $\tau \geq \tau_0$,
		\begin{multline}\label{Carleman-Strip-w-lf-2-Improved-Preliminary}
		\tau^{\frac{3}{2}} \| w_{lf,\tau}\|_{L^2(\Omega)}
		+ 
		\tau^{\frac{1}{2}} \| \nabla w_{lf,\tau} \|_{L^2(\Omega)}
		\\
		\leq
		C\left(
		\| f_2 \|_{L^2(\Omega)} +
		\tau \| F_2 \|_{L^2(\Omega)}
		+ 
		\tau^{ \frac{3}{4}-\frac{1}{2d}} \left( \| f_{2*'} \|_{L^{\frac{2d}{d+2}}(\Omega)}+ \tau \| F_{2*'} \|_{L^{\frac{2d}{d+2}}(\Omega)}\right) 
		\right),
		\end{multline}
		and, consequently,  
		\begin{equation} \label{Carleman-Strip-w-lf-2*'-Improved-Preliminary}
		\tau^{\frac{3}{4}+\frac{1}{2d}} \| w_{lf,\tau}\|_{L^{\frac{2d}{d-2}}(\Omega)} 
		\leq
		C\left(
		\| f_2 \|_{L^2(\Omega)} 
		+
		\tau \| F_2 \|_{L^2(\Omega)}+ 
		\tau^{\frac{3}{4}+\frac{1}{2d}}\left( \| f_{2*'} \|_{L^{\frac{2d}{d+2}}(\Omega)}  + \tau \| F_{2*'} \|_{L^{\frac{2d}{d+2}}(\Omega)} \right)  \right).
		\end{equation}
		
		Combining estimates \eqref{Carleman-Strip-w-hf-2-Improved-Preliminary} and \eqref{Carleman-Strip-w-lf-2-Improved-Preliminary}, we deduce the  Carleman estimate \eqref{Carleman-Strip-w-1-Improved}. Similarly, combining estimates \eqref{Carleman-Strip-w-hf-2*'-Improved-Preliminary} and \eqref{Carleman-Strip-w-lf-2*'-Improved-Preliminary}, we deduce the Carleman estimate \eqref{Carleman-Strip-w-2-Improved} except for the localized estimates on the set $E$.
		
		To proceed with \eqref{Carleman-Strip-w-2-Improved}, we thus focus on the localized estimates within the subset $E \subset \Omega$. In order to do so, on one hand, at low frequencies, we use Hölder and Bernstein estimates:
		$$
		\tau \| w_{lf,\tau} \|_{L^2(E)} +\|  \nabla' w_{lf,\tau} \|_{L^2(E)}
		\leq 
		|E|^{\frac{1}{d}}\left(\tau \| w_{lf,\tau} \|_{L^{\frac{2d}{d-2}}(E)} + \| \nabla' w_{lf,\tau} \|_{L^{\frac{2d}{d-2}}(E)}\right)
		\leq 
		C |E|^{\frac{1}{d}} \tau \| w_{lf, \tau} \|_{L^{\frac{2d}{d-2}}(\Omega)} .
		$$
		Accordingly, multiplying the above estimate by $\tau^{\frac{3}{4}+\frac{1}{2d}}$ and using \eqref{Carleman-Strip-w-lf-2*'-Improved-Preliminary}, we deduce 
		\begin{multline*}
		\tau^{\frac{3}{4}+\frac{1}{2d}} \left(\tau \| w_{lf,\tau} \|_{L^2(E)} +\|  \nabla' w_{lf,\tau} \|_{L^2(E)}\right)
		\leq 
		C  |E|^{\frac{1}{d}} \tau 
		\left(
		\| f_2 \|_{L^2(\Omega)} 
		+
		\tau \| F_2 \|_{L^2(\Omega)}
		\right.
		\\
		\left.
		+ 
		\tau^{\frac{3}{4}+\frac{1}{2d}}( \| f_{2*'} \|_{L^{\frac{2d}{d+2}}(\Omega)}  + \tau \| F_{2*'} \|_{L^{\frac{2d}{d+2}}(\Omega)})
		\right). 
		\end{multline*}
		
		On the other hand, the estimates \eqref{Carleman-Strip-w-hf-2-Improved-Preliminary} give
		\begin{align*}
		& \tau^{\frac{3}{4}+\frac{1}{2d}}\left( \tau \| w_{hf,\tau} \|_{L^2(E)} +  \|  \nabla' w_{hf,\tau} \|_{L^2(E)}\right)
		\leq
		\tau^{\frac{3}{4}+\frac{1}{2d}}\left( \tau \| w_{hf,\tau} \|_{L^2(\Omega)} +  \|  \nabla' w_{hf,\tau} \|_{L^2(\Omega)}\right)
		\\
		& \qquad \leq
		C\left(
		\tau^{-\frac{1}{4}+\frac{1}{2d}}\| f_2 \|_{L^2(\Omega)} 
		+
		\tau^{\frac{3}{4}+\frac{1}{2d}}  \| F_2 + F_{2*'} \|_{L^2(\Omega)}+ 
		\tau^{\frac{3}{4}+\frac{1}{2d}}\| f_{2*'} \|_{L^{\frac{2d}{d+2}}(\Omega)} \right).
		\end{align*}	
		Finally, based on Propositions \ref{Prop-OpNorm-Est-K0}--\ref{Prop-OpNorm-Est-R-tau}, we can derive the following estimate for the term $\partial_1 w$:
		\begin{align*}
		\tau^{\frac{3}{4}+\frac{1}{2d}} \| \partial_1 w \|_{L^2(E)} 
		& \leq 
		\tau^{\frac{3}{4}+\frac{1}{2d}}  \| \partial_1 w \|_{L^2(\Omega)}
		\\ 
		& \leq 
		C \left( \tau^{-\frac{1}{4}+\frac{1}{2d}} \| f_2 \|_{L^2(\Omega)} +  \tau^{\frac{3}{4}+\frac{1}{2d}} \| F_2 +F_{2*'}\|_{L^2(\Omega)} 
		+ 
		\tau^{\frac{3}{4}+\frac{1}{2d}} \| f_{2*'}\|_{L^{\frac{2d}{d+2}}(\Omega)} \right) .	 
		\end{align*}
		By combining the last three estimates, we conclude 
		\begin{multline*}
		\tau^{\frac{3}{4}+\frac{1}{2d}}  \min\left\{\frac{1}{\tau |E|^{\frac{1}{d}}} ,1\right\}
		\left( \tau \| w \|_{L^2(E)} + \| \nabla w \|_{L^2(E)}\right) 
		\\
		\leq
		C\left(
		\| f_2 \|_{L^2(\Omega)} 
		+
		\tau \| F_2 \|_{L^2(\Omega)}
		+ 
		\tau^{\frac{3}{4}+\frac{1}{2d}} \left(
		\| f_{2*'} \|_{L^{\frac{2d}{d+2}}(\Omega)}  + \tau \| F_{2*'} \|_{L^{\frac{2d}{d+2}}(\Omega)} 
		\right)			  
		+ 
		\tau^{\frac{3}{4}+\frac{1}{2d}} \| F_{2*'} \|_{L^2(\Omega)} 
		\right),
		\end{multline*}
		and the proof of Theorem \ref{Thm-Carleman-Strip-2-improved} is thus completed.
	\end{proof}
	%
%
%
	%
	\subsection{Proof of Theorem \ref{Thm-Improved-Carleman-estimates}}\label{Subsec-Local+Glue}
	The goal of this section is to deduce Theorem \ref{Thm-Improved-Carleman-estimates} in the case of a general geometry from Theorem \ref{Thm-Carleman-Strip-2-improved} which was considering only the case of a strip. 
	
	In order to do so, we rely on two main steps: 
	\begin{itemize}
		\item A localization process, which allows through a suitable change of variables, to use Theorem \ref{Thm-Carleman-Strip-2-improved} to deduce a local Carleman estimate. 
		\item A gluing argument to patch these local estimates. 
	\end{itemize}
	This is the strategy used in \cite[Section 7]{Dehman-Ervedoza-Thabouti-2023}. We only sketch it below for the convenience of the reader since it does not involve any new difficulty compared to \cite{Dehman-Ervedoza-Thabouti-2023}. 
	%
	%
	
	\subsubsection{Local Carleman estimates}\label{Subsubsec-Local}

	For $\tau \geq 1$, we introduce 
	\begin{equation*}
	w = e^{ \tau \varphi}u, 
	\qquad 
	\tilde f_2 := e^{\tau \varphi} (f_2 - \tau \nabla \varphi \cdot F), 
	\qquad 
	\tilde f_{2*'} := e^{\tau \varphi} f_{2*'}, 
	\qquad
	\tilde F = \tilde F_2+\tilde F_{2*'}:= e^{\tau \varphi} F_2+e^{\tau \varphi} F_{2*'},  
	\end{equation*}
	%
	so that the function $u$ solves \eqref{Elliptic-Eq-general} if and only if $w$ solves  
	\begin{equation} 
	\label{Eq-Elliptic-w} 
	\Delta w - 2\tau \nabla \varphi \cdot \nabla w + \tau^2 |\nabla \varphi|^2 w - \tau \Delta \varphi w = \tilde f_2 + \tilde f_{2*'} + \div(\tilde F) \quad  \text{ in } \Omega, 
	\end{equation}
	We now introduce a local version of \eqref{Eq-Elliptic-w}. Namely, for $x_0 \in \overline\Omega\setminus \omega$, we introduce $\eta_{x_0}(x)$ a cut-off function defined by 
	\begin{equation}
	\label{Choice-Localization}
	\eta_{x_0}(x ) = \eta( \tau^{\frac{1}{3}} (x-x_0)), \qquad \qquad x \in \R^{d}, 
	\end{equation}
	where $\eta$ is a non-negative smooth radial function (in $\mathscr{C}^\infty_c(\R^d)$)  such that $\eta(\rho) = 1$ for $|\rho| \leq 1/2$ and vanishing outside the unit ball. We set 
	\begin{equation*}
	w_{x_0}( x) = \eta_{x_0}(x) w(x), \qquad \qquad x \in \Omega,
	\end{equation*}
	which solves 
	\begin{equation}
	\label{Eq-Elliptic-w-x0} 
	\Delta w_{x_0} - 2\tau \nabla \varphi \cdot \nabla w_{x_0} + \tau^2 |\nabla \varphi|^2 w_{x_0} =  f_{2,x_0} + f_{2*',x_0} + \div( F_{x_0})
	,   \text{ in } \Omega,  
	\end{equation}
	where 
	\begin{align} 
	& f_{2,x_0} =  \eta_{x_0} \tilde f_2 -\nabla \eta_{x_0} \cdot \tilde F_2 + \tau \Delta \varphi w_{x_0} +2  \nabla \eta_{x_0} \cdot \nabla w + \Delta \eta_{x_0} w - 2 \tau \nabla \varphi \cdot \nabla \eta_{x_0} w, \label{source-term-f-2-x-0} 
	\\
	& f_{2*',x_0} = \eta_{x_0} \tilde f_{2*'}-\nabla \eta_{x_0} \cdot \tilde F_{2*'}, 
	\qquad
	F_{x_0} = \eta_{x_0} \tilde F. \label{source-term-f-2-*'-x-0-F-x-0} 
	\end{align}
	
	Recall that $u$ is assumed to be compactly supported in some compact set $K$, such that $K \Subset \Omega$. Accordingly, there exists $\varepsilon >0$ such that $K_\varepsilon = \{x \in \R^d, \, d(x, K) \leq \varepsilon\}$ is a subset of $\Omega$. 
	
	We then derive the following lemma, whose proof is similar to the one in \cite[Lemma 7.1]{Dehman-Ervedoza-Thabouti-2023}.   
	\begin{lemma}
		\label{Lemma-Local-Carleman-within-subset-E}
		There exist constants $C>0$ and $\tau_0 \geq 1$ (depending only on $\alpha$, $\beta$, $\| \varphi \|_{C^3(\overline\Omega)}$, $K$, $\omega$ and $\Omega$) such that for all $\tau \geq \tau_0$, for all $x_0 \in K_\varepsilon \setminus \omega_0$, for all $(f_{2,x_0}, f_{2*',x_0},F_{x_0} = F_{2,x_0} + F_{2*',x_0})$ satisfying 
		\begin{equation*} 
		f_{2,x_0} \in L^2(\Omega), 
		\ \ 
		f_{2*',x_0} \in L^{\frac{2d}{d+2}}(\Omega),
		\ \ 
		F_{2,x_0} \in L^2(\Omega; \C^{d}), 
		\ \ 
		\hbox{and} 
		\ \  F_{2*',x_0} \in L^{\frac{2d}{d+2}}(\Omega; \C^{d}) \cap L^2(\Omega; \C^{d}),
		\end{equation*}
		and $w_{x_0}$ satisfying \eqref{Eq-Elliptic-w-x0} and supported in $B_{x_0}( \tau^{-\frac{1}{3}})$, we have
		\begin{multline}
		\label{Carleman-Local-w-x0-Improved-1}
		\tau^{\frac{3}{2}} \| w_{x_0}\|_{L^2(\Omega)}
		+ 
		\tau^{\frac{1}{2}} \| \nabla w_{x_0} \|_{L^2(\Omega)}
		\\
		\leq
		C\left(
		\| f_{2,x_0} \|_{L^2(\Omega)} +
		\tau \| F_{2,x_0} \|_{L^2(\Omega)}
		+ 
		\tau^{ \frac{3}{4}-\frac{1}{2d}} \left( \| f_{2*', x_0} \|_{L^{\frac{2d}{d+2}}(\Omega)}+ \tau \| F_{2*', x_0} \|_{L^{\frac{2d}{d+2}}(\Omega)}\right) 
		+ 
		\tau^{\frac{1}{2}}  \| F_{2*', x_0} \|_{L^2(\Omega)}  
		\right),
		\end{multline}
		and, for all measurable sets $E$ of $\Omega$,
		\begin{multline}
		\label{Carleman-Local-w-x0-Improved-2} 
		\tau^{\frac{3}{4}+\frac{1}{2d}} \| w_{x_0} \|_{L^{\frac{2d}{d-2}} (\Omega)}
		+ \tau^{\frac{3}{4}+\frac{1}{2d}}  \min\left\{\frac{1}{\tau |E|^{\frac{1}{d}}} ,1\right\}
		\left( \tau \| w_{x_0} \|_{L^2(E)} + \| \nabla w_{x_0} \|_{L^2(E)}\right) 
		+
		\tau^{\frac{3}{2}} \| w_{x_0}\|_{L^2(\Omega)}
		+ 
		\tau^{\frac{1}{2}} \| \nabla w_{x_0} \|_{L^2(\Omega)}
		\\
		\leq 
		C \left( 
		\|  f_{2,x_0} \|_{L^2(\Omega)}
		+
		\tau 	\|  F_{2,x_0} \|_{L^2(\Omega)}
		+\tau^{\frac{3}{4}+\frac{1}{2d}} \left(
		\| f_{2*',x_0} \|_{L^{\frac{2d}{d+2}}(\Omega)}  + \tau \| F_{2*',x_0} \|_{L^{\frac{2d}{d+2}}(\Omega)} + \| F_{2*',x_0} \|_{L^2(\Omega)}
		\right)
		\right).
		\end{multline}
	\end{lemma}
	
	\begin{proof}[Sketch of the proof]
		The main idea is to build a suitable change of coordinates which allows to rewrite the equation \eqref{Eq-Elliptic-w-x0} under the form \eqref{Elliptic-Strip-w}, up to some lower order terms which can be handled using the localization properties of $w_{x_0}$. 
		
		Namely, let $x_0 \in K_\varepsilon \setminus \omega_0$, and introduce $L_1 \in \R^d$ and $A_1 \in \R^{d\times d}$ as follows:
		\begin{equation*}
		L_1 = \nabla \varphi(x_0) \in \R^d, \quad A_1 =  {\rm Hess\,} \varphi(x_0)  \in \R^{d\times d}.
		\end{equation*}

		The bilinear form 
		$$
		\xi \in \R^d \mapsto ({\rm Hess\,} \varphi(x_0)) \xi \cdot \xi 
		$$		
		is symmetric on $\R^d$ and on $\text{Span\,} \{L_1\}^\perp$. Accordingly, there exists a family of orthogonal vectors $(L_j)_{j \in \{2, \cdots, d\} } $ of $\text{Span\,} \{L_1\}^\perp$ which diagonalizes this form, that we normalize so that for all $j \in \{2, \cdots, d\}$, $|L_j | = |L_1|$. Since the family $(L_j)_{j \in \{2, \cdots, d\} } $ of $\text{Span\,} \{L_1\}^\perp$ diagonalizes the form $\xi \mapsto  ({\rm Hess\,} \varphi(x_0)) \xi \cdot \xi$ in $\text{Span\,} \{L_1\}^\perp$, for all $j \in \{2, \cdots, d\}$, there exist $\alpha_j $ and $\mu_j$ in $\R$ such that 
		$$
		({\rm Hess\,} \varphi(x_0))  L_j = \mu_j L_j + \alpha_j L_1, \qquad \qquad  j \in \{2, \cdots, d\}.
		$$
		Note that by symmetry of ${\rm Hess\,} \varphi(x_0)$, we then necessarily have 
		$$
		({\rm Hess\,} \varphi(x_0))  L_1 = \mu_1 L_1 + \sum_{ k \geq 2} \alpha_k L_k, 
		$$
		where 
		$$
		\mu_1 = \frac{1}{|L_1|^2}  ({\rm Hess\,} \varphi(x_0)) L_1 \cdot  L_1
		= 
		\frac{1}{|\nabla \varphi(x_0)|^2} ({\rm Hess\,} \varphi(x_0)) \nabla \varphi(x_0) \cdot  \nabla \varphi(x_0).
		$$
		For $j \in \{2, \cdots, d\}$, we then introduce the self-adjoint matrix $A_j \in \R^{d \times d}$ defined by 
		\begin{equation}
		\label{Choices-A-j}
		\left\{
		\begin{array}{ll}
		& \ds A_j L_1 = - \alpha_j L_1 - \mu_j L_j  ,
		\\
		& \ds A_j L_k =  \alpha_k L_j - \alpha_j L_k, \quad \text{ if } k \in \{2, \cdots, d\}\setminus \{ j\},
		\\ 
		& \ds A_j L_j   = - \mu_j L_1 + \sum_{k \geq 2} \alpha_k L_k.  
		\end{array}
		\right.
		\end{equation}
		(It is easy to check that each matrix $A_j$ defined that way is indeed symmetric.)

		We then define the following change of coordinates for $x$ in a neighbourhood of $x_0$:
		\begin{align*}
		& y_1(x)  =  \varphi(x)- \varphi(x_0), 
		\\
		&    
		y_j(x) =    L_j \cdot  (x-x_0)  + \frac{1}{2}  A_j (x-x_0) \cdot (x-x_0)\qquad  \text{ for } \, j \in \{2, \cdots, d\}. 
		\end{align*}
		
		By construction, there exists a neighbourhood, whose size depends on the $C^2$ norm of $\varphi$ only, such that $x \mapsto y(x)$ is a local diffeomorphism between a neighbourhood $\mathcal{V}$ of $x_0$ in $\overline\Omega\setminus \omega$ and a neighbourhood of $0$, that we call $\Omega_y$.
		
		For $\tau$ large enough, we can ensure that the ball of center $x_0$ and radius $\tau^{-\frac{1}{3}}$, when intersected with $\overline\Omega$, is included in a set on which $x \mapsto y(x)$ is a diffeomorphism, and its image is included in a ball $B_0( C \tau^{-\frac{1}{3}})$.

		Therefore, for $w_{x_0}$ solving \eqref{Eq-Elliptic-w-x0}, 
		we set 
		\begin{equation*}
		\check w(y) = w_{x_0}(x) \quad \text{ for } \quad y = y(x),
		\end{equation*}
		Tedious computations, detailed in \cite[Section 7.2]{Dehman-Ervedoza-Thabouti-2023}, show that 
		$\check w$ then satisfies
		\begin{equation*}
		\label{Elliptic-Strip-check w}
		\Delta_y \check w-y_{1} \sum_{j=2}^{d} \lambda_{j} \partial_{y_j}^{2}  \check w
		- 2 \tau \partial_{y_1} \check w + \tau^2 \check w
		= \check f_2 + \check f_{2*'} + 	\hbox{div}_y\, \check F \qquad \text { in } (Y_0, Y_1) \times \R^{d-1},  
		\end{equation*}
		where the coefficients $(\lambda_j)_{j \in \{2, \cdots, d\}}$ are given by 
		\begin{align}
		\lambda_j &= \frac{2}{|L_1|^2} \left( A_1 L_1 \cdot L_1 + A_1 L_j \cdot L_j \right)
		\notag
		\\
		\label{Def-lambda-j}
		&= \frac{2}{| \nabla \varphi(x_0)|^2}  \left( 
		({\rm Hess\,} \varphi(x_0)) \nabla \varphi(x_0) \cdot \nabla \varphi(x_0) +   ({\rm Hess\,} \varphi(x_0))  L_j  \cdot L_j \right), 
		\end{align}
		the source terms are 
		\begin{align*}
		&
		\check f_{2}(y ) =  \frac{1}{|\nabla \varphi(x(y))|^2} f_{2,x_0}(x(y)) - \sum_{j,k} \partial_{y_j} \rho_{k,j} F_{x_0,2,k}(x(y))   + \check f_{2,a}(y) + \check f_{2,b}(y) + \check f_{2,c} (y),
		\\
		&\check f_{2*'}(y) =  \frac{1}{|\nabla \varphi(x(y))|^2}  f_{2*',x_0}(x(y)) - \sum_{j,k} \partial_{y_j} \rho_{k,j} F_{x_0,2*',k}(x(y)) , 
		\\
		& \check F_j (y) = \underbrace{\sum_{k = 1}^d \rho_{k,j}(y) F_{x_0,2,k}(x(y))+ \check F_{j,a}(y)}_{=:\check F_{j,2}}+ \underbrace{\sum_{k = 1}^d \rho_{k,j}(y) F_{x_0,2*',k}(x(y))}_{=:\check F_{j,2*'}}, \qquad \qquad j \in \{1, \cdots, d\}, 
		\end{align*}
		in which $\rho$ is defined as 
		$$
		\rho_{j,k}(y) = \frac{\partial_{x_k} y_j(x(y))}{|\nabla \varphi(x(y))|^2}, 
		$$
		and $ \check f_{2,a}$, $\check f_{2,b}$,  $\check f_{2,c}$ and $\check F_a$ satisfy, due to the localization of $w$ in $B_{x_0}(\tau^{- \frac{1}{3}})$ (equivalently, of $\check w$ in  $B_{0}(C\tau^{- \frac{1}{3}})$),
		\begin{align*}
		& \| \check f_{2,a} \|_{L^2(\Omega_y)} \leq C \tau^{-\frac{1}{3}} \| \nabla_y \check w\|_{L^2(\Omega_y)}, 
		\quad
		\| \check f_{2,b} \|_{L^2(\Omega_y)} \leq C \tau^{\frac{1}{3}} \| \nabla_y \check w\|_{L^2(\Omega_y)}, 
		\quad
		\| \check f_{2,c} \|_{L^2(\Omega_y)} \leq C \| \nabla_y \check w\|_{L^2(\Omega_y)},
		\\
		&
		\| \check F_a \|_{L^2(\Omega_y)} \leq C \tau^{-\frac{2}{3}} \| \nabla_y \check w\|_{L^2(\Omega_y)}. 
		\end{align*}

		Now, the condition \eqref{CarlemanWeight-Cond2} implies that all the $\lambda_j$ in \eqref{Def-lambda-j} are positive, i.e. that the condition \eqref{Pseudo-Convexity-Strip} is satisfied. Accordingly, the Carleman estimates in Theorem \ref{Thm-Carleman-Strip-2-improved} apply, and we can readily deduce the estimates in Lemma \ref{Lemma-Local-Carleman-within-subset-E}. Let us focus for instance on the proof of the Carleman estimate \eqref{Carleman-Local-w-x0-Improved-2} (the proof of the Carleman estimate \eqref{Carleman-Local-w-x0-Improved-1} is completely similar and in fact easier and left to the reader).
		
		For $\tau \geq \tau_0$ and a measurable set $E_y$ of $\Omega_y $, we have from \eqref{Carleman-Strip-w-1-Improved} and \eqref{Carleman-Strip-w-2-Improved} that 
		\begin{multline*} 
		\tau^{\frac{3}{4}+\frac{1}{2d}} \|\check  w\|_{L^{\frac{2d}{d-2}}(\Omega_y)} 
		+ \tau^{\frac{3}{4}+\frac{1}{2d}}  \min\left\{\frac{1}{\tau |E_y|^{\frac{1}{d}}} ,1\right\}
		\left( \tau \| \check w  \|_{L^2(E_y)} + \| \nabla \check w   \|_{L^2(E_y)}\right)
		+
		\tau^{\frac{3}{2}} \|  \check w \|_{L^2(\Omega_y)}
		+
		\tau^{\frac{1}{2}} \| \nabla \check w \|_{L^2(\Omega_y)}
		\\
		\leq
		C\left(
		\| \check f_2 \|_{L^2(\Omega_y)} 
		+  
		\tau \| \check F_2 \|_{L^2(\Omega_y)}
		+\tau^{\frac{3}{4}+\frac{1}{2d}} \left(
		\| \check f_{2*'} \|_{L^{\frac{2d}{d+2}}(\Omega_y)}  + \tau \| \check F_{2*'} \|_{L^{\frac{2d}{d+2}}(\Omega_y)} +  \| \check F_{2*'} \|_{L^2(\Omega_y)} \right)  
		\right).
		\end{multline*} 
		We then simply remark that, from the expression of $\check f_2, \check f_{2*'}, $ and $ \check F= \check F_2+\check F_{2*'},$ 
		\begin{multline*}
		\| \check f_2 \|_{L^2(\Omega_y)} 
		+  
		\tau \| \check F_2 \|_{L^2(\Omega_y)}
		+\tau^{\frac{3}{4}+\frac{1}{2d}} \left(
		\| \check f_{2*'} \|_{L^{\frac{2d}{d+2}}(\Omega_y)}  + \tau \| \check F_{2*'} \|_{L^{\frac{2d}{d+2}}(\Omega_y)} +  \| \check F_{2*'} \|_{L^2(\Omega_y)} \right) 
		\\
		\leq 
		C 
		\left( \|   f_{2,x_0} \|_{L^2(\Omega)} 
		+  
		\tau \| F_{2,x_0} \|_{L^2(\Omega)}
		+\tau^{\frac{3}{4}+\frac{1}{2d}} \left(
		\|  f_{2*',x_0} \|_{L^{\frac{2d}{d+2}}(\Omega)}  + \tau \|   F_{2*',x_0} \|_{L^{\frac{2d}{d+2}}(\Omega)} +  \|   F_{2*',x_0} \|_{L^2(\Omega)} \right)  
		\right. 
		\\
		\left. 
		+
		\tau^{\frac{1}{3}} \| \nabla_y \check w\|_{L^2(\Omega_y)}\right), 
		\end{multline*}
		Accordingly, taking $\tau_0\geq 1$ larger if necessary, we get for all $\tau \geq \tau_0$, 
		\begin{multline*} 
		\tau^{\frac{3}{4}+\frac{1}{2d}} \|\check w\|_{L^{\frac{2d}{d-2}}(\Omega_y)} 
		+ \tau^{\frac{3}{4}+\frac{1}{2d}}  \min\left\{\frac{1}{\tau |E_y|^{\frac{1}{d}}} ,1\right\}
		\left( \tau \| \check w  \|_{L^2(E_y)} + \| \nabla \check w   \|_{L^2(E_y)}\right)
		+
		\tau^{\frac{3}{2}} \|  \check w \|_{L^2(\Omega_y)}
		+
		\tau^{\frac{1}{2}} \| \nabla \check w \|_{L^2(\Omega_y)} 
		\\
		\leq C 
		\left( \|   f_{2,x_0} \|_{L^2(\Omega)} 
		+  
		\tau \| F_{2,x_0} \|_{L^2(\Omega)}
		+\tau^{\frac{3}{4}+\frac{1}{2d}} \left(
		\|  f_{2*',x_0} \|_{L^{\frac{2d}{d+2}}(\Omega)}  + \tau \|   F_{2*',x_0} \|_{L^{\frac{2d}{d+2}}(\Omega)} +  \|   F_{2*',x_0} \|_{L^2(\Omega)} \right)  
		\right). 
		\end{multline*}
		Undoing the change of variables on the left-hand side, we easily deduce the estimates \eqref{Carleman-Local-w-x0-Improved-2} for $\tau \geq \tau_0$ and a measurable set $E$ of $\Omega$.
		
		The fact that the constants above do not depend on $x_0 \in K_\varepsilon \setminus \omega_0$ can be tracked in the above proof: it comes from uniformity properties of the diffeomorphism $x \mapsto y$, and relies heavily on the uniform bounds \eqref{CarlemanWeight-Cond1}--\eqref{CarlemanWeight-Cond2}, on the fact that $\varphi \in C^3(\overline\Omega)$, and that the constants in Theorem \ref{Thm-Carleman-Strip-2-improved} depend only on $c_0$, $m_*$ and $M_*$  in \eqref{Coercivity}, and \eqref{Pseudo-Convexity-Strip} for $X_0 < 0 < X_1$ with $|X_0|, |X_1| \leq 1$.  This ends the proof of Lemma \ref{Lemma-Local-Carleman-within-subset-E}.
	\end{proof}
	%
	%
	%
	%
	\subsubsection{A gluing argument}
	
	Here again, the proof closely follows the one in \cite[Section 7.3]{Dehman-Ervedoza-Thabouti-2023}, and we focus on the proof of the estimate \eqref{Carleman-General-2-Improved}, as the estimate \eqref{Carleman-General-1-Improved} can be done similarly and is thus left to the reader. 
	
	We start from \eqref{Carleman-Local-w-x0-Improved-2}:  There exist constants $C>0$ and $\tau_0 \geq 1$ such that for all $x_0 \in K_\varepsilon \setminus \omega_0$ and $\tau \geq \tau_0$ and $w_{x_0}$ solution of \eqref{Eq-Elliptic-w-x0},
	\begin{multline} 
	\label{First-Step-Gluing}
	\tau^{\frac{3}{2}+\frac{1}{d}} \| w_{x_0} \|_{L^{\frac{2d}{d-2}} (\Omega)}^2
	+ \tau^{\frac{3}{2}+\frac{1}{d}}  \min\left\{\frac{1}{\tau |E|^{\frac{1}{d}}} ,1\right\}^2
	\left( \tau^2 \| w_{x_0} \|_{L^2(E)}^2 + \| \nabla w_{x_0} \|_{L^2(E)}^2 \right) 
	+
	\tau^{3} \| w_{x_0}\|_{L^2(\Omega)}^2
	+ 
	\tau \| \nabla w_{x_0} \|_{L^2(\Omega)}^2
	\\
	\leq
	C\left(
	\| f_{2,x_0} \|^2_{L^2(\Omega)}
	+
	\tau ^2\| F_{2,x_0} \|^2_{L^2(\Omega)}
	+ 
	\tau^{\frac{3}{2}+\frac{1}{d}} \left(
	\| f_{2*',x_0} \|^2_{L^{\frac{2d}{d+2}}(\Omega)}  + \tau^2 \| F_{2*',x_0} \|^2_{L^{\frac{2d}{d+2}}(\Omega)} +  \| F_{2*',x_0} \|^2_{L^2(\Omega)} \right)  
	\right).
	\end{multline}
	Using the explicit expressions of the source terms \eqref{source-term-f-2-x-0} and \eqref{source-term-f-2-*'-x-0-F-x-0}, we obtain 
	\begin{multline*}
	\text{Right-Hand Side of }\eqref{First-Step-Gluing}  
	\leq 
	C
	\left(\|\eta_{x_0} \tilde f_{2} \|_{L^2(\Omega)} ^2
	+ \tau^2  \|\eta_{x_0} \tilde F_{2} \|_{L^2(\Omega)} ^2
	\right. 
	\\ 
	\left.
	+
	\tau^{\frac{3}{2}+\frac{1}{d}}\left( \| \eta_{x_0} \tilde f_{2*'} \|_{L^{\frac{2d}{d+2}}(\Omega)}^2 
	+
	\tau^2 \| \eta_{x_0} \tilde F_{2*'} \|_{L^\frac{2d}{d+2}(\Omega)}^2+\| \eta_{x_0} \tilde F_{2*'} \|_{L^2(\Omega)}^2\right)
	\right)
	\\
	+
	C
	\left(\| \nabla \eta_{x_0} \cdot \tilde F_2 \|_{L^2(\Omega)}^2
	+
	\tau^2 \| w_{x_0} \|_{L^2(\Omega)}^2
	+ 
	\| \nabla \eta_{x_0} \cdot \nabla w \|_{L^2(\Omega)}^2
	+ 
	\| \Delta \eta_{x_0}  w \|_{L^2(\Omega)}^2
	+ 
	\tau^2 \| |\nabla \eta_{x_0}| w \|_{L^2(\Omega)}^2  
	\right.
	\\ 
	\left. 
	+ \tau ^{\frac{3}{2}+\frac{1}{d}}\norm{\nabla\eta_{x_0} \tilde F_{2*'}}_{L^{\frac{2d}{d+2}}(\Omega)}^2 
	\right).
	\end{multline*}
	By taking $\tau_0\geq 1$ larger if necessary (which can be done uniformly in $x_0 \in K_\varepsilon \setminus\omega_0$), we can absorb the term $\tau^2 \| w_{x_0} \|_{L^2(\Omega)}^2$ by the left hand side of \eqref{First-Step-Gluing}. Then, by integrating in $x_0$ on $K_\varepsilon \setminus \omega_0$, using Fubini's identity for the Hilbertian norms, we get 
	\begin{multline*}
	\tau^{\frac{3}{2}+\frac{1}{d}}  \left\| \| \eta_{x_0}  w \|_{L_x^{\frac{2d}{d-2}}(\Omega)}\right\|_{L^2_{x_0}(K_\varepsilon \setminus \omega_0)}^2 
	+
	\tau^{\frac{3}{2}+\frac{1}{d}}  \min\left\{\frac{1}{\tau^2 |E|^{\frac{2}{d}}} ,1\right\} \left(	\tau^2 \| \rho_0^{\frac{1}{2}} w  \|_{L^2(E)} ^2
	+ 
	\| \rho_0^{\frac{1}{2}} \nabla w  \|_{L^2(E)} ^2 \right)		
	\\ 
	+\tau^3  \| \rho_0^{\frac{1}{2}} w  \|_{L^2(\Omega)} ^2
	+ 
	\tau\| \rho_0^{\frac{1}{2}} \nabla w  \|_{L^2(\Omega)} ^2 
	\\
	\leq
	C 
	\left(
	\| \rho_0^{\frac{1}{2}} \tilde f_2\|^2_{L^2(\Omega)}
	+
	\| (\tau^2 \rho_0 + \rho_{r,1})^{\frac{1}{2}} \tilde F_2 \|^2_{L^2(\Omega)} 
	+ \tau ^{\frac{3}{2}+ \frac{1}{d}} 
	\left\|  \| \eta_{x_0} \tilde F_{2*'} \|_{L^{\frac{2d}{d+2}}(\Omega)}\right\|_{L^2_{x_0}(K_\varepsilon \setminus \omega_0)}^2
	\right)
	\\ 
	+
	C 
	\tau^{\frac{3}{2}+\frac{1}{d}}\left( 
	\left\|   \| \eta_{x_0} \tilde f_{2*'} \|_{L^{\frac{2d}{d+2}}_x(\Omega)}\right\|_{L^2_{x_0}(K_\varepsilon \setminus \omega_0)}^2
	+ \tau^2
	\left\|   \| \eta_{x_0} \tilde F_{2*'} \|_{L^{\frac{2d}{d+2}}_x(\Omega)}\right\|_{L^2_{x_0}(K_\varepsilon \setminus \omega_0)}^2
	+
	\| \rho_{0}^{\frac{1}{2}} \tilde F_{2*'} \|_{L^{2}(\Omega)}^2
	\right)
	\\
	\qquad \qquad \quad
	+
	C 
	\left( 
	\| (\rho_{r,2} + \tau^2 \rho_{r,1} ) ^{\frac{1}{2}} w \|^2 _{L^2(\Omega)}
	+ 
	\| \rho_{r,1}^{\frac{1}{2}} \nabla w\|^2_{L^{2}(\Omega)}
	\right), 
	\end{multline*}
	where the weights $\rho_0$, $\rho_{r,i}$ are defined as follows:
	\begin{equation*}
	\rho_0(x)  = \int_{K_\varepsilon \setminus \omega_0} |\eta_{x_0}(x)|^2 \, dx_0, \qquad 
	\rho_{r,1}(x)  = \int_{K_\varepsilon \setminus \omega_0} |\nabla \eta_{x_0}(x)|^2 \, dx_0,
	\qquad
	\rho_{r,2}(x)  = \int_{K_\varepsilon \setminus \omega_0} |\Delta \eta_{x_0}(x)|^2 \, dx_0.
	\end{equation*}
	Minkowski’s integral inequality (\cite[p.271]{stein2016singular}) for the non-Hilbertian norms then gives:
	\begin{multline*}
	\tau^{\frac{3}{2}+\frac{1}{d}}  \|  \rho_0^{\frac{1}{2}} w \|_{L^{\frac{2d}{d-2}}(\Omega)}^2  
	+
	\tau^{\frac{3}{2}+\frac{1}{d}}  \min\left\{\frac{1}{\tau^2 |E|^{\frac{2}{d}}} ,1\right\} \left(	\tau^2 \| \rho_0^{\frac{1}{2}} w  \|_{L^2(E)} ^2
	+ 
	\| \rho_0^{\frac{1}{2}} \nabla w  \|_{L^2(E)} ^2 \right)		
	\\ 
	+\tau^3  \| \rho_0^{\frac{1}{2}} w  \|_{L^2(\Omega)} ^2
	+ 
	\tau\| \rho_0^{\frac{1}{2}} \nabla w  \|_{L^2(\Omega)} ^2 
	\\
	\leq
	C 
	\left(
	\| \rho_0^{\frac{1}{2}} \tilde f_2\|^2_{L^2(\Omega)}
	+
	\| (\tau^2 \rho_0 + \rho_{r,1})^{\frac{1}{2}} \tilde F_2 \|^2_{L^2(\Omega)} + \tau ^{\frac{3}{2}+ \frac{1}{d}}  \| \rho_{r,1}^{\frac{1}{2}} \tilde F_{2*'} \|_{L^{\frac{2d}{d+2}}(\Omega)}^2
	\right)
	\\ 
	+
	C 
	\tau^{\frac{3}{2}+\frac{1}{d}}\left( 
	\| \rho_{0}^{\frac{1}{2}} \tilde f_{2*'} \|_{L^{\frac{2d}{d+2}}(\Omega)}^2 
	+ \tau^2
	\| \rho_{0}^{\frac{1}{2}} \tilde F_{2*'} \|_{L^{\frac{2d}{d+2}}(\Omega)}^2
	+
	\| \rho_{0}^{\frac{1}{2}} \tilde F_{2*'} \|_{L^{2}(\Omega)}^2
	\right)
	\\
	\qquad \qquad \quad
	+
	C 
	\left( 
	\| (\rho_{r,2} + \tau^2 \rho_{r,1} ) ^{\frac{1}{2}} w \|^2 _{L^2(\Omega)}
	+ 
	\| \rho_{r,1}^{\frac{1}{2}} \nabla w\|^2_{L^{2}(\Omega)}
	\right), 
	\end{multline*}
	It is  easy to check from the choice \eqref{Choice-Localization} that, for $\tau$ sufficiently large, 
	\begin{align*}
	&  \rho_0(x)  = \tau^{-\frac{d}{3}} \norm{\eta}_{L^2}^2, \quad \forall x \in  
	K \setminus \omega ,
	\qquad \qquad
	\rho_{0}(x) \leq C \tau^{- \frac{d}{3}},  \quad \forall x \in \Omega, 
	\\
	&
	\rho_{r,1}(x) \leq C \tau^{\frac{2}{3} - \frac{d}{3}}, \quad  \forall x \in \Omega,	 
	\qquad \text{ and } \hspace{6ex}
	\rho_{r,2}(x)\leq C \tau^{\frac{4}{3} - \frac{d}{3}},  \hspace{5mm} \forall x \in \Omega. 
	\end{align*}
	Thus, for $\tau$ large enough, 
	\begin{multline*}
	\tau^{\frac{3}{2}+\frac{1}{d}}  \|    w \|_{L^{\frac{2d}{d-2}}(\Omega \setminus \omega)}^2  
	+
	\tau^{\frac{3}{2}+\frac{1}{d}}  \min\left\{\frac{1}{\tau^2 |E|^{\frac{2}{d}}} ,1\right\} \left(	\tau^2 \|  w  \|_{L^2(E\cap (K\setminus \omega))} ^2
	+ 
	\|  \nabla w  \|_{L^2(E\cap (K\setminus \omega))} ^2 \right)
	\\ 
	+\tau^3  \|   w  \|_{L^2(\Omega \setminus \omega)} ^2
	+ 
	\tau\|   \nabla w  \|_{L^2(\Omega \setminus \omega)} ^2 
	\\
	\leq
	C 
	\left(
	\| \tilde f_2\|^2_{L^2(\Omega)}
	+
	\tau^{2}\|   \tilde F_2 \|^2_{L^2(\Omega)} 
	+
	\tau^{\frac{3}{2}+\frac{1}{d}}\left( 
	\|   \tilde f_{2*'} \|_{L^{\frac{2d}{d+2}}(\Omega)}^2 
	+ \tau^2
	\|   \tilde F_{2*'} \|_{L^{\frac{2d}{d+2}}(\Omega)}^2
	+
	\|   \tilde F_{2*'} \|_{L^{2}(\Omega)}^2
	\right)  
	\right) 
	\\
	\qquad \qquad \quad
	+
	C 
	\left( \tau^{\frac{8}{3}}
	\|    w \|^2 _{L^2(\omega)}
	+ \tau^{\frac{2}{3}}
	\|   \nabla w\|^2_{L^{2}(\omega)}
	\right), 
	\end{multline*}
	We then add 
	\begin{multline*}
	\tau^{\frac{3}{2}+\frac{1}{d}}  \|    w \|_{L^{\frac{2d}{d-2}}( \omega)}^2   
	+
	\tau^{\frac{3}{2}+\frac{1}{d}}  \min\left\{\frac{1}{\tau^2 |E|^{\frac{2}{d}}} ,1\right\} \left(	\tau^2 \|  w  \|_{L^2( \omega)} ^2
	+ 
	\|  \nabla w  \|_{L^2( \omega)} ^2 \right)
	+
	\tau^3  \|   w  \|_{L^2( \omega)} ^2
	+ 
	\tau\|   \nabla w  \|_{L^2( \omega)} ^2 
	\end{multline*}
	to both sides of the previous estimate and get
	\begin{multline*}
	\tau^{\frac{3}{2}+\frac{1}{d}}  \|    w \|_{L^{\frac{2d}{d-2}}(\Omega)}^2 
	+
	\tau^{\frac{3}{2}+\frac{1}{d}}  \min\left\{\frac{1}{\tau^2 |E|^{\frac{2}{d}}} ,1\right\} \left(	\tau^2 \|  w  \|_{L^2(E )} ^2
	+ 
	\|  \nabla w  \|_{L^2(E )} ^2 \right) 
	+\tau^3  \|   w  \|_{L^2(\Omega)} ^2
	+ 
	\tau\|   \nabla w  \|_{L^2(\Omega )} ^2 
	\\
	\leq
	C 
	\left( \| \tilde f_2\|^2_{L^2(\Omega)}
	+
	\tau^{2}\|   \tilde F_2 \|^2_{L^2(\Omega)} 
	+
	\tau^{\frac{3}{2}+\frac{1}{d}}\left( 
	\|   \tilde f_{2*'} \|_{L^{\frac{2d}{d+2}}(\Omega)}^2 
	+ \tau^2
	\|   \tilde F_{2*'} \|_{L^{\frac{2d}{d+2}}(\Omega)}^2
	+
	\|   \tilde F_{2*'} \|_{L^{2}(\Omega)}^2
	\right)  
	\right) 
	\\
	\qquad \qquad \quad
	+
	C 
	\tau^{\frac{3}{2}+\frac{1}{d}}
	\left( \tau^{2}
	\|    w \|^2 _{L^2(\omega)}
	+  
	\|   \nabla w\|^2_{L^{2}(\omega)}
	+  \|    w \|_{L^{\frac{2d}{d-2}}( \omega)}^2 \right). 
	\end{multline*}
	This concludes the proof of Theorem \ref{Thm-Improved-Carleman-estimates}.

	%
	%
	
	\section{A specific geometric setting}\label{A specific geometric setting} In this section, we will focus on a specific geometric setting involving a ball with a radius $R > 0$ (recall that $B_0(r)$ denotes the ball centred at $0$ and of radius $r$). Within this context, we aim to prove the following lemma on quantitative unique continuation, as presented in  Theorem \ref{Thm-QUCP}. 
	\begin{lemma}\label{lemma-a-specific-geomertric-setting}
		Let  $R>0$ and $d \geqslant 3$. We consider the following geometric setting (see Figure \ref{figure of the specific geometric setting}):
		\begin{equation*}
		\Omega = B_0(2R) \cap \left\{ x_1 < -\frac{R}{4}\right\}, \, \mathcal{O} = B_0\left(\frac{3R}{2}\right) \cap \left\{ x_1 < -\frac{R}{3} \right\} , \, \text{and } \omega = (B_0(2R ) \setminus B_0 (R)) \cap \left\{ x_1 < -\frac{R}{4}\right\} ,
		\end{equation*}
		
		There exist constants $C=C(R,d)>0$ and $\alpha \in (0,1)$ depending only on $R$ and $d$ so that any solution $u\in H^{1}(\Omega)$ of \eqref{Elliptic-UC} with $(V,W_1,W_2)$ as in \eqref{potential-assumptions} satisfies the quantitative unique continuation estimate \eqref{Quantitative-0} with $\gamma$ and $\delta$ as in \eqref{Conditions-gamma(q)-delta(q)}. 
	\end{lemma}
	
	\begin{figure}[!ht]
		\begin{tikzpicture}
		\filldraw[black] (0,0) circle (2pt) node[anchor=south,scale=0.9]{$0$} ;
		\draw (0,2) arc[start angle=90, end angle=270, radius=2cm];
		\draw (0,3) arc[start angle=90, end angle=270, radius=3cm];
		\draw (0,4) arc[start angle=90, end angle=270, radius=4cm];
		\draw[dotted] (0,-4.5) -- (0, 4.5);
		\begin{scope}
		\clip (-0.4,-5)--(-8,0)--(-0.4,5) -- cycle;
		\draw[pattern=north west lines, opacity=.5, pattern color=blue,start angle=0, end angle=180] (0,0) circle (4cm);
		\draw[fill=white,start angle=0, end angle=180] (0,0) circle (-2cm);
		\end{scope}
		\begin{scope}
		\clip (-0.8,-5)--(-8,0)--(-0.8,5) -- cycle;
		\draw[pattern=north west lines, opacity=.7, pattern color=red] (0,0) circle (3cm);
		\end{scope}
		\node at (-2.5,2.7) {${\color{blue} \omega}$};
		\draw[->] (-5,0)--(2,0);
		\draw[densely dotted](-0.8,-4.5)--(-0.8,4.5);
		\draw(-0.4,-4.5)--(-0.4,4.5);
		\node at (2.3,-0.4) {$ x_1$};
		\node at (0.4,2) {$ R$};
		\node at (0.4,4) {$2 R$};
		\node at (0.4,3) {$\frac{3R}{2} $};
		\node at (-1.2,4.5) {$ -\frac{R}{3} $};
		\node at (-0.3,4.5) {$ -\frac{R}{4} $};
		\node at (-1.2,0.5) {${\color{red}\mathcal{O}}$};
		\end{tikzpicture}
		\caption{The geometric setting of Lemma \ref{lemma-a-specific-geomertric-setting}: Propagation of smallness from the left to the right.}
		\label{figure of the specific geometric setting}
	\end{figure}
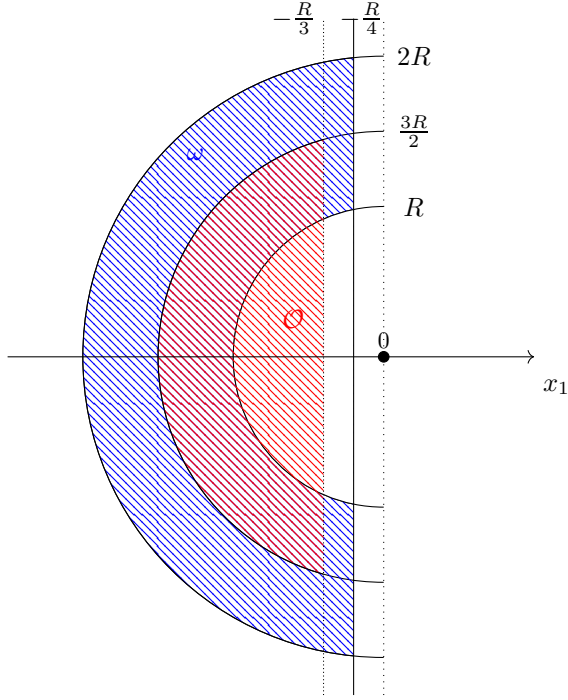

	Our goal is to explain how we can combine the Carleman estimates established in  Theorem \ref{Thm-Improved-Carleman-estimates} and Wolff's argument (Lemma \ref{Lemma-Wolff}) in order to obtain the quantitative unique continuation estimate \eqref{Quantitative-0}.
	
	A key remark is that Wolff's lemma applies for linear weight functions of the form $y \mapsto k \cdot y$ for $k \in \R^d$,  whereas our Carleman estimates are valid under appropriate subellipticity conditions on the weight function (\eqref{CarlemanWeight-Cond1}--\eqref{CarlemanWeight-Cond2}), while the parameter $\tau$ is a positive real number. 
	
	To employ both tools simultaneously,
	we construct a family of weight functions that satisfy the subellipticity conditions \eqref{CarlemanWeight-Cond1}--\eqref{CarlemanWeight-Cond2} and Wolff's argument. 
	 
	\begin{lemma} \label{lemma family of weight functions}
		Within the same setting as in Lemma \ref{lemma-a-specific-geomertric-setting}, for $k \in \R^d$, we set
		\begin{equation}
		\label{Def-varphi-k}
		\varphi_k(x) = k_1 x_1^2 + k' \cdot x', \qquad x \in \Omega.
		\end{equation}
		
		Then there exists $\epsilon>0$ such that 
		\begin{enumerate}
			\item \label{item1-lemma family of weight functions} For all  $k \in B_{e_1}(\epsilon)$, the function $\varphi_k$ satisfy \eqref{CarlemanWeight-Cond1} and \eqref{CarlemanWeight-Cond2} with some positive constants $\alpha >0$ and $\beta >0$ independent of $k \in B_{e_1}(\epsilon)$, and its $C^3$ norm on $\overline\Omega$ is bounded independently on  $k \in B_{e_1}( \epsilon)$.
			
			\item \label{item2-lemma family of weight functions} There exists $\rho >0 $ such that 
			\begin{equation} 
			\label{Weights-Comparison}
			\inf_{ k \in B_{e_1 }( \epsilon)  \ } \inf_{x \in \O} \{\varphi_k(x) \} 
			\geq 
			(1+\rho)
			\sup_{ k \in B_{e_1}(\epsilon) } 
			\left \{ 
			\sup_{x \in  \Omega \cap \{ x_1 \in (-\frac{7R}{24}, -\frac{R}{4}) \} } \{\varphi_k(x) \}
			\right\}.
			\end{equation} 
			\item \label{item3-lemma family of weight functions} Setting $\Sigma_\epsilon=\{ k \in \R^{d}\setminus\{0\} \text{ with } |k/|k| - e_1 | \leq \epsilon\}$, the family $(\varphi_k)_{k \in \Sigma_\epsilon}$ satisfies the following property: 
			If $ f$ is a positive compactly supported function in $\Omega$, we define the family $d \mu_k(x) = e^{\varphi_k(x)}f(x) dx$, then for $\mathcal{C} \subset \Sigma_\varepsilon$ there exist a family $(k_j)_{j \in J} $ of elements of $ \mathcal{C} $  and two by two disjoint sets $(E_{k_j})_{j \in J}$ included in $\Omega$ so that the measures $d \mu _{k_j}$  satisfy \eqref{concentration-property} with $T=0$ and  the family $( E_{  k_j})_{j \in J}$ satisfy  \eqref{summation-property}  with $C_W$ a positive constant depending only on $d$ and $\Omega$. 
		\end{enumerate}
	\end{lemma}
	
	\begin{proof} Items $1$ and $2$ can be checked directly using immediate computations, the fact that $\varphi_{e_1} (x) = x_1^2$ satisfies \eqref{CarlemanWeight-Cond1} and \eqref{CarlemanWeight-Cond2}, and 
		$$
		\inf_{x \in \O} \{\varphi_{e_1}(x) \} = \frac{R^2}{9} 
		>
		\sup_{x \in  \Omega \cap \{ x_1 \in (-\frac{7R}{24}, -\frac{R}{4}) \} } \{\varphi_{e_1}(x) \}
		= 
		\left(\frac{7}{24}\right)^2 R^2.
		$$

		It remains to check item \ref{item3-lemma family of weight functions}.  In order to do so, let us denote by $Y:x \mapsto y$ the diffeomorphism given by $y_1(x) = x_1^2$, and $y' = x'$ (this is clearly a diffeomorphism from $\Omega$ to $Y(\Omega)$ since $\Omega$ is away from $\{ x_1 = 0\}$), and $X$ the inverse of the map $Y$. Then remark that for all $k \in \R^d$
		$$
		\varphi_k(x) = k_1 x_1^2 + k' \cdot x' = k \cdot Y(x).
		$$
		
		Let then $ f$ be a positive compactly supported function in $\Omega$. Hence, in the new coordinates $y$, the family of measures $d \mu _{k}(x)= e^{\varphi_k(x)} f(x) dx$ becomes
		\begin{align}\label{d-mu_{k}(x)-in-coordinates-y}
		d \tilde \mu _k(y) = e^{k \cdot y}  f(X(y)) |\operatorname{Jac}(X) | dy ,    
		\end{align}
		where $\operatorname{Jac}(X)$ is the Jacobian of the map $X$. Consequently, by Lemma \ref{Lemma-Wolff}, for $  \mathcal{C} \subset \R^d$, there is a
		family $( k_j )_{j \in J}$ of elements of $\mathcal{C}$ and disjoint convex sets $( \tilde E_{  k_j})_{j \in J}$ such that the families $(d \tilde \mu_{   k_j})_{j \in J}$,   $( k_j )_{j \in J}$, and $(\tilde E_{  k_j})_{j \in J}$ satisfy \eqref{concentration-property} and \eqref{summation-property}  with $ \tilde C_W$ a positive constant depending only on $d$. Consequently, we  consider the sets $E_{k_j} = X(\tilde E_{k_j})$, which are disjoint  (not necessarily convex) and satisfy
		\begin{equation*}
		| E_{k_j} | \leqslant \| \operatorname{Jac} (X)\|_{\infty} |\tilde E_{k_j}|. 
		\end{equation*}
		Using this inequality, the summation property \eqref{summation-property} for the sets $E_{k_j}$ holds, with the constant $C_W= \|\operatorname{Jac}(X)\|^{-1}_{\infty} \tilde C_W$. On the other hand, the concentration property \eqref{concentration-property} with $T=0$  on $E_{k_j}$ for each $d   \mu_{   k_j}$ follows from the concentration property \eqref{concentration-property} for the family $d\tilde \mu_{k_j}$ on $\tilde E_{k_j}$ and the identity  \eqref{d-mu_{k}(x)-in-coordinates-y}.  
	\end{proof}
	
	As a consequence of the previous result, let us point that Theorem \ref{Thm-Improved-Carleman-estimates} holds for any $\varphi_k$ with $k \in B(e_1, \epsilon)$, with constants which are uniform with respect to $k \in B_{e_1}( \epsilon)$. Therefore, applying Theorem \ref{Thm-Improved-Carleman-estimates}, we readily deduce the following result: 
	
	\begin{lemma} \label{Lemma-Improved-Carleman-estimates-for-phi-k}
		Let $d\geqslant 3$. Let
		$$
		\Omega = B_0(2R) \cap \left\{ x_1 < -\frac{R}{4}\right\}, \, \mathcal{O} = B_0\left(\frac{3R}{2}\right) \cap \left\{ x_1 < -\frac{R}{3} \right\} , \, \text{and } \omega = (B_0(2R ) \setminus B_0 (R))  \cap \left\{ x_1 < -\frac{R}{4}\right\},
		$$
		
		Then, for all compact subset $K$ of $\Omega$ there exist $C>0$ and $\tau_0 \geq 1$ such that for all $u \in H^1(\Omega)$ satisfying $\operatorname{supp} u \subset K$ and \eqref{Elliptic-Eq-general}
		with $(f_2, f_{2*'}, F = F_2 + F_{2*'})$ as in \eqref{Reg-Source-Terms-Improved-Gal},
		we have, for all $k \in \Sigma_\epsilon$ with $|k| \geq \tau_0$, with $\varphi_k$ as in \eqref{Def-varphi-k},
		\begin{multline}
		\label{Carleman-General-1-Improved-k}	
		|k|^{\frac{3}{2}} \| e^{\varphi_k}  u \|_{L^2(\Omega)} +
		|k|^{\frac{1}{2}} \| e^{ \varphi_k}  \nabla u \|_{L^2(\Omega)} 
		\leq 
		C \left( 
		\| e^{\varphi_k}  f_2 \|_{L^2(\Omega)}
		+ 
		|k| 	\| e^{ \varphi_k}  F_2 \|_{L^2(\Omega)}+|k|^{\frac{1}{2}}   \| e^{ \varphi_k}  F_{2*'}\|_{L^2(\Omega)}
		\right.
		\\
		\left.
		+
		|k|^{\frac{3}{4}-\frac{1}{2d}}
		\left( \| e^{\varphi_k} f_{2*'} \|_{L^{\frac{2d}{d+2}} (\Omega)}
		+ 
		|k| 	\| e^{ \varphi_k}  F_{2*'} \|_{L^{\frac{2d}{d+2}}(\Omega)}
		\right)
		+ 
		|k|^{\frac{3}{2}} \| e^{\varphi_k} u\|_{H^1(\omega)}
		\right), 
		\end{multline}
		and, for all measurable sets $E$ of $\Omega$, 
		\begin{multline}
		\label{Carleman-General-2-Improved-k}		 
		|k|^{\frac{3}{4}+\frac{1}{2d}} \| e^{\varphi_k} u \|_{L^{\frac{2d}{d-2}} (\Omega)}
		+ |k|^{\frac{3}{4}+\frac{1}{2d}}  \min\left\{\frac{1}{|k| |E|^{\frac{1}{d}}} ,1\right\}
		\left( |k|\| e^{\varphi_k} u \|_{L^2(E)} + \| e^{\varphi_k} \nabla u \|_{L^2(E)}\right) 
		\\
		\leq 
		C \left( 
		\| e^{\varphi_k}  f_2 \|_{L^2(\Omega)}
		+
		|k| 	\| e^{\varphi_k}  F_2 \|_{L^2(\Omega)}
		+|k|^{\frac{3}{4}+\frac{1}{2d}} \left(
		\| e^{\varphi_k} f_{2*'} \|_{L^{\frac{2d}{d+2}}(\Omega)}  + |k| \| e^{\varphi_k}F_{2*'} \|_{L^{\frac{2d}{d+2}}(\Omega)}
		\right) 
		\right.
		\\
		\left.
		+|k|^{\frac{3}{4}+\frac{1}{2d}}  \|e^{\varphi_k} F_{2*'} \|_{L^2(\Omega)} 
		+ 
		|k|^{\frac{7}{4}+\frac{1}{2d}} \| e^{\varphi_k} u\|_{H^1(\omega)}
		\right).
		\end{multline}
	\end{lemma}

	We are now in a position to prove Lemma \ref{lemma-a-specific-geomertric-setting}.
	
	\begin{proof}[Proof of Lemma \ref{lemma-a-specific-geomertric-setting}] For sake of clarity, we divide the proof in several steps. 
		\medskip
		
		\textit{Step 1: Application of the Carleman estimates.} For $u \in H^1(\Omega)$ with 
		$$
		\Delta u = V u + W_1 \cdot \nabla u + \div(W_2 u) \quad \text{ in } \Omega,   
		$$
		we set $v = \eta u$, where $\eta $ is a smooth cut-off function that takes $1$ in $B_0(3R/2) \cap \{x_1 < - 7R/24\} $ and vanishes in a neighbourhood of $\partial\Omega$, so that we have 
		$$
		\Delta v = V v + W_1 \cdot \nabla v + \div(W_2 v) + f_\eta \quad \text{ in } \Omega,   
		$$
		where $ f_\eta$ is defined by 
		\begin{align}\label{def_f-eta}
		f_\eta=2 \nabla \eta \cdot \nabla u + \Delta \eta u - W_1 \cdot \nabla \eta u- W_2 \cdot \nabla \eta u,  
		\end{align}  
		and thus satisfies
		\begin{equation}
		\label{Support-f-eta}
		\text{Supp\,} f_\eta \subset \omega \cup (\Omega \cap \{ x_1 \in (- 7R/24, - R/4)\} ).
		\end{equation}
		
		Now, for $V \in L^{q_0}(\Omega)$, $W_1 \in L^{q_1}(\Omega; \C^d)$, and $W_2 \in L^{q_2}(\Omega; \C^d)$, with $q_0>d/2$, $q_1>d$, and  $q_2>d$, we will perform a decomposition of the form
		\begin{align*}
		& V  = V_{\frac{d}{2}} +V_{d}+  V_\infty, 
		\quad &&\hbox{ with } 
		V_{\frac{d}{2}} \in L^{\frac{d}{2}}(\Omega), 
		\
		V_{d}  \in L^{d}(\Omega),
		\
		V_{\infty} \in L^{\infty}(\Omega), 
		\\ 
		& W_1  = W_{1,d} + W_{1,\infty}, 
		\quad &&\hbox{ with } 
		W_{1,d} \in L^{d}(\Omega; \C^d), 
		\
		W_{1,\infty} \in L^{\infty}(\Omega); \C^d), 
		\\ 
		& W_2  = W_{2,d} + W_{2,\infty}, 
		\quad && \hbox{ with } 
		W_{2,d} \in L^{d}(\Omega); \C^d), 
		\
		W_{2,\infty} \in L^{\infty}(\Omega; \C^d). 
		\end{align*}
		We will explain later, in Step 4 of the proof, the precise decomposition we will choose. 

		We then apply Lemma \ref{Lemma-Improved-Carleman-estimates-for-phi-k}. The Carleman estimate \eqref{Carleman-General-1-Improved-k} with $f_{2*'} = V_{\frac{d}{2}} v + V_{d} v+ W_{1,d} \cdot \nabla v$, $f_2 =  V_{\infty} v+ W_{1,\infty} \cdot \nabla v +f_\eta$, $F_{2*'} = W_{2,d}  v$, and  $F_{2} = W_{2,\infty}  v$ yields
		that for all $k \in \Sigma_\varepsilon$ with $|k| \geq \tau_0$,
		\begin{multline*}
		|k|^{\frac{3}{2}} \| e^{\varphi_k}  v \|_{L^2(\Omega)} 
		+
		|k|^{\frac{1}{2}} \| e^{\varphi_k}  \nabla v \|_{L^2(\Omega)}	
		\leq 
		C \left( 
		\| V_\infty \|_{L^\infty(\Omega)} \| e^{\varphi_k} v \|_{L^2(\Omega)}
		+
		\| W_{1,\infty} \|_{L^\infty} \| e^{\varphi_k} \nabla v \|_{L^2(\Omega)}
		\right.
		\\
		+
		\| e^{ \varphi_k}  f_\eta \|_{L^2(\Omega)}
		+
		|k| \| W_{2,\infty} \|_{L^\infty} \| e^{\varphi_k} v \|_{L^2(\Omega)}
		+ 
		|k|^{\frac{1}{2}}\| W_{2,d} \|_{L^d(\Omega)} \| e^{ \varphi_k}   v \|_{L^{\frac{2d}{d-2}}(\Omega)} 
		\\
		+ 
		|k|^{\frac{3}{2}} \| e^{ \varphi_k} v\|_{H^1(\omega)}
		+
		\left.
		|k|^{\frac{3}{4}-\frac{1}{2d}} 
		\left(
		\| V_d \|_{L^d(\Omega)}   \| e^{\varphi_k}  v \|_{L^2(\Omega)} 
		+ 
		\| V_{\frac{d}{2}} \|_{L^{\frac{d}{2}}(\Omega)}   \| e^{\varphi_k}  v \|_{L^{\frac{2d}{d-2}}(\Omega)} 
		\right)
		\right.
		\\
		\left.
		+
		|k|^{\frac{3}{4}-\frac{1}{2d}} \left( \| e^{\varphi_k}  W_{1,d} \cdot \nabla v \|_{L^{\frac{2d}{d+2}}  (\Omega)}
		+ 
		|k|\| e^{\varphi_k}  W_{2,d}    v \|_{L^{\frac{2d}{d+2}}(\Omega)} \right)
		\right). 
		\end{multline*}
		
		Accordingly, there exists $c_0 >0$ such that if 
		\begin{equation}
		\label{Constraints-1}
		\left( \|V_\infty \|_{L^\infty(\Omega)} + |k|^{\frac{3}{4}-\frac{1}{2d}} \|V_d \|_{L^d(\Omega)}\right) \leq c_0 |k|^{\frac{3}{2}}, 
		\quad \text{ and } \quad
		\left(\| W_{1, \infty}\|_{L^\infty(\Omega)} +\| W_{2, \infty}\|_{L^\infty(\Omega) } \right) \leq c_0 |k|^{\frac{1}{2}},
		\end{equation}
		for all $k \in \Sigma_\varepsilon$ with $|k| \geq \tau_0$,
		\begin{multline}
		\label{Carleman-General-1-v-k-OurPb}	
		|k|^{\frac{3}{2}} \| e^{\varphi_k}  v \|_{L^2(\Omega)} 
		+
		|k|^{\frac{1}{2}} \| e^{\varphi_k}  \nabla v \|_{L^2(\Omega)}	
		\leq 
		C_1 \left( 
		\| e^{ \varphi_k}  f_\eta \|_{L^2(\Omega)}
		+ 
		|k|^{\frac{3}{2}} \| e^{ \varphi_k} u\|_{H^1(\omega)}
		\right.
		\\
		+ 
		\left(|k|^{\frac{1}{2}}\| W_{2,d} \|_{L^d(\Omega)} + |k|^{\frac{3}{4}-\frac{1}{2d}}   \| V_{\frac{d}{2}} \|_{L^{\frac{d}{2}}(\Omega)}  \right) \| e^{ \varphi_k}   v \|_{L^{\frac{2d}{d-2}}(\Omega)} 
		\\
		+
		\left. 
		|k|^{\frac{3}{4}-\frac{1}{2d}} 
		\left( 
		\| e^{\varphi_k}  W_{1,d} \cdot \nabla v \|_{L^{\frac{2d}{d+2}}  (\Omega)}
		+ 
		|k|\| e^{\varphi_k}  W_{2,d}    v \|_{L^{\frac{2d}{d+2}}(\Omega)} \right)
		\right).
		\end{multline}
		
		Similarly, the Carleman estimate \eqref{Carleman-General-2-Improved} with $f_{2*'} = V_{\frac{d}{2}} v + W_{1,d} \cdot \nabla v$, $f_2 =V_{\infty} v + V_{d} v + W_{1,\infty} \cdot \nabla v+f_{\eta}$,  $F_{2*'} = W_{2,d}  v$, and  $F_{2} = W_{2,\infty}   v$ yields, %
		that for all $k \in \Sigma_\varepsilon$ with $|k| \geq \tau_0$, and for all measurable set $E$;
		\begin{multline*}
		|k|^{\frac{3}{4}+\frac{1}{2d}} \| e^{ \varphi_k} v \|_{L^{\frac{2d}{d-2}} (\Omega)}
		+ 
		|k|^{\frac{3}{4}+\frac{1}{2d}}  \min\left\{\frac{1}{|k| |E|^{\frac{1}{d}}} ,1\right\}
		\left( |k| \| e^{ \varphi_k}  v \|_{L^2(E)} + \| e^{ \varphi_k}  \nabla v \|_{L^2(E)}\right) 
		\\
		\leq 
		C \left( 
		\| V_\infty\|_{L^\infty(\Omega)} \| e^{ \varphi_k} v \|_{L^{2} (\Omega)}
		+ 
		\| V_d \|_{L^d (\Omega)} \| e^{\varphi_k}v \|_{L^{\frac{2d}{d-2} (\Omega)}}
		+ 
		\| W_{1, \infty} \|_{L^\infty(\Omega)} \| e^{\varphi_k} \nabla v\|_{L^2(\Omega) }
		+ 
		\| e^{\varphi_k}  f_\eta \|_{L^2(\Omega)}
		\right.
		\\
		+
		|k|^{\frac{3}{4}+\frac{1}{2d}} 
		\left( 
		\|V_{\frac{d}{2}} \|_{L^{\frac{d}{2}}(\Omega)}  \| e^{ \varphi_k} v \|_{L^{\frac{2d}{d-2}} (\Omega)}
		+	
		\| e^{\varphi_k}  W_{1,d} \cdot \nabla v \|_{L^{\frac{2d}{d+2}}(\Omega)}   
		+ 
		|k| \| e^{\varphi_k}  W_{2,d}    v \|_{L^{\frac{2d}{d+2}}(\Omega)}  
		\right) 
		\\
		\left.
		+
		|k| \| W_{2, \infty}\|_{L^\infty(\Omega)} \| e^{\varphi_k} v\|_{L^2(\Omega)} 
		+ 
		|k|^{\frac{3}{4}+\frac{1}{2d}} \| W_{2, d}\|_{L^d(\Omega)} \| e^{|k| \varphi_k}   v \|_{L^{\frac{2d}{d-2}} (\Omega)}
		+ 
		|k|^{\frac{7}{4}+\frac{1}{2d}} \| e^{\varphi_k} u \|_{H^1(\omega)}
		\right).
		\end{multline*}
		
		Accordingly, there exists $c_1>0$ such that if 
		\begin{equation}
		\label{Constraints-2}
		\|V_{\frac{d}{2}} \|_{L^{\frac{d}{2}}(\Omega)} \leq c_1, 
		\qquad 
		\|V_d \|_{L^d(\Omega)} \leq c_1 |k|^{\frac{3}{4}+\frac{1}{2d}},
		\quad \text{ and } \quad
		\|W_{2,d} \|_{L^d(\Omega)} \leq c_1,  
		\end{equation}
		for all $k \in \Sigma_\varepsilon$ with $|k| \geq \tau_0$, and for all measurable set $E$;
		\begin{multline}
		\label{Carleman-General-2-v-k-OurPb}	 
		|k|^{\frac{3}{4}+\frac{1}{2d}} \| e^{ \varphi_k} v \|_{L^{\frac{2d}{d-2}} (\Omega)}
		+ 
		|k|^{\frac{3}{4}+\frac{1}{2d}}  \min\left\{\frac{1}{|k| |E|^{\frac{1}{d}}} ,1\right\}
		\left( |k| \| e^{ \varphi_k}  v \|_{L^2(E)} + \| e^{ \varphi_k}  \nabla v \|_{L^2(E)}\right) 
		\\
		\leq 
		C_2 \left( 
		\left(\| V_\infty\|_{L^\infty(\Omega)} +
		|k| \| W_{2, \infty}\|_{L^\infty(\Omega)} \right)
		\| e^{ \varphi_k} v \|_{L^{2} (\Omega)}
		+ 
		\| W_{1, \infty} \|_{L^\infty(\Omega)} \| e^{\varphi_k} \nabla v\|_{L^2(\Omega) }
		+ 
		\| e^{\varphi_k}  f_\eta \|_{L^2(\Omega)}
		\right.
		\\
		+
		|k|^{\frac{3}{4}+\frac{1}{2d}} 
		\left( 
		\| e^{\varphi_k}  W_{1,d} \cdot \nabla v \|_{L^{\frac{2d}{d+2}}(\Omega)}   
		+ 
		|k| \| e^{\varphi_k}  W_{2,d}    v \|_{L^{\frac{2d}{d+2}}(\Omega)}  
		\right) 
		\left.
		+ 
		|k|^{\frac{7}{4}+\frac{1}{2d}} \| e^{\varphi_k} u \|_{H^1(\omega)}
		\right).
		\end{multline}
		From now on, we will assume that conditions \eqref{Constraints-1} and \eqref{Constraints-2} are satisfied.
		\medskip
		
		\textit{Step 2: Application of Wolff's argument.} Let $n \in \R$ be larger than $\tau_0/(1- \epsilon)$, with $\epsilon$ as in Lemma \ref{lemma family of weight functions}. We set $\tilde {\mathcal{C}_n }= \{ k \in \R^d, \text{ such that } |k - n e_1| \leq \epsilon n \}$, so that $\tilde {\mathcal{C}_n }= n B_{e_1}(\epsilon)\subset \Sigma_\epsilon$. For all $k \in \tilde {\mathcal{C}_n }$, we define the measure
		\begin{align*}
		d\mu_{k} & = \left(|e^{ \varphi_k(x)}  W_{1,d}(x) \cdot \nabla v(x)|^{\frac{2d}{d+2}} + | (1+ \epsilon) n e^{ \varphi_k(x)}  W_{2,d}(x)  v(x)|^{\frac{2d}{d+2}} \right)\, dx
		\\ 
		& = e^{ \varphi_{2dk/(d+2)}(x)} \left(| W_{1,d}(x) \cdot \nabla v(x)|^{\frac{2d}{d+2}} + | (1+ \epsilon) n W_{2,d}(x)  v(x)|^{\frac{2d}{d+2}} \right)\, dx. 
		\end{align*}
		Then Lemma \ref{lemma family of weight functions} (applied to $\mathcal{C}_n = 2d\tilde {\mathcal{C}_n }/(d+2) $) implies
		the existence of a constant $C_{W} >0$ such that for all $n \in \N$, there exists a set of index $J_{n}$, a family $(k_{j,n})_{j \in J_{n}}$ of elements of $\tilde {\mathcal{C}_n }$ and a corresponding family of pairwise disjoint sets $(E_{k_{j,n}})_{j\in J_{n}}$ such that for all $j \in J_{n}$,  we have 
		\begin{multline}
		\label{Everything-in-E}
		\| |e^{\varphi_{k_{j,n}}} W_{1,d} \cdot \nabla v|^{\frac{2d}{d+2}} + | n (1+ \epsilon) e^{ \varphi_{k_{j,n}}}  W_{2,d}    v |^{\frac{2d}{d+2}}  \|^{\frac{d+2}{2d}}_{L^{1} (\Omega)}
		\\
		\leq 
		2 \| |e^{ \varphi_{k_{j,n}}(x)}  W_{1,d}(x) \cdot \nabla v(x)|^{\frac{2d}{d+2}} + |  n (1+ \epsilon) e^{  \varphi_{k_{j,n}}(x)}  W_{2,d}(x)  v(x)|^{  \frac{2d}{d+2}} \|^{\frac{d+2}{2d}}_{L^{1}(E_{k_{j,n}})}
		\end{multline}
		and 
		\begin{equation}
		\sum_{j \in J_{n}} |E_{k_{j,n}}|^{-1} \geq \frac{1}{C_{W} } \left(\frac{2d}{d+2}\right)^d n^d.
		\end{equation}
		
		Hence, we claim that if the conditions
		\begin{equation}
		\label{Smallness-W-1-2}
		\left\{
		\begin{array}{l}
		\ds  \| W_{1,d} \|_{L^d(\Omega)}^d + \left(\frac{1+\epsilon}{1-\epsilon} \right)^d \| W_{2,d} \|_{L^d(\Omega)}^d < \frac{1}{C_{W}  (16 C_2 (1+ \epsilon))^d} \left(\frac{2d}{d+2}\right)^d , 
		\\
		\ds 8 C_2  \left(  \| W_{1,d} \|_{L^d(\Omega)} + \frac{1+\epsilon}{1- \epsilon}  \| W_{2,d} \|_{L^d(\Omega)} \right) \leq 1,
		\end{array}
		\right.
		\end{equation}
		(where $C_2$ is the constant in \eqref{Carleman-General-2-v-k-OurPb}) 
		are satisfied, then for all $n  \in \N$, there exists $j_{*,n} \in J_{n}$ such that 
		\begin{equation}
		\label{Def-k-i-*-n}
		8C_2 \left( \| W_{1,d} \|_{L^d(E_{k_{j_* ,n}})} +\frac{1+ \epsilon}{1- \epsilon} \| W_{2,d} \|_{L^d(E_{k_{j_* ,n}})}  \right)\leqslant \frac{1}{|k_{j_{*,n} ,n}| |E_{k_{j_{*,n} ,n}} |^{\frac{1}{d}} }. 
		\end{equation}
		Indeed, if not, for all $j \in J_n$, we would have  
		$$|E_{k_{j,n}} |^{-1}  \leq (16 C_2 (1+ \epsilon) n)^d  
		\left( \| W_{1,d} \|_{L^d(E_{k_{j,n}})}^d+\left(\frac{1+\epsilon}{1-\epsilon} \right)^d\| W_{2,d} \|_{L^d(E_{k_{j,n}})}^d\right), 
		$$
		(where we use the elementary estimate $(a+b)^d \leq 2^d( a^d + b^d)$ for $a, b \geq 0$).
		
		By summing these estimates over $j \in J_{n}$ and taking into account that the sets $(E_{k_{j,n}})_{j \in J_n}$ are pairwise disjoint, we would get   
		$$
		\frac{1}{C_{W} }\left(\frac{2d}{d+2}\right)^d  n^d \leq \sum_{j \in J_{n}} |E_{k_{j,n}}|^{-1}
		\leq
		(16 C_2 (1+ \epsilon) {n})^d  \left(\| W_{1,d} \|_{L^d(\Omega)}^d + \left(\frac{1+\epsilon}{1-\epsilon} \right)^d\| W_{2,d} \|_{L^d(\Omega)}^d\right),
		$$
		which would contradict \eqref{Smallness-W-1-2}. 
		
		We thus assume condition \eqref{Smallness-W-1-2}. For $n \geq \tau_0/(1- \epsilon)$, we set $k_n = k_{j_{*,n}}$, where $  k_{j_{*,n}}$ is such that \eqref{Def-k-i-*-n} holds, and we set $E_n = E_{k_{j_*,n}}$. 
		We then deduce from \eqref{Carleman-General-2-v-k-OurPb} that 
		\begin{multline}
		\label{Carleman-General-2-v-Local-terms}	 
		|k_n|^{\frac{3}{4}+\frac{1}{2d}}  \min\left\{\frac{1}{|k_n| |E_n|^{\frac{1}{d}}} ,1\right\}
		\left( |k_n| \| e^{ \varphi_{k_n}}  v \|_{L^2(E_n)} + \| e^{ \varphi_{k_n}}  \nabla v \|_{L^2(E_n)}\right) 
		\\
		\leq 
		C_2 \left( 
		\left(\| V_\infty\|_{L^\infty(\Omega)} +
		|k_n| \| W_{2, \infty}\|_{L^\infty(\Omega)} \right)
		\| e^{ \varphi_{k_n}} v \|_{L^{2} (\Omega)}
		+ 
		\| W_{1, \infty} \|_{L^\infty(\Omega)} \| e^{\varphi_{k_n}} \nabla v\|_{L^2(\Omega) }
		+ 
		\| e^{\varphi_{k_n}}  f_\eta \|_{L^2(\Omega)}
		\right.
		\\
		+
		|k_n|^{\frac{3}{4}+\frac{1}{2d}} 
		\left( 
		\| e^{\varphi_{k_n}}  W_{1,d} \cdot \nabla v \|_{L^{\frac{2d}{d+2}}(\Omega)}   
		+ 
		|k_n| \| e^{\varphi_{k_n}}  W_{2,d}    v \|_{L^{\frac{2d}{d+2}}(\Omega)}  
		\right) 
		\left.
		+ 
		|k_n|^{\frac{7}{4}+\frac{1}{2d}} \| e^{\varphi_{k_n}} u \|_{H^1(\omega)}
		\right).
		\end{multline}
		Using then the classical estimates $|a|^\alpha+|b|^\alpha \leq 2 (|a| + |b|)^\alpha$ and $(|a| +|b|)^\alpha \leq |a|^\alpha + |b|^\alpha$ for $\alpha \in [0,1]$ and $a, b \in \R$, we obtain that 
		\begin{align*}
		& 
		\| e^{\varphi_{k_n}}  W_{1,d} \cdot \nabla v \|_{L^{\frac{2d}{d+2}}(\Omega)}   
		+ 
		|k_n| \| e^{\varphi_{k_n}}  W_{2,d}    v \|_{L^{\frac{2d}{d+2}}(\Omega)}  
		\\
		& \leq 2 \| e^{\frac{2d}{d+2} \varphi_{k_n}} \left( | W_{1,d} \cdot \nabla v|^{{\frac{2d}{d+2}}}
		+ 
		(n (1+ \epsilon)) |W_{2,d} v |^{\frac{2d}{d+2}} \right)
		\|_{L^{1}(\Omega)}^{\frac{d+2}{2d}} 
		\\ 
		& 
		\leq 4 \| e^{\frac{2d}{d+2} \varphi_{k_n}} \left( | W_{1,d} \cdot \nabla v|^{{\frac{2d}{d+2}}}
		+ 
		(n (1+ \epsilon)) |W_{2,d} v |^{\frac{2d}{d+2}} \right)
		\|_{L^{1}(E_n)}^{\frac{d+2}{2d}} 
		\\ 
		& 
		\leq 
		4 \| e^{\varphi_{k_n}}  W_{1,d} \cdot \nabla v \|_{L^{\frac{2d}{d+2}}(E_n)}   
		+ 
		4 n (1+\epsilon) \| e^{\varphi_{k_n}}  W_{2,d}    v \|_{L^{\frac{2d}{d+2}}(E_n)}  
		\\
		& \leq 
		4 \left(
		\|W_{1, d}\|_{L^d(E_n)} \| e^{\varphi_{k_n}} \nabla v \|_{L^2(E_n)}
		+n(1+ \epsilon) \|W_{2, d}\|_{L^d(E_n)} \| e^{\varphi_{k_n}} v \|_{L^2(E_n)}
		\right)
		\\
		& \leq 
		\frac{1}{2C_2} \min\left\{\frac{1}{|k_n| |E_n|^{\frac{1}{d}}} ,1\right\} \left( (1-\epsilon) n  \| e^{\varphi_{k_n}} v \|_{L^2(E_n)} +  \| e^{\varphi_{k_n}} \nabla v \|_{L^2(E_n)}\right) , 
		\end{align*}
		where we used \eqref{Everything-in-E}, and the fact that, from \eqref{Def-k-i-*-n}, 
		$$
		4 \| W_{1,d} \|_{L^d(E_n)} \leq \frac{1}{2C_2} \frac{1}{|k_n| |E_n|^{\frac{1}{d}}}, 
		\quad \text{ and } \quad 
		4 (1+\epsilon) \| W_{2,d} \|_{L^d(E_n)} \leq \frac{1-\epsilon}{2C_2} \frac{1}{|k_n| |E_n|^{\frac{1}{d}}} 
		$$
		and, from \eqref{Smallness-W-1-2}$_{(2)}$,
		$$
		4 \| W_{1,d} \|_{L^d(E_n)} \leq 4 \|W_{1,d}\|_{L^d(\Omega)} \leq \frac{1}{2C_2}, 
		\quad \text{ and } \quad 
		4 (1+\epsilon) \| W_{2,d} \|_{L^d(E_n)} \leq 4 (1+\epsilon) \| W_{2,d} \|_{L^d(\Omega)} \leq \frac{1-\epsilon}{2C_2}.
		$$
		
		Accordingly, from \eqref{Carleman-General-2-v-Local-terms}, we deduce that 
		\begin{multline}
		\label{Estimee-v-Local-terms-Wolff}
		|k_n|^{\frac{3}{4}+\frac{1}{2d}} 
		\left( 
		\| e^{\varphi_{k_n}}  W_{1,d} \cdot \nabla v \|_{L^{\frac{2d}{d+2}}(\Omega)}   
		+ 
		|k_n| \| e^{\varphi_{k_n}}  W_{2,d}    v \|_{L^{\frac{2d}{d+2}}(\Omega)}  
		\right) 
		\\
		\leq 
		2 C_2 
		\left( 
		\left(\| V_\infty\|_{L^\infty(\Omega)} +
		|k_n| \| W_{2, \infty}\|_{L^\infty(\Omega)} \right)
		\| e^{ \varphi_{k_n}} v \|_{L^{2} (\Omega)}
		+ 
		\| W_{1, \infty} \|_{L^\infty(\Omega)} \| e^{\varphi_{k_n}} \nabla v\|_{L^2(\Omega) }
		\right.
		\\
		\left.
		+ 
		\| e^{\varphi_{k_n}}  f_\eta \|_{L^2(\Omega)}
		+ 
		|k_n|^{\frac{7}{4}+\frac{1}{2d}} \| e^{\varphi_{k_n}} u \|_{H^1(\omega)}
		\right).
		\medskip
		\end{multline}
		
		\emph{Step 3: Combining the Carleman estimates \eqref{Carleman-General-1-v-k-OurPb} and \eqref{Carleman-General-2-v-k-OurPb}, and the estimate \eqref{Estimee-v-Local-terms-Wolff}.} Using \eqref{Carleman-General-1-v-k-OurPb}, \eqref{Carleman-General-2-v-k-OurPb} and \eqref{Estimee-v-Local-terms-Wolff}, we get that there exists a constant $C>0$ such that for all $n \geq \tau_0/(1- \epsilon)$, 
		\begin{multline}
		\label{Carleman-General-1-v-k-OurPb+Wolff-0}	
		|k_n|^{\frac{3}{2}} \| e^{\varphi_{k_n}}  v \|_{L^2(\Omega)} 
		+
		|k_n|^{\frac{1}{2}} \| e^{\varphi_{k_n}}  \nabla v \|_{L^2(\Omega)}	
		\leq 
		C \left( 
		\| e^{ \varphi_{k_n}}  f_\eta \|_{L^2(\Omega)}
		+ 
		|k_n|^{\frac{3}{2}} \| e^{ \varphi_{k_n}} u\|_{H^1(\omega)}
		\right.
		\\
		+ 
		\left(|k_n|^{\frac{1}{2}}\| W_{2,d} \|_{L^d(\Omega)} + |k_n|^{\frac{3}{4}-\frac{1}{2d}}   \| V_{\frac{d}{2}} \|_{L^{\frac{d}{2}}(\Omega)}  \right) \| e^{ \varphi_{k_n}}   v \|_{L^{\frac{2d}{d-2}}(\Omega)} 
		\\
		+
		|k_n|^{-\frac{1}{d}} 
		\left( 
		\left(\| V_\infty\|_{L^\infty(\Omega)} +
		|k_n| \| W_{2, \infty}\|_{L^\infty(\Omega)} \right)
		\| e^{ \varphi_{k_n}} v \|_{L^{2} (\Omega)}
		+ 
		\| W_{1, \infty} \|_{L^\infty(\Omega)} \| e^{\varphi_{k_n}} \nabla v\|_{L^2(\Omega) }
		\right.
		\\
		\left.
		+ 
		\| e^{\varphi_{k_n}}  f_\eta \|_{L^2(\Omega)}
		+ 
		|k_n|^{\frac{7}{4}+\frac{1}{2d}} \| e^{\varphi_{k_n}} u \|_{H^1(\omega)}
		\right),
		\end{multline}
		and
		\begin{multline}
		\label{Carleman-General-2-v-k-OurPb+Wolff}	
		|k_n|^{\frac{3}{4}+\frac{1}{2d}} 
		\| e^{\varphi_{k_n}} v \|_{L^{\frac{2d}{d+2}}(\Omega)}  
		\leq 
		C
		\Big( 
		\left(\| V_\infty\|_{L^\infty(\Omega)} +
		|k_n| \| W_{2, \infty}\|_{L^\infty(\Omega)} \right)
		\| e^{ \varphi_{k_n}} v \|_{L^{2} (\Omega)}
		\\	
		+ 
		\| W_{1, \infty} \|_{L^\infty(\Omega)} \| e^{\varphi_{k_n}} \nabla v\|_{L^2(\Omega) }
		+ 
		\| e^{\varphi_{k_n}}  f_\eta \|_{L^2(\Omega)}
		+ 
		|k_n|^{\frac{7}{4}+\frac{1}{2d}} \| e^{\varphi_{k_n}} u \|_{H^1(\omega)}
		\Big).
		\end{multline}
		
		Note in particular that, in view of the assumptions \eqref{Constraints-1}, we get from \eqref{Carleman-General-1-v-k-OurPb+Wolff-0} that there exists $\tau_1\geq \tau_0/(1-\epsilon)$ such that,  for all $n \geq \tau_1$, 
		\begin{multline}
		\label{Carleman-General-1-v-k-OurPb+Wolff}	
		|k_n|^{\frac{3}{2}} \| e^{\varphi_{k_n}}  v \|_{L^2(\Omega)} 
		+
		|k_n|^{\frac{1}{2}} \| e^{\varphi_{k_n}}  \nabla v \|_{L^2(\Omega)}	
		\leq 
		C \left( 
		\| e^{ \varphi_{k_n}}  f_\eta \|_{L^2(\Omega)}
		+ 
		|k_n|^{\frac{7}{4}+\frac{1}{2d}} \| e^{ \varphi_{k_n}} u\|_{H^1(\omega)}
		\right.
		\\
		\left.+ 
		\left(|k_n|^{\frac{1}{2}}\| W_{2,d} \|_{L^d(\Omega)} + |k_n|^{\frac{3}{4}-\frac{1}{2d}}   \| V_{\frac{d}{2}} \|_{L^{\frac{d}{2}}(\Omega)}  \right) \| e^{ \varphi_{k_n}}   v \|_{L^{\frac{2d}{d-2}}(\Omega)}\right).
		\end{multline}
		
		Thus, combining \eqref{Carleman-General-2-v-k-OurPb+Wolff} and \eqref{Carleman-General-1-v-k-OurPb+Wolff}, we obtain, for all $n \geq \tau_1$, 
		\begin{multline}
		\label{Carleman-OurPb+Wolff}	
		|k_n|^{\frac{3}{2}} \| e^{\varphi_{k_n}}  v \|_{L^2(\Omega)} 
		+
		|k_n|^{\frac{1}{2}} \| e^{\varphi_{k_n}}  \nabla v \|_{L^2(\Omega)}	
		\leq 
		C \left( 
		\| e^{ \varphi_{k_n}}  f_\eta \|_{L^2(\Omega)}
		+ 
		|k_n|^{\frac{7}{4}+\frac{1}{2d}} \| e^{ \varphi_{k_n}} u\|_{H^1(\omega)}
		\right.
		\\
		+ 
		\left(|k_n|^{\frac{1}{2}}\| W_{2,d} \|_{L^d(\Omega)} + |k_n|^{\frac{3}{4}-\frac{1}{2d}}   \| V_{\frac{d}{2}} \|_{L^{\frac{d}{2}}(\Omega)}  \right)
		|k_n|^{- \frac{3}{4} - \frac{1}{2d}} 
		\left(
		\left( \| V_\infty\|_{L^\infty(\Omega)} +
		|k_n| \| W_{2, \infty}\|_{L^\infty(\Omega)} \right)
		\| e^{ \varphi_{k_n}} v \|_{L^{2} (\Omega)}
		\right.
		\\
		\left.
		+ 
		\| W_{1, \infty} \|_{L^\infty(\Omega)} \| e^{\varphi_{k_n}} \nabla v\|_{L^2(\Omega) }
		+ 
		\| e^{\varphi_{k_n}}  f_\eta \|_{L^2(\Omega)}
		+ 
		|k_n|^{\frac{3}{2}} \| e^{\varphi_{k_n}} u \|_{H^1(\omega)}
		\right).
		\end{multline}
		With the constraints \eqref{Constraints-1}, \eqref{Constraints-2} and \eqref{Smallness-W-1-2}, there exists $C$ such that 
		\begin{equation*}
		\left(|k_n|^{\frac{1}{2}}\| W_{2,d} \|_{L^d(\Omega)} + |k_n|^{\frac{3}{4}-\frac{1}{2d}}   \| V_{\frac{d}{2}} \|_{L^{\frac{d}{2}}(\Omega)}  \right)
		|k_n|^{- \frac{3}{4} - \frac{1}{2d}} 
		\leq
		C |k_n|^{-\frac{1}{d}}, 
		\end{equation*}
		and 
		\begin{equation*}	
		\left( \| V_\infty\|_{L^\infty(\Omega)} +
		|k_n| \| W_{2, \infty}\|_{L^\infty(\Omega)} \right)
		\leq 2 c_0 |k_n|^{\frac{3}{2}}, 
		\quad \text{ and } \quad 
		\| W_{1, \infty} \|_{L^\infty(\Omega)}
		\leq c_0 |k_n|^{\frac{1}{2}}.
		\end{equation*}
		Accordingly, we deduce from \eqref{Carleman-OurPb+Wolff} that there exists $\tau_2 \geq \tau_1$ such that for all $n \geq \tau_2$, 
		\begin{equation}
		\label{Final-Est-Carl+Wolff}	
		|k_n|^{\frac{3}{2}} \| e^{\varphi_{k_n}}  v \|_{L^2(\Omega)} 
		+
		|k_n|^{\frac{1}{2}} \| e^{\varphi_{k_n}}  \nabla v \|_{L^2(\Omega)}	
		\leq 
		C \left( 
		\| e^{ \varphi_{k_n}}  f_\eta \|_{L^2(\Omega)}
		+ 
		|k_n|^{\frac{7}{4}+\frac{1}{2d}} \| e^{ \varphi_{k_n}} u\|_{H^1(\omega)}
		\right).
		\medskip
		\end{equation}

		\textit{Step 4: Quantification.} 
		To quantify the unique continuation property, we simply need to choose appropriate values for $n$ (recall that $k_n$ is of the order of $n$) and suitable decompositions of $V$, $W_1$ and $W_2$ as $V_{\frac{d}{2}} + V_d + V_\infty$, $W_{1,d} + W_{1 \infty}$, $W_{2,d} + W_{2 \infty}$.
		
		We thus recall the constraints needed so far (see \eqref{Constraints-1}, \eqref{Constraints-2}, \eqref{Smallness-W-1-2}), which we sum up as follows: 
		\begin{align}
		&
		\|V_{\frac{d}{2}} \|_{L^{\frac{d}{2}}(\Omega)} \lll 1, \qquad  |n|^{\frac{3}{4}-\frac{1}{2d}} \|V_d \|_{L^d(\Omega)} \lll n^{\frac{3}{2}}, 	\qquad   \|V_\infty \|_{L^\infty(\Omega)} \lll n^{\frac{3}{2}},\label{Constraints-for-V}
		\\
		&  \| W_{1, d}\|_{L^d(\Omega)} \lll 1, \qquad 
		\| W_{1, \infty}\|_{L^\infty(\Omega)} \lll n^{\frac{1}{2}}, \label{Constraints-for-W-1}
		\\
		&  \| W_{2, d}\|_{L^d(\Omega)} \lll 1, \qquad 
		\| W_{2, \infty}\|_{L^\infty(\Omega)} \lll n^{\frac{1}{2}}.  \label{Constraints-for-W-2}
		\end{align}
		
		\noindent \textit{Satisfying conditions \eqref{Constraints-for-W-1}--\eqref{Constraints-for-W-2}.}
		For $W_1 \in L^{q_1}(\Omega)$ and $W_2 \in L^{q_2}(\Omega)$, with $q_1$ and $q_2$  in $[d,\infty]$, and for positive numbers $\lambda_1,$ $\lambda_2$ yet to be determined, we set $W_{1,d} = W_1 1_{|W_1| > \lambda_1}$, $W_{1, \infty} = W_1 1_{|W_1| \leq \lambda_1}$, $W_{2,d} = W_2 1_{|W_2| > \lambda_2}$, and $W_{2, \infty} = W_2 1_{|W_2| \leq \lambda_2}$. Conditions  \eqref{Constraints-for-W-1}--\eqref{Constraints-for-W-2} then read:
		\begin{align*} 
		& \lambda_{1}^{1-\frac{q_1}{d}} \| W_{1}\|_{L^{q_1}(\Omega)}^{\frac{q_1}{d}} \lll 1, \qquad 
		\lambda_1 n^{-\frac{1}{2}} \lll 1
		\\
		& \lambda_{2}^{1-\frac{q_1}{d}} \| W_{2}\|_{L^{q_2}(\Omega)}^{\frac{q_1}{d}} \lll 1, \qquad 
		\lambda_2 n^{-\frac{1}{2}}\lll 1.
		\end{align*}
		If we choose $\lambda_1$ and $\lambda_2$ such that $ \lambda_{j}^{1-\frac{q_j}{d}} \| W_{j}\|_{L^{q_j}(\Omega)}^{\frac{q_j}{d}} =  \lambda_j n^{-\frac{1}{2}}$ ($j \in \{1, 2\}$), that is 
		\begin{align*}
		\lambda_1 = \|W_1\|_{L^{q_1}(\Omega)} n^{\frac{d}{2q_1}} \ \text{  and   } \ \lambda_2 = \|W_2\|_{L^{q_2}(\Omega)} n^{\frac{d}{2q_2}},
		\end{align*}
		then this yields  the conditions
		\begin{align}\label{quantified-condition-W_1-W_2-q1-q_2-[d-infty[}
		n^{\frac{1}{2} - \frac{d}{2q_1}} \ggg \|W_1\|_{L^{q_1}(\Omega)} ,  \quad \text{ and } \quad 
		n^{\frac{1}{2} - \frac{d}{2q_2}} \ggg \|W_2\|_{L^{q_2}(\Omega)}.    
		\end{align} 
		
		\noindent \textit{Satisfying conditions \eqref{Constraints-for-V}.}
		Now, we consider the following two cases for the potential $V$: 
		
		\textit{Case $V \in L^{q_0}(\Omega)$ with $q_0 \in [d,\infty]$.} For $\lambda_0>0$ to be chosen later, we set $V_{\frac{d}{2}} = 0$, and $V_d = V 1_{|V|>\lambda_0}$,  $V_\infty  =  V 1_{|V| \leq \lambda_0}$, so that the conditions \eqref{Constraints-for-V} read
		\begin{align*} 
		\lambda_0^{1 -\frac{q_0}{d}}  \|V\|_{L^{q_0}(\Omega)}^{\frac{q_0}{d}} \lll n^{\frac{3}{4}+\frac{1}{2d}}, 	\qquad 
		\lambda_0  \lll n^{\frac{3}{2}}. 
		\end{align*}
		With the choice $\lambda_0= \|V\|_{L^{q_0}(\Omega)}^{\frac{q_0}{d}}n^{(\frac{3}{4} - \frac{1}{2d})\frac{d}{q_0}} $,  this gives 
		\begin{align}\label{quantified-condition-V-q0-[d-infty[}
		n^{(2 - \frac{d}{q_0}) (\frac{3}{4}-\frac{1}{2d})}
		\ggg
		\|V \|_{L^{q_0}(\Omega)} .
		\end{align}
		%
		
		\textit{Case $V\in L^{q_0}(\Omega)$ with $q_0 \in (d/2, d]$.}
		For $\lambda_0 >0$ to be determined later, we set $V_\infty=0$, and $ V_{\frac{d}{2}} = V 1_{|V| > \lambda_0}, \  V_{d} = V 1_{|V| \leq \lambda_0},$ so that  the conditions \eqref{Constraints-for-V} read
		\begin{align*}
		&
		\lambda_0^{1- \frac{2q_0}{d} } \|V\|_{L^{q_0}}^{\frac{2q_0}{d}} \lll 1, \qquad 
		\lambda_0^{1 -\frac{q_0}{d}}  \|V \|_{L^{q_0}}^{\frac{q_0}{d}} \lll n^{\frac{3}{4}+\frac{1}{2d}}.  
		\end{align*}
		With the choice $\lambda_0 = \| V\|_{L^{q_0}} n^{(\frac{3}{4} + \frac{1}{2d})\frac{d}{q_0}}$, this gives
		\begin{align}\label{quantified-condition-V-q0-[d/2-d]}
		n^{(2 - \frac{d}{q_0})(\frac{3}{4}+\frac{1}{2d})} \ggg \|V\|_{L^{q_0}(\Omega)}.   
		\end{align} 
		
		In the following, we assume that the conditions \eqref{quantified-condition-W_1-W_2-q1-q_2-[d-infty[}--\eqref{quantified-condition-V-q0-[d/2-d]} are satisfied, that is, with the notations \eqref{Conditions-gamma(q)-delta(q)},
		$$
		n \geq \tau_3(V, W_1, W_2) : = C\left(1+ \| V\|_{L^{q_0}(\Omega)}^{\gamma(q_0)} + \| W_1\|_{L^{q_1}(\Omega)}^{\delta(q_1)} + \| W_2\|_{L^{q_2}(\Omega)}^{\delta(q_2)}\right),
		$$
		for some sufficiently large $C$, 
		so that in particular the estimate \eqref{Final-Est-Carl+Wolff} holds for all $n \geq \tau_3(V, W_1, W_2)$.
		\medskip
		
		{\it Step 5. Getting a stability estimate.} We start by estimating the term $\| e^{ \varphi_{k_n}}  f_\eta \|_{L^2(\Omega)}$ as follows (recall \eqref{def_f-eta}--\eqref{Support-f-eta}): for $n \geq \tau_3(V, W_1, W_2)$,
		\begin{align*}
		\| e^{ \varphi_{k_n}}  f_\eta \|_{L^2(\Omega)}
		&\leq
		C \left(1 + \|W_1\|_{L^{q_1}(\Omega)} + \|W_2\|_{L^{q_2}(\Omega)} \right)
		\| e^{ \varphi_{k_n}}  u \|_{H^1(\omega)}
		\\
		&\quad+
		C \left(1 + \|W_1\|_{L^{q_1}(\Omega)} + \|W_2\|_{L^{q_2}(\Omega)} \right)
		e^{\sup_{x_1 \in \left(-\frac{7R}{24}, -\frac{R}{4}\right)}\{ \varphi_{k_n}\} }  \| u \|_{H^1(\Omega)}
		\\
		&\leq C |n|^{\frac{1}{2}}  \| e^{ \varphi_{k_n}}  u \|_{H^1(\omega)}
		+ 
		C |n|^{\frac{1}{2}} e^{\sup_{x_1 \in \left(-\frac{7R}{24}, -\frac{R}{4}\right)}\{ \varphi_{k_n}\} }  \| u \|_{H^1(\Omega)},
		\end{align*}	
		where we used the localization properties of the gradient of the cut-off function $\eta$ and the bound  \eqref{quantified-condition-W_1-W_2-q1-q_2-[d-infty[}.
		
		Bounding the weight function $e^{\varphi_{k_n}}$ from above and from below in \eqref{Final-Est-Carl+Wolff}, we get for all $n \geq \tau_3(V, W_1, W_2)$ such that 
		\begin{equation*}
		\label{Final-Est-Carl+Wolff-2}	
		e^{\inf_{O} \{\varphi_{k_n}\} } \|   v \|_{H^1(\mathcal{O})} 
		\leq 
		C
		|n|^{\frac{5}{4}+\frac{1}{2d}}
		e^{ \sup_{\omega} \{\varphi_{k_n}\} }\| u\|_{H^1(\omega)}
		+
		C e^{\sup_{x_1 \in \left(-\frac{7R}{24}, -\frac{R}{4}\right)}\{ \varphi_{k_n}\} }  \| u \|_{H^1(\Omega)}.
		\end{equation*}
		Using then properties \eqref{Weights-Comparison}, we deduce that there exist two positive constants $A$ and $B$ such that for all $n \geq \tau_3(V, W_1, W_2)$, 
		$$
		\|   u \|_{H^1(\mathcal{O})} 
		\leq 
		C 
		e^{ A n }\| u\|_{H^1(\omega)}
		+
		C e^{- Bn}  \| u \|_{H^1(\Omega)}.
		$$
		Optimizing the right hand side with respect to $n \geq \tau_3(V, W_1, W_2)$, we obtain 
		$$
		\|   u \|_{H^1(\mathcal{O})} 
		\leq
		C \| u\|_{H^1(\omega)}^{\frac{B}{A+B} } \|u\|_{H^1(\Omega)}^{\frac{A}{A+B}} \exp(C \tau_3(V, W_1,W_2)).
		$$
		This concludes the proof of Lemma \ref{lemma-a-specific-geomertric-setting}. 
	\end{proof}
	
	\begin{remark}
		\label{Rem-Sum-of-Pot}
		It is clear from the above proof that, if $V$ is the finite sum of potentials $V_i \in L^{p_i}(\Omega)$, the estimate \eqref{Quantitative-0} still holds by replacing 
		$
		\| V \|_{L^p(\Omega)}^{\gamma(p)}
		$
		by 
		$
		\sum_i \norm{V_i}^{\gamma(p_i)}_{L^{p_i}(\Omega)}.
		$
	\end{remark}

	%
	%
	\section{Other geometries and proof of Theorem \ref{Thm-QUCP}}\label{Other geometries and proof of Theorem 1.1}
	This section is devoted to the proof of Theorem \ref{Thm-QUCP}. We will do that using several geometrical settings, up to a quantitative three balls estimate.
	
	%
	%
	
	\subsection{An annulus observed from a neighbourhood of its external boundary}
	
	\begin{lemma} \label{Lemma A ball observed from a neighbourhood of its boundary} Let $R>0$  and $d \geqslant 3$. We consider the following geometric setting (see Figure \ref{geometry of a ball observed from a neighbourhood of its boundary}):
		\begin{align*}  
		\Omega = \A_0\left(\frac{R}{4},2R\right), \quad    \O = \A_0\left(\frac{R}{2},2R\right), \quad   \omega =  \A_0\left(R,2R\right). 
		\end{align*}
		(Here, $\A_0(r_1, r_2)$ denotes the annulus $B_0(r_2) \setminus B_0(r_1)$.)
		
		There exist constants $C=C(R,d)>0$ and $\alpha \in (0,1)$ depending only on $R$ and $d$ so that any solution $u\in H^{1}(\Omega)$ of \eqref{Elliptic-UC} with $(V,W_1,W_2)$ as in \eqref{potential-assumptions} satisfies the quantitative unique continuation estimate \eqref{Quantitative-0} with $\gamma$  and $\delta$ as in \eqref{Conditions-gamma(q)-delta(q)}. 
	\end{lemma}

	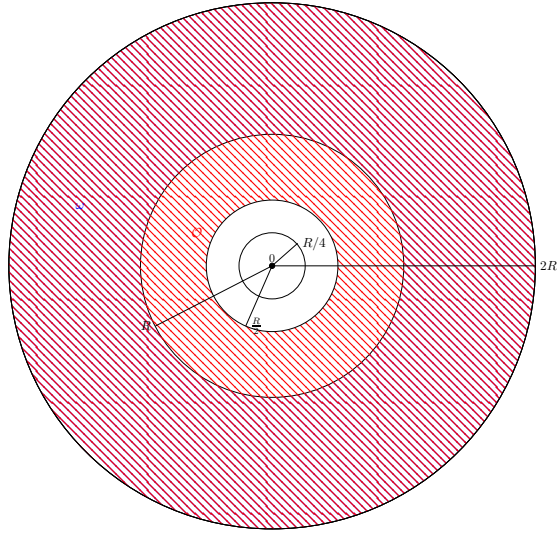
\begin{figure}[!ht]
		\centering
		\scalebox{0.5}{\begin{tikzpicture}
			\draw [pattern=north west lines, opacity=.5, pattern color=blue]  (0,0) circle (7cm) ;
			\draw (0,0) circle (7cm);
			\draw (0,0) circle (0.88cm);
			\draw [fill=white] (0,0) circle (3.5cm) ;
			\draw [pattern=north west lines, opacity=.5, pattern color=red]  (0,0) circle (7cm) ;
			\draw [fill=white] (0,0) circle (1.75cm) ;
			\node at ( -2,0.9) {${\color{red} \O}$};	
			\node at ( -5.1,1.6) {${\color{blue} \omega}$};	
			\filldraw[black] (0,0) circle (2pt) node[anchor=south,scale=0.9]{$0$};
			\draw (0,0) circle (0.88cm);
			\draw(0,0)--(-3.1,-1.6)  node[left]{$R$};
			\draw (0,0)--(7,0)node[right]{$2 R$};
			\draw (0,0)--(0.68,0.6)node[right]{$ R/4$};
			\draw(0,0)--(-0.7,-1.6)  node[above,right] {$\frac{R}{2}$};
			\end{tikzpicture}}
		\caption{The geometric setting of Lemma \ref{Lemma A ball observed from a neighbourhood of its boundary}: Propagation of smallness from a neighborhood of a sphere to its interior.} 
		\label{geometry of a ball observed from a neighbourhood of its boundary}
	\end{figure}
	\begin{proof} 
		The proof of Lemma \ref{Lemma A ball observed from a neighbourhood of its boundary} directly follows from Lemma  \ref{lemma-a-specific-geomertric-setting}. Let $x_0 \in \S^{d-1}$, and apply Lemma \ref{lemma-a-specific-geomertric-setting} with 
		$$
		\Omega_{x_0} = B_0(2R) \cap \left\{ x \cdot x_0 < -\frac{R}{4}\right\},  
		\quad 
		\O_{x_0} = B_0\left(\frac{3R}{2}\right) \setminus \left\{ x \cdot x_0 < -\frac{R}{3}\right\},
		\quad 
		\omega_{x_0} = (B_0(2R) \setminus B_0(R)) \cap
		\Omega_{x_0}.
		$$
		Accordingly, for all $x_0 \in \S^{d-1}$, there exists a constant $C_{x_0}>0$ such that 
		\begin{align*}
		\norm{u }_{H^1(\O_{x_0}))} 
		& \leqslant 
	C_{x_0} 
		e^{
			C_{x_0} \left(\norm{V}_{L^{q_0}(\Omega)}^{\gamma(q_0)}+\norm{W_1}_{L^{q_1}(\Omega)}^{\delta(q_1)}
			+ \norm{W_2}_{L^{q_2}(\Omega)}^{\delta(q_2)}
			\right)
		}\norm{u}_{H^1(\omega_{x_0})}^\alpha
		\norm{u}_{H^1(\Omega_{x_0})}^{1-\alpha}
		\\
		&\leq C_{x_0} 
		e^{
			C_{x_0}\left(\norm{V}_{L^{q_0}(\Omega)}^{\gamma(q_0)}+\norm{W_1}_{L^{q_1}(\Omega)}^{\delta(q_1)}
			+ \norm{W_2}_{L^{q_2}(\Omega)}^{\delta(q_2)}
			\right)
		}\norm{u}_{H^1(\omega)}^\alpha
		\norm{u}_{H^1(\Omega)}^{1-\alpha}
		.   
		\end{align*}
		The constant $C_{x_0}$ is in fact independent of $x_0$ due to the invariance by rotation of the problem. Accordingly, we simply denote it by $C$ in the following. Consequently, the right-hand side of the previous estimate does not depend on $x_0$. Accordingly, taking the square and integrating this inequality with respect to $x_0$ over the sphere $\S^{d-1}$, we obtain an estimate on 
		\begin{align*}
		\int_{x\in  B_0\left(\frac{3R}{2}\right) \setminus B_0\left(\frac{R}{3}\right) } (|u|^2 +|\nabla u|^2 ) \rho_{R}(x) d x,
		\end{align*}
		where $\rho_{R}(x)=\int_{x_0 \in \S^{d-1}} 1_{x\cdot x_0 < -\frac{R}{3}}(x_0) d x_0$. It is then easy to check that $\rho_{R}$ is a radial function, vanishing for $|x| \in (0,R/3)$, and increasing. Consequently, we derive 
		$$
		\| u \|_{H^1\left(B_0\left(\frac{3R}{2}\right)\setminus B_0\left(\frac{R}{2}\right)\right)} 
		\leq 
		\frac{C}{\sqrt{\rho_R(\frac{R}{2})}} 
		e^{
			C \left(\norm{V}_{L^{q_0}(\Omega)}^{\gamma(q_0)}+\norm{W_1}_{L^{q_1}(\Omega)}^{\delta(q_1)}
			+ \norm{W_2}_{L^{q_2}(\Omega)}^{\delta(q_2)}
			\right)
		}\norm{u}_{H^1(\omega)}^\alpha
		\norm{u}_{H^1(\Omega)}^{1-\alpha}
		.   
		$$
		The estimate on $u$ in $H^1(B_0 (2R) \setminus B_0({3R/2}))$ is straightforward since $B_0 (2R) \setminus B_0(3R/2) \subset \omega \subset \Omega$.
	\end{proof}
	%
	%
	
	\subsection{A three balls estimate}
	
	In this part, we prove a quantitative three balls inequality:  
	\begin{lemma}[Three balls estimate] \label{lemma Three spheres estimate} Let $R>0$ and $d \geqslant 3$. 
		We consider the following geometric setting (see Figure \ref{geometry of Three spheres}): 
		\begin{align*}
		\Omega = B_0(4 R), \, \mathcal{O} = B_0(2 R), \text{ and } \omega = B_0(R); 
		\end{align*}
		
		Then there exist constants $C=C(R,d)>0$ and $\alpha \in (0,1)$ depending only on $R$ and $d$ so that any solution $u\in H^{1}(\Omega)$ of \eqref{Elliptic-UC} with $(V,W_1,W_2)$ as in \eqref{potential-assumptions} satisfies the quantitative unique continuation estimate \eqref{Quantitative-0} with $\gamma$ and $\delta$ as in \eqref{Conditions-gamma(q)-delta(q)}.  
	\end{lemma}

	\begin{figure}[ht]
		\centering
		\scalebox{0.6}{\begin{tikzpicture}
			\draw  (0,0) circle (5cm);
			\draw [pattern=north west lines, opacity=.5, pattern color=red]  (0,0) circle (3cm) ;
			\draw [pattern=north west lines, opacity=.5, pattern color=blue] (0,0) circle (1.5cm) ;
			\node at ( -2.3,2.6) {${\color{red} \O}$};	
			\filldraw[black] (0,0) circle (2pt) node[anchor=south,scale=0.9]{$0$};
			
			\draw(0,0)--(-2.6,-1.5)  node[left]{$2 R_0$};
			\draw (0,0)--(5,0)node[right]{$4 R_0$};
			\draw(0,0)--(-1,-1.2)  node[left] {$ R_0$};
			\end{tikzpicture}}
		\caption{The geometric setting of Lemma \ref{lemma Three spheres estimate}: Propagation of smallness from a ball to its exterior.} 
		\label{geometry of Three spheres}
	\end{figure}
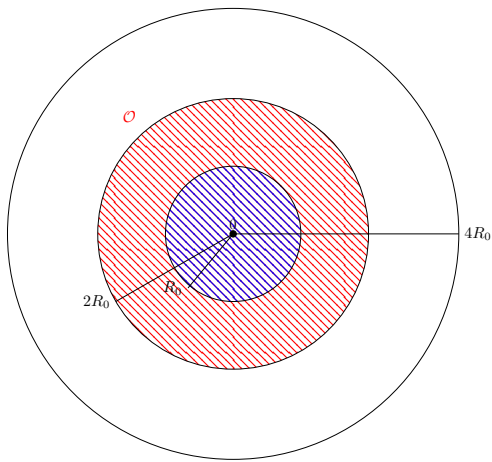
	
	\begin{proof}\textit{Step 1: Conformal reflection.} 
		First, we consider the following geometric setting 
		\begin{align*}
		\Omega_0= \A_0\left(\frac{R}{2},4R\right), \quad \O_0 = \A_0\left(\frac{R}{2},2R\right), \quad \omega_0= \A_0\left(\frac{R}{2},R\right). 
		\end{align*}
		We denote by $T$ the conformal reflection with respect to the sphere $S_0(R)$, given by:
		\begin{align} \label{Def-reflection-T}
		\R^d\setminus \{0\} \ni x \mapsto \tilde x=  T(x):=\frac{R^{2}}{|x|^{2}}x, 
		\end{align}
		
		The images of the sets $\Omega_0$, $\O_0$ and $ \omega_0$ are then given by:
		\begin{align} \label{def of tilde Omega tilde O tilde omega}
		\tilde \Omega = T \Omega_0 = \A_0\left(\frac{R}{4},2R\right), \quad \tilde \O= T \O_0 = \A_0\left(\frac{R}{2},2R\right), \quad  \tilde \omega = T \omega_0=  \A_0\left(R,2R\right). 
		\end{align}
		Therefore, for $u\in H^{1}(\Omega)$ a solution of \eqref{Elliptic-UC} with $(V,W_1,W_2)$ as in \eqref{potential-assumptions}, we consider the Kelvin transform of $u$ (see, for example, \cite{wermer1981potential}), 
		\begin{equation}
		\label{Def-u-R}
		\tilde u _{R} (x) = \left(\frac{R}{|x|}\right)^{(d-2)} u \left( \frac{R^{2}}{|x|^{2}}  x\right), \qquad x\in  \tilde \Omega.
		\end{equation}  
		By a classical computation, using the chain rule, we can verify that 
		for all $i \in \{1,\cdots,d\}$,
		\begin{align*}
		\partial_{x_i} \tilde u_R (x)&=-\frac{(d-2)x_i }{|x|^2}   \tilde u_R(x)+\left(\frac{R}{|x|}\right)^d \sum_{j=1}^d\left(\delta_{i j}- 2 \frac{x_i x_j}{|x|^2}   \right) \partial_{x_j} u\left(\frac{R^{2}}{|x|^{2}} x\right),  \qquad x\in \tilde \Omega, 
		\end{align*}
		and 
		\begin{equation*}
		\Delta \tilde u _R(x) =  \left(\frac{R}{|x|}\right)^{(d+2)} \Delta u\left(\frac{R^{2}}{|x|^{2}}x\right),  \qquad x\in  \tilde \Omega. 
		\end{equation*}
		We then consider the following potentials
		\begin{align*}
		\tilde W_{R,j}(x) &:=  \frac{R^{2}}{|x|^{2}} \left( W_j\left(\frac{R^{2}}{|x|^{2}}x\right)-2 x \cdot W_j\left(\frac{R^{2}}{|x|^{2}}x\right) \frac{x}{|x|^2} \right), \quad j \in \{1,2\}, 
		\\
		\tilde V_R(x)& :=  \underbrace{\frac{R^{4}}{|x|^{4}} \left(  V\left(\frac{R^{2}}{|x|^{2}}x\right) \right) }_{=:\tilde V_{R,1}}+ \frac{(d-2)}{|x|^2} x \cdot \left(\tilde W_{R,1}(x)+ \tilde W_{R,2}(x) \right). 
		\end{align*}
		Consequently, $u\in H^{1}(\Omega)$ is a solution of \eqref{Elliptic-UC} with $(V,W_1,W_2)$ if and only if $u_R$ given by \eqref{Def-u-R} solves
		\begin{align*}
		\Delta \tilde u_R = \tilde V_R \tilde u_R + \tilde W_{R,1} \cdot \nabla \tilde u_R + \div\left(\tilde W_{R,2} \tilde u_R\right) \quad \text{ in } \tilde \Omega.  
		\end{align*}

		\textit{Step 2: Application of Lemma \ref{Lemma A ball observed from a neighbourhood of its boundary}.}  Applying Lemma \ref{Lemma A ball observed from a neighbourhood of its boundary} to $\tilde u_R$ with the geometric setting defined in \eqref{def of tilde Omega tilde O tilde omega} (together with Remark \ref{Rem-Sum-of-Pot}), there exists $C>0$ depending only on $d$ and $R$ such that 
		\begin{align*}
		\norm{\tilde  u_R }_{H^1(\tilde \O)} \leqslant 
		C 
		e^{
			C \left(\norm{ \tilde  V_{R,1}}_{L^{q_0}(\tilde \Omega)}^{\gamma(q_0)}+\norm{ \tilde  W_{R,1}}_{L^{q_1}(\tilde \Omega)}^{\gamma(q_1)}+\norm{ \tilde  W_{R,2}}_{L^{q_2}(\tilde \Omega)}^{\gamma(q_2)}+\norm{\tilde  W_{R,1}}_{L^{q_1}(\tilde \Omega)}^{\delta(q_1)}
			+ \norm{\tilde W_{R,2}}_{L^{q_2}( \tilde \Omega)}^{\delta(q_2)}
			\right)
		}\norm{\tilde u_R }_{H^1( \tilde \omega)}^\alpha
		\norm{\tilde u_R}_{H^1(\tilde \Omega)}^{1-\alpha}, 
		\end{align*}  
		with $\gamma$ and $\delta$ as defined in \eqref{Conditions-gamma(q)-delta(q)}. Then, using the change of variables $y = (R/|x|)^2x$, one can verify that  
		\begin{align*}
		\norm{ \tilde  W_{R,j}}_{L^{q_j}(\tilde \Omega)} &\simeq \norm{    W_{j}}_{L^{q_j}(\Omega_0)}, \quad j \in \{1,2\},  
		\\ 
		\norm{ \tilde V_{R,1}}_{L^{q_0}(\tilde \Omega)} &\simeq \norm{ V }_{L^{q_0}( \Omega_0)},   
		\\ 
		\norm{\tilde  u_R }_{H^1(\tilde \Pi)} & \simeq \norm{  u }_{H^1( T^{-1} \Pi )}, \quad \Pi \in \{\tilde \Omega , \tilde \O , \tilde \omega \},  
		\end{align*}
		where we have used that the Jacobian 
		determinant of the map $T^{-1}$ ($=T$) is bounded. 
		
		Since $\gamma (q ) \leqslant \delta (q )$ for all $q \in  (d, \infty]$, there exists a positive constant $C$ depending only on $d$ and $R$, such that 
		\begin{align*}
		\norm{   u  }_{H^1(  \O_0)} & \leqslant 
		C 
		e^{
			C \left(\norm{   V }_{L^{q_0}(  \Omega_0)}^{\gamma (q_0)}+ \norm{    W_{1}}_{L^{q_1}(  \Omega_0)}^{\delta(q_1)}
			+ \norm{ W_{2}}_{L^{q_2}(  \Omega_0)}^{\delta(q_2)}
			\right)
		}\norm{ u }_{H^1(  \omega_0)}^\alpha
		\norm{ u }_{H^1(  \Omega_0)}^{1-\alpha}
		\\ 
		& 
		\leqslant 
		C 
		e^{
			C \left(\norm{ V }_{L^{q_0}( \Omega)}^{\gamma(q_0)}+ \norm{    W_{1}}_{L^{q_1}(  \Omega)}^{\delta(q_1)}
			+ \norm{ W_{2}}_{L^{q_2}(  \Omega)}^{\delta(q_2)}
			\right)
		}\norm{ u }_{H^1(  \omega)}^\alpha
		\norm{ u }_{H^1(  \Omega)}^{1-\alpha}.
		\end{align*}
		To conclude Lemma \ref{lemma Three spheres estimate}, one should also get a similar estimate for $ \norm{ u }_{H^1( B_0({R/4}))}$; this latter estimate is straightforward as $B_0(R/4) \subset \omega \subset \Omega$. The proof of Lemma \ref{lemma Three spheres estimate} is thus completed.
	\end{proof} 
	%
	%
	\subsection{The general case: Proof of Theorem \ref{Thm-QUCP}}

	\begin{proof}[Proof of Theorem \ref{Thm-QUCP}] The strategy follows the same lines as the one of  \cite[Theorem 5.6]{leelliptic}, see also \cite[Theorem 1.2]{ervedoza2023cost}, and is based on the classical ideas that three balls estimates allow to propagate the information.   
		
		\textit{Step 1: Propagation of smallness in neighborhoods of points $y$ in 
			$\overline\O$.} Recall the geometric condition {\bf(GC)}: For all $y \in \overline\O$, there exist $x_0 \in \omega$, $r_{y}>0$ and a smooth path $\gamma_{y}$ of finite length such that $\gamma_{ y}(0) = x_0$, $\gamma_{y}(1) = y$, and $\cup_{s \in [0,1]} B_{\gamma_{y}(s)}(r_{y}) \subset \Omega$.
		
		Accordingly, for $y \in \overline\O$, we take such path $\gamma_{y}$, and define $R_y=\min\{ r_{y}/ 4, r_0\}$, where $r_0$ is such that $B_{x_0} (r_0) \subset \omega$.
		
		We define a sequence $(x_{(j)})_j$, for $j \geq 0$, by $x_{(j)}=\gamma_{y}\left(t_j\right)$ where $t_0=0$ and, for $j \geq 1$, 
		\begin{align*}
		t_j=\left\{\begin{array}{ll}
		\inf A_j & \text { if } A_j \neq \emptyset, \\
		1 & \text { if } A_j=\emptyset,
		\end{array} \quad \text { where }  A_j=\left\{\sigma \in\left(t_{j-1}, 1\right] ; \gamma_{y}(\sigma) \notin B_{x_{(j-1)}}( R_y)\right\} .\right.    
		\end{align*}
		The sequence $(x_{(j)})_j$ is finite since the length of $ \gamma_{y}$ is finite. Let $\left(x_{(0)}, \cdots, x_{(N_y)}\right)$ be such a sequence with $x_{(N_y)}=y$. Note that we have $B_{x_{(j+1)}}(R_y) \subset B_{x_{(j)}}(2 R_y) \subset B_{x_{(j)}}(4 R_y)\subset \Omega$ for $j=0, \cdots, N_y-1$, because of the choice we made for $R_y$ above. By Lemma \ref{lemma Three spheres estimate} there exist $C>0$ and $\alpha \in(0,1)$ such that
		$$
		\|u\|_{H^1\left(B_{x_{(j+1)}}( R_y)\right)} \leq\|u\|_{H^1\left(B_{x_{(j)}} ( 2 R_y )\right)} 
		\leq C e^{
			C \left(\norm{   V  }_{L^{q_0}(  \Omega)}^{\gamma(q_0)}+\norm{    W_{1}}_{L^{q_1}(  \Omega)}^{\delta(q_1)}
			+ \norm{ W_{2}}_{L^{q_2}(  \Omega)}^{\delta(q_2)}
			\right)
		}\|u\|_{H^1(\Omega)}^{1-\alpha}  \|u\|_{H^1(B_{x_{(j)}}(R_y))}^\alpha,
		$$
		for $j=0, \ldots, N-1$. Iterating this estimate we obtain  
		\begin{align*}
		\|u\|_{H^1(B_y( R_y))} 
		&   \leqslant C^{\sum_{j = 0}^{N_y} \alpha^j}  e^{
			C \left(\sum_{j = 0}^{N_y} \alpha^j\right) \left(\norm{   V }_{L^{q_0}(  \Omega)}^{\gamma(q_0)}+\norm{   V_2 }_{L^{p_2}(  \Omega)}^{\gamma_2(p_2)}+\norm{    W_{1}}_{L^{q_1}(  \Omega)}^{\delta(q_1)}
			+ \norm{ W_{2}}_{L^{q_2}(  \Omega)}^{\delta(q_2)}
			\right)
		} \|u\|_{H^1(\Omega)}^{1-\alpha^{N_y}}\|u\|_{H^1\left(B_{x_{(0)}}( R_y)\right)} ^{\alpha^{N_y}}
		\\
		&  \leqslant C_y  e^{
			C_y \left(\norm{   V }_{L^{q_0}(  \Omega)}^{\gamma(q_0)}+\norm{   V_2 }_{L^{p_2}(  \Omega)}^{\gamma_2(p_2)}+\norm{    W_{1}}_{L^{q_1}(  \Omega)}^{\delta(q_1)}
			+ \norm{ W_{2}}_{L^{q_2}(  \Omega)}^{\delta(q_2)}
			\right)
		} \|u\|_{H^1(\Omega)}^{1-\alpha_y}\|u\|_{H^1(\omega)} ^{\alpha_y},	
		\end{align*}
		for some $C_y>0$ and $\alpha_y\in (0,1)$.
		
		\textit{Step 2: Compactness argument.} Because of the compactness of $\bar \O$, we can choose a finite number of balls $(B_{y_j}( R_{y_j}/2))_{j \in \{1, \cdots, p\}}$ with $y_j \in \overline\O$ and $R_{y_j}$ as above such that $\bar\O \subset \cup_{j \in \{1, \cdots, p\}} B_{y_j}( R_{y_j}/2)$. We then construct a partition of unity of $\bar  \O$ by choosing smooth functions $(\chi_j)_{0 \leqslant j \leqslant N}$, each one being supported in $B_{y_j}( R_{y_j})$, such that 
		\begin{align*}
		\sum_{j=1}^p \chi _j = 1 \, \text{ in a neighborhood of } \, \bar \O, \quad 0 \leqslant \chi_j \leqslant 1. 
		\end{align*}
		Therefore, 
		\begin{align*}
		\norm{u}_{H^1(\O)} 
		& \leqslant C \sum_{j=1}^p \norm{u}_{H^1( \operatorname{supp} \chi_j )} 
		\leq C \sum_{j=1}^p \norm{u}_{H^1\left(B_{y_j}\left(R_{y_j}\right)\right)} 
		\\
		&\leq 
		C  e^{
			C \left(\norm{   V }_{L^{q_0}(  \Omega)}^{\gamma(q_0)}+\norm{   V_2 }_{L^{p_2}(  \Omega)}^{\gamma_2(p_2)}+\norm{    W_{1}}_{L^{q_1}(  \Omega)}^{\delta(q_1)}
			+ \norm{ W_{2}}_{L^{q_2}(  \Omega)}^{\delta(q_2)}
			\right)
		} \|u\|_{H^1(\Omega)}^{1-\alpha}\|u\|_{H^1\left(\omega\right)} ^{\alpha}, 
		\end{align*}
		with $\alpha = \min_{j \in \{1, \cdots, p\}} \{ \alpha_{y_j}\}$. This concludes the proof of Theorem \ref{Thm-QUCP}.
	\end{proof}
	
	%
	%
	\appendix 
	\section{$L^p$-$L^q$ estimates for Fourier multipliers}\label{appendix Fourier multiplier operators} 
	
	In this section, we present the machinery used to get estimates on the operators $(K_{\tau,j})_{j \in \{0, \cdots, d\}}$, and that was developed in \cite[Section 5]{Dehman-Ervedoza-Thabouti-2023}. 
	
	Let $n \geqslant 2$. We consider $X_0 <X_1$ and coefficients $(\lambda_j)_{j \in \{1, \cdots, n\}}$ satisfying %
	\begin{equation}
	\label{Coercivity-Restriction}
	\exists c_0 >0, \quad \forall a \in [X_0, X_1], \, \forall \xi \in \R^{n}, \quad 
	\frac{1}{c_0} | \xi|^2 
	\leq 
	\sum_{j =1}^n (1 - a \lambda_j) |\xi_j|^2 
	\leq 
	c_0 |\xi|^2.  
	\end{equation}
	We also introduce the function $\psi $ defined by 
	\begin{equation}
	\label{Def-Psi-R-n}
	\psi(a, \xi)=\sqrt{ \sum_{j=1}^{n}(1-a  \lambda_{j}) \xi_{j} ^2},
	\qquad \qquad a \in [X_0,X_1], \  \xi \in \R^{n}.
	\end{equation}
	and $\Sigma_a$ the ellipsoid defined for $a \in [X_0,X_1]$ by 
	\begin{equation}
	\label{Def-Sigma-a}
	\Sigma_a = \{\xi \in \R^n, \psi(a, \xi) = 1 \}.
	\end{equation}
	
	For $a \in [X_0, X_1]$ and $k \in L^\infty(\R_+, L^\infty(\Sigma_a))$, we consider operators given as follows:
	\begin{equation}
	\label{Class-Fourier-Multipliers}
	K_{a,k} : L^2(\R^n) \to L^2(\R^n), 
	\text{ given by } 
	\widehat{K_{a,k}(f)} (\xi ) 
	= 
	k\left(\psi(a,\xi),\frac{\xi}{\psi(a,\xi)}  \right) 
	\hat f(\xi),
	\quad \qquad \xi \in \R^n.
	\end{equation}
	In other words, $K_{a,k}$ is defined as a Fourier multiplier, and we look at the multiplier in some kind of radial coordinates associated to $\Sigma_a$: $\psi(a, \xi)$ is a positive real number corresponding to a radius, and $\xi/ \psi(a, \xi)$ is an element of the ellipsoid $\Sigma_a$. Also note that for $a = 0$, this coincides with the classical radial coordinates for $\R^n$.

	We have the following result: 
	\begin{proposition}[Proposition 5.3. \cite{Dehman-Ervedoza-Thabouti-2023}]
		\label{Prop-Est-Fourier-Op} 
		Let $n\in \N$, $n\geq 2$. Let $X_0 <X_1$, and the coefficients $(\lambda_j)_{j \in \{1, \cdots, n\}}$ satisfy \eqref{Coercivity-Restriction}. For $a \in [X_0,X_1]$, let $\psi$ and $\Sigma_a$ be as in \eqref{Def-Sigma-a}--\eqref{Def-Psi-R-n}.
		Then there exists a constant $C>0$ such that 		%
		\begin{itemize}[topsep=0pt,itemsep=0pt,parsep=0pt]
			
			\item for all $a \in [X_0,X_1]$, for all $k \in L^\infty(\R_+, L^\infty(\Sigma_a))$, the Fourier multiplier operator $K_{a,k}$ in \eqref{Class-Fourier-Multipliers} maps $L^2(\R^{n})$ to itself and 
			\begin{equation}
			\label{Est-K-a-k-L2-L2}
			\norm{K_{a,k}}_{\mathscr{L}(L^2(\R^n))}
			\leq 
			\| k \|_{L^\infty(\R_+, L^\infty(\Sigma_a))}.
			\end{equation}
			\item for all $a \in [X_0,X_1]$, for all $k \in L^\infty(\R_+, L^\infty(\Sigma_a))$ satisfying
			$$
			\int_0^\infty \| k(\lambda, \cdot )\|_{L^\infty(\Sigma_a)} \, \lambda^{\frac{n-1}{n+1}} \, d\lambda < \infty, 
			$$
			the Fourier multiplier operator $K_{a,k}$ in \eqref{Class-Fourier-Multipliers} belongs to $\mathscr{L}(L^{\frac{2(n+1)}{(n+3)}}(\R^n),L^{\frac{2(n+1)}{(n-1)}}(\R^n))$ and 
			\begin{equation}
			\label{Est-K-a-k-Lp-Lp'}
			\norm{K_{a,k}}_{\mathscr{L}(L^{\frac{2(n+1)}{(n+3)}}(\R^n),L^{\frac{2(n+1)}{(n-1)}}(\R^n))} 
			\leq 
			C \int_0^\infty \| k(\lambda, \cdot )\|_{L^\infty(\Sigma_a)} \, \lambda^{\frac{n-1}{n+1}} \, d\lambda.
			\end{equation}
			\item for all $a \in [X_0,X_1]$, for all $k \in L^\infty(\R_+, L^\infty(\Sigma_a))$ satisfying
			$$
			\int_0^\infty \| k(\lambda, \cdot )\|_{L^\infty(\Sigma_a)}^2 \, \lambda^{\frac{n-1}{n+1}} \, d\lambda < \infty, 
			$$
			the Fourier multiplier operator $K_{a,k}$ in \eqref{Class-Fourier-Multipliers} belongs to 
			$$
			\mathscr{L}(L^{\frac{2(n+1)}{(n+3)}}(\R^n),L^{2}(\R^n)) \cap \mathscr{L}(L^{2}(\R^n), L^{\frac{2(n+1)}{(n-1)}}(\R^n)), 
			$$
			and  
			\begin{align}
			\label{Est-K-a-k-Lp-L2}
			\norm{K_{a,k}}_{\mathscr{L}(L^{\frac{2(n+1)}{(n+3)}}(\R^n),L^{2}(\R^n))} 
			\leq 
			C \sqrt{\int_0^\infty \| k(\lambda, \cdot )\|_{L^\infty(\Sigma_a)}^2 \, \lambda^{\frac{n-1}{n+1}} \, d\lambda}, 
			\\ 
			\label{Est-K-a-k-L2-Lp'}
			\norm{K_{a,k}}_{\mathscr{L}(L^{2}(\R^n), L^{\frac{2(n+1)}{(n-1)}}(\R^n))}
			\leq
			C \sqrt{\int_0^\infty \| k(\lambda, \cdot )\|_{L^\infty(\Sigma_a)}^2 \, \lambda^{\frac{n-1}{n+1}} \, d\lambda}.
			\end{align}
		\end{itemize}
	\end{proposition}
	%

	\bibliographystyle{plain}

\end{document}